\documentclass[12pt]{amsart}

\usepackage{amsmath}
\usepackage{amsthm}
\usepackage{amssymb}
\usepackage{mathrsfs}
\usepackage[utf8]{inputenc}
\usepackage{a4wide}
\usepackage{slashed}
\usepackage{enumitem}
\usepackage[all]{xy}
\usepackage{color}
\usepackage{comment}

\newtheorem{Theorem}{Theorem}[section]
\newtheorem{Definition}[Theorem]{Definition}

\newtheorem{Example}[Theorem]{Example}
\newtheorem{Lemma}[Theorem]{Lemma}
\newtheorem{Proposition}[Theorem]{Proposition}

\newtheorem{Conjecture}[Theorem]{Conjecture}
\newtheorem{Corollary}[Theorem]{Corollary}

\newtheorem{Remark}[Theorem]{Remark}

\newcommand{\dccomm}[1]{\begingroup\color{blue}DC:~#1\endgroup}

\newcommand{\del}{\partial}
\newcommand\C{\mathbb{C}}

\newcommand\R{\mathbb{R}}

\newcommand\Q{\mathbb{Q}}
\newcommand\Z{\mathbb{Z}}

\newcommand\Hom{\textup{Hom}}
\newcommand\lra{\longrightarrow}
\newcommand\xra{\xrightarrow}
\newcommand\ul{\underline}
\newcommand\ol{\overline}
\newcommand\wh{\widehat}
\newcommand{\wt}{\widetilde}
\newcommand\an[1]{\langle #1 \rangle}
\newcommand\cal[1]{\mathcal{#1}}
\newcommand\id{\textup{id}}
\newcommand\QHS{\Q-\textup{HS}}
\newcommand{\V}{\mathcal{V}}
\newcommand{\E}{\mathcal{E}}

\DeclareMathOperator{\PD}{PD}
\DeclareMathOperator{\coker}{coker}

\advance \topmargin by -\headheight
\advance \topmargin by -\headsep
     
\evensidemargin \oddsidemargin
\title{Positive Ricci curvature on highly connected manifolds}
\author{Diarmuid Crowley and David J. Wraith}
\date{\today}

\begin{document}
\maketitle

\begin{abstract}
For $k \ge 2,$ let $M^{4k-1}$ be a closed $(2k{-}2)$-connected manifold. If $k \equiv 1$~mod~$4$ assume further that $M$ is $(2k{-}1)$-parallelisable. Then there is a homotopy sphere $\Sigma^{4k-1}$ such that $M \sharp \Sigma$ admits a Ricci positive metric. This follows from a new description of these manifolds as the boundaries of explicit plumbings.
\end{abstract}

\section{Introduction} 
\bigskip
The results of this paper arise from a synthesis of ideas coming from Geometric Topology on the one hand and Riemannian Geometry on the other. We present our main results in Section \ref{subsec:main_results} below. In order to place these results in the correct context, in 
Section~\ref{subsec:geometric_background} we carefully outline the geometric background from which the questions we address in this paper have emerged. 

In this paper all manifolds will be smooth, oriented, connected and compact (unless stated otherwise), and all Riemannian metrics will be complete.

\subsection{Main results} \label{subsec:main_results}

Understanding the implications of curvature having a definite sign is a fundamental problem in Riemannian Geometry. 
%
We can ask about the topological implications for a manifold if it admits a metric satisfying such a curvature assumption, or conversely ask whether it is possible to deduce from a set of topological data whether or not such a metric exists.  
In some situations, including positive Ricci curvature, we suffer from a lack of examples. Thus a basic and difficult task is to look for methods which are capable of generating new metrics on new manifolds. 


The primary aim of this paper is to greatly extend the collection of examples of closed manifolds $M$ known to admit metrics of positive Ricci curvature. We achieve this by topological means, with the geometric input coming from a result of the second author (see Theorem 1.1 below).
%
The manifolds in question are so called `highly connected manifolds'. We shall call a manifold highly connected if it has dimension $2n$ or $(2n{+}1)$ and is $(n-1)$-connected. 
Such manifolds were studied extensively from a topological point-of-view in the 1950s, 1960s and 1970s by Smale, Kervaire and Milnor, Wall and others (see for example \cite{Sm}, \cite{Kervaire&Milnor}, \cite{Wa1}, \cite{Wa4}, \cite{Wi2}). Despite their topological simplicity, highly-connected manifolds constitute a rich class of manifolds. As an illustration of this fact note that in every dimension $n=4k-1 \geq 7$, every finitely generated abelian group can arise as the integral cohomology group $H^{2k}(M)$ of such a manifold.

The group of diffeomorphsim classes of homotopy $n$-spheres, $\Theta_n$, defines a very special class of highly connected manifolds.
Homotopy $n$-spheres are manifolds $\Sigma^n$ which have the homotopy type of the standard sphere $S^n.$ 
For $n \geq 5$ Kervaire and Milnor showed that $\Theta_n$ is finite.
In fact it was the study of highly connected manifolds which led Milnor to the discovery of exotic spheres in the 1950s.  


As a result of the techniques we use, we will be particularly interested in highly connected manifolds in dimensions $4k-1 \geq 7$, and in order to make the topology a little more tractable we further assume that these manifolds are $(2k{-}1)$-parallelisable: something which is automatic unless $k \equiv 1$~mod~$4$.
(A manifold is said to be $j$-parallelisable if the tangent bundle restricted to some $j$-skeleton is trivial.) 
Even though these manifolds have been studied intensively from the topological point of view, to the best of our knowledge much less is known about these manifolds from a geometric perspective.  Our first main theorem addresses this issue.\\

\noindent{\bf Theorem A.} {\it Let $k \ge 2,$ and let $M^{4k-1}$ be a $(2k{-}2)$-connected manifold. If $k \equiv 1$~mod~$4$ assume further that $M$ is $(2k{-}1)$-parallelisable. Then there is a homotopy sphere $\Sigma^{4k-1}$ such that $M \sharp \Sigma$ admits a metric of positive Ricci curvature.}\\

From the refined version of Theorem A, Theorem A$'$ below, we obtain the following special cases:\\

\noindent{\bf Theorem B.} {\it All 2-connected 7-manifolds and all 4-connected 11-manifolds admit Ricci positive metrics.}\\

We can give a more precise reformulation of Theorem A using the language of bordism. 
Let $BO\an{2k} \to BO$ be the $(2k{-}1)$-connected cover of the classifying space $BO$: 
for example $BO\an{2} = BSO$ and $BO\an{4} = BSpin$.  There are associated bordsim 
groups, denoted $\Omega_*^{O\an{2k-1}}$; see Section \ref{subsec:BO<n>-bordism}.  
If $M^{4k-1}$ as above is $(2k{-}1)$ parallelisable then
$M$ admits a lift of its stable normal bundle to a map $M \to BO\an{2k}$,
unique up to equivalence, and hence $M$ defines an bordism class
\[ [M] \in \Omega_{4k-1}^{O\an{2k-1}}. \]
%
%
In particular, every homotopy $(4k{-}1)$-sphere $\Sigma$ defines such a bordism class $[\Sigma]$, and there is an induced homomorphism
\[ \Theta_{4k-1} \to \Omega_{4k-1}^{O\an{2k-1}}, \quad \Sigma \mapsto [\Sigma], \]
which is onto; see Theorem \ref{thm:delTP}. 
We also recall that in dimensions $4k{-}1 \ge 7$, Brumfiel \cite{Brumfiel} defined for each homotopy sphere $\Sigma$, an invariant $f(\Sigma) \in bP_{4k}$, where $bP_{4k}$ is the group of homotopy $(4k{-}1)$-spheres bounding parallelisable manifolds; see Theorem \ref{thm:homotopy_spheres}.
We can now reformulate and refine Theorem A: \\ 

\noindent{\bf Theorem A$'$.} 
{\it Let $k \geq 2$ and let $M^{4k-1}$ be a closed $(2k{-}2)$-connected manifold which is $(2k{-}1)$-parallelisable.
If
$ [M] = 0 \in \Omega_{4k-1}^{O\an{2k-1}}, $
then $M$ admits a metric of positive Ricci curvature. Moreover, the homotopy sphere $\Sigma$ appearing in the statement of Theorem A can be chosen so that its Brumfiel invariant $f(\Sigma)$ vanishes.}\\

So far, two important technical difficulties have arisen in trying to construct positive Ricci curvature metrics on highly connected manifolds. Topologically these objects decompose in a natural way into a connected sum of indecomposable pieces. The first problem arises from the connected sum construction: in general, positive Ricci curvature cannot be extended across such sums. The second problem is the more subtle issue of dealing with the torsion linking form on middle dimensional cohomology for each of the indecomposable pieces. For the manifolds under consideration we manage to resolve both of these problems.

The key to proving the above theorems is a new topological description of highly connected manifolds, which involves a construction technique known as `plumbing'. This is a classic technique which goes back (at least) to Milnor in 1958 \cite{Mi1}. It involves the gluing together of collections of disc bundles. (We give a full description of this technique in Section \ref{sec:plumbing}.) The classic plumbing scenario involves plumbing disc bundles with fibre $D^{j}$ over spheres $S^{j}.$ The relevance of plumbing to questions of positive Ricci curvature is given by the following theorem:\\

\begin{Theorem} (\cite{Wr1}) The boundary of any simply-connected plumbing of $n$-disc bundles ($n \ge 3$) over $S^n$ admits a metric of positive Ricci curvature.
\end{Theorem}

In the light of the above theorem, it is reasonable to ask which manifolds can arise as the boundary of a simply-connected manifold constructed by plumbing collections of $D^{j}$-bundles over $S^{j}$. For example, it was shown in \cite{Wr1} that all exotic spheres which bound a parallelisable manifold arise in this way. Our next main result provides a complete answer to this question when $j = 2k \geq 4$:\\

\noindent{\bf Theorem C.} {\it For $k \geq 2$, the set of manifolds which arise as boundaries of simply-connected manifolds constructed by plumbing collections of $D^{2k}$-bundles over $S^{2k}$ coincides with the set of boundaries of handlebodies constructed by adding $2k$-handles to $D^{4k}.$ Moreover, for every $(2k{-}2)$-connected manifold $M^{4k-1}$ which is $(2k{-}1)$-parallelisable, there is a homotopy sphere $\Sigma^{4k-1}$ with $f(\Sigma) = 0$ such that $M \sharp \Sigma$ belongs to this class.}\\

It is clear that Theorem A follows immediately by combining Theorem C with Theorem 1.1. The reformulation Theorem A$'$ is proven by combining Theorem \ref{thm:delTP} and
Theorem 1.1. Theorem B follows from Theorem 1.1 by showing that in dimensions 7 and 11, all highly connected manifolds are boundaries of simply-connected plumbings. (This is Corollary \ref{cor:7-11}.)

In order to prove Theorem C, we employ the following strategy. We first show that for any $(2k{-}2)$-connected $(2k{-}1)$-parallelisable $(4k{-}1)$-manifold $M,$ there is a homotopy sphere $\Sigma$ for which $M \sharp \Sigma$ is the boundary of a handlebody which comprises some number of $2k$-handles added to $D^{4k}.$ (This is Theorem \ref{thm:delTP}~\eqref{thm:delTP:homotopy_sphere}.) We then show that the set of boundaries of such handlebodies coincides with the set of boundaries of simply-connected plumbings involving $D^{2k}$-bundles over $S^{2k}$ (Theorem \ref{thm:delTP}~\eqref{thm:delTP:main}). Establishing this last step turns out to be the major task in this paper. The problem easily reduces to the special case of handlebody boundaries which are rational homotopy spheres. These in turn are classified up to almost diffeomorphism (that is, up to connected sum with a homotopy sphere) by their `extended quadratic linking form' in Theorem \ref{thm:classification_of_QHS}. (Extended quadratic linking forms are introduced and discussed from an algebraic perspective in Section \ref{subsec:eqlfs}, and the way in which they arise in a topological context is explained in Section \ref{subsec:rational_homotopy_spheres}.) It then suffices to show that all extended quadratic linking forms can be realised by boundaries of simply-connected plumbings of $D^{2k}$-bundles over $S^{2k}$. In order to do this we observe that any such linking form can be expressed as a sum of indecomposables. In Section \ref{sec:realising_extended_quadratic_linking_forms} we show that each
of the three families of indecomposables can be realised by the boundary of a simply-connected plumbing of the required type. As a sum of indecomposables corresponds topologically to a connected sum of the plumbing boundaries, it remains to show that such a connected sum can be realised as the boundary of a single plumbing (see Theorems \ref{thm:connected_sum_for_delTP-G} or \ref{thm:TP_and_connected_sum}).

This last point is perhaps of independent interest.  In dimension $3$, it was already known that the connected sum
of the boundaries of plumbings was again the boudary of a plumbing: this is implicit in early work of Waldhausen \cite{Waldhausen},
and appears explicitly in work of Neumann \cite[Proposition 2.1]{Neumann2}.
Theorem \ref{thm:connected_sum_for_delTP-G} extends this result to higher dimensions,
providing general conditions under which a family of plumbing boundaries is closed under the connected sum operation.
From a topological point of view the connected sum operation is basic for constructing new manifolds from old. Moreover from a geometric perspective, there is the hope that Theorem \ref{thm:connected_sum_for_delTP-G}, when combined with suitable adaptations of Theorem 1.1 could be used to establish even wider families of Ricci positive examples. For further discussion, see Section \ref{subsec:geometric_background} below. Note that the proof of Theorem \ref{thm:connected_sum_for_delTP-G} is by explicit construction.

Although we are primarily concerned with manifolds of dimension $4k{-}1$ in this paper, we can also use our approach to say something about dimensions $4k{+}1$ (see Section \ref{subsec:4k+1}):\\

\noindent{\bf Theorem D.} {\it When $k \ge 1,$ for every $(2k{-}1)$-connected manifold $M^{4k+1}$ which is $2k$-parallelisable and has torsion-free (integer coefficient) cohomology, there is a homotopy sphere $\Sigma^{4k+1}$ such that $M \sharp \Sigma$ admits a metric of positive Ricci curvature.}\\

This paper is laid out as follows. After sketching the relevant geometric background to our results in \ref{subsec:geometric_background}, we turn our attention to the plumbing construction in Section \ref{sec:plumbing}, paying special attention to creating connected sums via this technique. Section \ref{sec:eqfs_and_eqlfs} is purely algebraic, in which we study extended quadratic forms and extended quadratic linking forms. These algebraic objects arise as topological invariants of the $(4k{-}1)$-dimensional manifolds we consider (in the case of extended quadratic linking forms) and of $4k$-dimensional bounding manifolds (extended quadratic forms). Topologically there is a close relationship between the linking form on the boundary and the quadratic form on the bounding manifold. Mirroring this, there is a purely algebraic `boundary' construction which we describe in Section \ref{subsec:boundaries_of_eqfs}. We also introduce `treelike' forms (in Section \ref{subsec:treelike_forms}) which are the extended quadratic forms that arise from simply-connected plumbings. In Section \ref{sec:handlebodies} we discuss handlebodies, especially those which consist of some number of $2k$-handles added to $D^{4k}.$  We describe Wall's classification of this family of handlebodies, and point out the connection with plumbing manifolds. 

In Section \ref{sec:(4k-1)_manifolds} we consider $(2k{-}1)$-parallelisable $(2k{-}2)$-connected $(4k{-}1)$-manifolds $M$ and state our main topological result,
Theorem \ref{thm:delTP}, which is a more detailed version of Theorem C.  The proof of Theorem \ref{thm:delTP} is based on the 
the classification of $(2k{-}1)$-parallelisable $(2k{-}2)$-connected $(4k{-}1)$-manifolds which are rational homotopy spheres, Theorem \ref{thm:classification_of_QHS}, which
is due to Wall and the first author.  The results in Section \ref{sec:(4k-1)_manifolds} reduce the proof of Theorem C to showing that every quadratic linking form on a finite abelian group can be presented as the `algebraic boundary' of an even symmetric bilinear treelike form. 
The novelty here is the word `treelike', since Wall \cite{Wa2} has proven this fact for general even symmetric bilinear forms.
We prove this algebraic result in Section \ref{sec:realising_extended_quadratic_linking_forms}.
This means that after the connected sum with a homotopy sphere, every $(2k{-}1)$-parallelisable $(2k{-}2)$-connected $(4k{-}1)$-manifold can be presented as 
the boundary of an explicit simply-connected plumbing manifold.

In Section \ref{sec:other_dimensions} we consider treelike plumbings of dimensions $4k+2$ and prove Theorem D.  We also briefly discuss treelike plumbings
in dimension $4$.  Finally, Section \ref{sec:symmetric_forms_are_stably_treelike} is a purely algebraic section, where we discuss an algebraic consequence
of our work, Theorem \ref{thm:treelike-forms}, which shows that treelike forms are in the appropriate sense, generic amongst even symmetric bilinear forms.

\vskip 0.1cm
\noindent
{\bf Acknowledgements:} We would like to thank Jim Davis, Anand Dessai, Karsten Grove, Matthias Kreck, Gabriele Nebe, Walter Neumann, Andrew Ranicki, Andr\'{a}s Stipsicz 
and Stephan Stolz for various helpful comments.  
The first author would like to thank the National University of Ireland Maynooth for its
hospitality during the 2013 Irish Geometry Conference, which directly supported research for this paper.
The first author acknowledges the support of the Leibniz Prize of Wolfgang L\"{u}ck, granted by the Deutsche Forschungsgemeinschaft.

\subsection{Geometric background} \label{subsec:geometric_background}

We begin this section with a few words about curvature for the benefit of readers whose background is not in Riemannian Geometry. For a more detailed summary of this issue, see \cite{JW} and the references therein. 

Any Riemannian manifold admits a wide range of Riemannian metrics - inner products on each of the tangent spaces which vary smoothly across the manifold. Each such metric endows the manifold with geometry, and in particular with curvature. As we will explain in more detail later, we are primarily interested in forms of positive curvature: this is special insofar as it is somewhat rare. There are three main measures of curvature: the sectional curvature, the Ricci curvature and the scalar curvature. The sectional curvature, which is a smooth real-valued function on the set of tangent 2-planes, is the strongest of the three measures of curvature. However the main geometric focus of this paper is on the Ricci curvature. Loosely speaking, this can be viewed as `curvature at a given point in a given tangent direction': indeed the Ricci curvature (precisely, the Ricci curvature of unit tangent vectors) is an average of sectional curvatures over tangent planes containing the given direction. In contrast, the scalar curvature is a smooth function on the manifold - in fact an average of Ricci curvatures - whose value at any point represents an overall measure of the curvature in an arbitrarily small neighbourhood of that point. 

Despite the weakness of the scalar curvature, positive scalar curvature has profound topological implications. In fact our understanding of these topological implications is increasingly complete: see for example \cite{Ro} for a survey of recent results. An early landmark in this area was a theorem of Gromov and Lawson \cite{GL} (proved independently by Schoen and Yau \cite{ScY}). This asserts that if a manifold $M$ admits a positive scalar curvature metric, any manifold constructed from $M$ via a surgery of codimension at least 3 will also admit a positive scalar curvature metric. (Note that the codimension 3 condition here is sharp.) This result led to a rapid development of our understanding of the topology of positive scalar curvature manifolds, and perhaps most famously to the classification of simply-connected positive scalar curvature manifolds in dimensions at least 5: see \cite{Stz}.

Given the success of surgery results for positive scalar curvature, it is reasonable to ask if similar results hold for other curvature conditions. Unfortunately there is no hope of proving such results for positive sectional curvature, as this curvature condition is simply too rigid to be amenable to this kind of toplogical modification. To date, all existence results for positive sectional curvature have arisen through the exploitation of symmetry. Roughly speaking, given a manifold which admits an effective Lie group action which in some sense is `large', we might hope to use this symmetry group to reduce the problem of finding positive sectional curvature metrics to a problem which could be tractable. In practice this is not so easy, and progress has been slow. As a stark illustration, the only simply-connected manifolds in dimensions greater than 24 known to admit positive sectional curvature metrics are rank-one symmetric spaces, see for example \cite{Z}.

Sitting between positive sectional and positive scalar curvature, one might hope to exploit both symmetry-based and surgery-based techniques to investigate the existence of metrics with positive Ricci curvature. Indeed both approaches have proved effective. For instance, taking a symmetry-based approach it has been possible to prove that all compact homogeneous spaces (\cite{Be}) and all compact cohomogeneity one spaces (\cite{GZ}) admit Ricci positive metrics if and only if the fundamental group is finite. More recently, families of Ricci positive manifolds have been constructed in higher cohomogeneities (\cite{BW} etc.) See also \cite{SW} for a new approach to finding Ricci positive manifolds with symmetry. On the topological side, one of the earliest major results about the existence of Ricci positive metrics was due to Nash \cite{Na}, who showed that a compact fibre-bundle will admit a Ricci positive metric if both its base and fibre admit such metrics, and the structural group is a Lie group which acts by isometries on the fibre. As an immediate corollary, we see that many exotic spheres in dimensions 7 and 15 admit Ricci positive metrics.

There are surgery results for positive Ricci curvature, which, while more limited than their counterparts in positive scalar curvature, are nontheless strong enough to have proved useful in various situations. (The basic difference between the two kinds of surgery results is that for positive Ricci curvature we need to make some assumption about the form of the metric in a neighbourhood of the surgery, whereas for positive scalar curvature a codimension condition suffices.) The first major result proved using Ricci positive surgery was due to Sha and Yang \cite{SY}: they prove that for $n,m \ge 2,$ any connected sum $\sharp_{i=1}^k S^n \times S^m$ admits a Ricci positive metric. This was significant since it showed that an upper bound on Betti numbers proved by Gromov for manifolds with a lower sectional curvature bound fails to hold in a Ricci positive context. (See also \cite{BG3} where some of the Sha-Yang examples are shown to admit Ricci positivity using Sasakian geometry.) This result was extended in \cite{Wr3}, where connected sums between products of pairs of spheres with possibly differing factor dimensions were shown to admit Ricci positive metrics. It is interesting to note that the problem of whether general connected sums between products of spheres involving more than two factors is completely open, and similarly for connected sums between products of other Ricci positive manifolds. The connected sum operation is basic in topology, and as a result it would be good to understand the extent to which Ricci positive (or other kinds of metric) can be extended across connected sums. Unfortunately, the connected sum is not so natural for positive Ricci curvature. It follows from Myers' Theorem \cite{My} that a connected sum between two non-simply connected manifolds cannot admit a Ricci positive metric, even if the individual manifolds are Ricci positive, since the resulting fundamental group must be infinite by the Seifert-Van Kampen theorem. On the other hand, it is an open question whether a connected sum between two {\it simply-connected} Ricci positive manifolds admits such a metric.

A further major application of Ricci positive surgery is in the study of exotic spheres. In \cite{Wr2}, a Ricci positive surgery theorem was proved which is applicable to surgeries using any trivialization of the normal bundle. (This was not the case for the Sha-Yang surgery result.) In \cite{Wr1} this enhanced surgery result was crucial in establishing that all exotic spheres which bound parallelisable manifolds admit Ricci positive metrics. (This result was subsequently re-proved using Sasakian geometry in \cite{BGN}.) This is a large family of exotic spheres, with the number of examples growing more than exponentially with dimension. The same surgery techniques also show that a number of exotic spheres which do not bound a parallelisable manifold also admit such metrics. However, to put this in context, it was shown by Hitchin \cite{Hi} that there are exotic spheres in dimensions 1 and 2 modulo 8 (beginning in dimension 9) which admit no metric of positive {\it scalar} curvature. Such manifolds clearly cannot support positive Ricci curvature! These exotic spheres do not bound any parallelisable manifold, however it is an open question which exotic spheres in this category admit or do not admit Ricci positive metrics. For a survey on the curvature of exotic spheres, see \cite{JW}.

It is philosophically reasonable to wonder - especially based on the initial examples one learns - if topological simplicity in some sense might facilitate the exisence of Ricci positive metrics. However, the above examples of Hitchin's `bad' exotic spheres show that one should be cautious. These are clearly simple from the point of view of continuous topology, even if they remain somewhat mysterious from the point of view of smooth topology. Further caution still is required in dimensions $n = 4k$: there is for example, a highly connected 8-manifold $M^8$ with the property that neither $M$ nor $M \sharp \Sigma$ admit positive scalar curvature, where $\Sigma$ is the unique exotic 8-sphere. (The contruction of $M$ and its failure to admit positive scalar curavture is discussed in Section 4 of \cite{Ca}; $\Sigma$ in contrast admits positive Ricci curvature, see \cite{Wr1} page 645.)  

Even though they form a rich class of examples, highly connected manifolds are by definition topologically simple.  So it is tempting in the light of Theorems A and D to conjecture that every highly connected manifold which admits a positive scalar curvature metric also admits a Ricci positive metric. 
Indeed, while there is no known obstruction to positive scalar curvature manifolds with finite fundamental group admitting positive Ricci curvature, see \cite[Section 6]{Wei}, 
there is a conjecture of Stolz \cite[Conjecture 1.1]{Stz0}, 
and independently H\"{o}hn, a special case of which we now state:\\
{\em If $M$ is a highly connected manifold of dimension $n > 8$ which admits a metric with positive Ricci curvature, then the Witten genus of $M$ vanishes.}\\ 
Here the Witten genus of $M$ is a certain power series whose coefficients are characteristic numbers of $M$: see \cite[\S 2]{Stz0}.  Extrapolating from the above conjecture, 
Stolz suggested that there may even be exotic spheres (necessarily with vanishing Witten genus, since they are stably parallelisable) which admit metrics of positive scalar curvature, but which do not admit metrics of positive Ricci curvature \cite[6.9]{Stz0}. Hence we are lead to formulate the following

\begin{Conjecture} \label{conj:pos_ricci}
Let $M$ be a highly connected manifold which admits a positive scalar curvature metric
and whose Witten genus vanishes.  Then there is a homotopy sphere 
$\Sigma$ such that $M \sharp \Sigma$ admits a Ricci positive metric.
\end{Conjecture}

\begin{Remark} \label{rem:pos_ricci}
Theorems A and D prove Conjecture \ref{conj:pos_ricci} for highly connected manifolds $M$ in all dimensions $4k-1 \geq 7$, (provided $M$ is also $(2k-1)$-parallelisable if $k \equiv 1$ mod 4), and in all dimensions $4k+1 \geq 5$, provided that $TH^{2k+1}(M) = 0$ and also that
$M$ is $2k$-parallelisable.
\end{Remark}


Note that our use of plumbing in this paper as a means to establish the existence of Ricci positive metrics is essentially a surgery-based approach. (See section 2 for details.)

For the sake of completeness, we should mention the work of Boyer and Galicki on 1-connected 5-manifolds \cite{BG1} (but see also \cite{BG2}). These objects were classified by Barden \cite{Bar} (following earlier work Smale \cite{Sm}). The classification falls into spin and non-spin cases. Using Sasakian geometry, Boyer and Galicki were able to establish the existence of Ricci positive metrics on most (but not all) of the spin family. In contrast, only two non-spin simply-connected 5-manifolds are known to admit Ricci positive metrics, though conjecturally all might be expected to display such metrics.

In conclusion let us give some indication why positive, as opposed to negative Ricci curvature conditions are particularly interesting. By work of Lohkamp (\cite{L1}, \cite{L2}, \cite{L3}) it turns out that negative scalar and negative Ricci curvatures are in some sense generic: any manifold - either compact or non-compact - of dimension at least 3 admits a complete metric or negative Ricci (and therefore negative scalar) curvature. Moreover, such negatively curved metrics are $C^0$-dense in the space of all Riemannian metrics. Thus there is a significant asymmetry between the negative and positive cases, and in particular negative Ricci and scalar curvatures have no topological implications.

\bigskip

\bigskip

\section{Plumbing}\label{sec:plumbing}
Plumbing is a construction which creates compact manifolds with boundary $W$ by gluing together a 
collection of disc bundles over closed manifolds.
In Section \ref{subsec:plumbing_arrangements} we review the construction of plumbing manifolds,
as well as the surgery description of the boundaries of plumbing manifolds.  We then describe the 
relevance of plumbing to the construction of Ricci positive metrics on the boundaries of plumbing manifolds.

A feature of the plumbing construction is that while the disjoint union of plumbing manifolds
$W_0$ and $W_1$ is by definition a plumbing manifold, it is by no means clear that the 
boundary connected sum $W_0 \natural W_1$ admits the structure of a plumbing manifold.
However, in Section \ref{subsec:connected_sums} we show that the situation with boundary
manifolds is often different.  Theorem \ref{thm:connected_sum_for_delTP-G} describes general
conditions for when the connected sum $\del W_0 \sharp \del W_1$ is the boundary of
a plumbing manifold.  This will be a key input to our arguments in Section \ref{sec:(4k-1)_manifolds}.

While techniques for constructing Ricci positive metrics via plumbing on $n$-manifolds only hold
for $n \geq 5$, in dimension $n=3$, the boundaries of plumbing manifolds, as we define them below, 
are a special case of graph manifolds.  It has long been known that the connected sum of two
graph manifolds is again a graph manifold: 
this follows by applying \cite[Proposition 2.1]{Neumann2} to the move splitting R6 \cite[p.\,305]{Neumann2}.
Hence Theorem \ref{thm:connected_sum_for_delTP-G} below can be viewed as an extension of a part of theory of graph manifolds
in dimensions $3$ to higher dimensions.  For a futher brief discussion of $4$-dimensional plumbings and their boundaries,
see Section \ref{subsec:3-manifolds}.

\subsection{Plumbing arrangements} \label{subsec:plumbing_arrangements}
We begin by recalling the construction of plumbing manifolds as described by Browder \cite[V \S2]{Bro}.
Let $p$ and $q$ be positive integers and set $m := p+q$.  
Consider oriented disc bundles $D^p \hookrightarrow E_1 \rightarrow B_1^q$ and $D^q \hookrightarrow E_2 \rightarrow B_2^p,$ 
with the fibres, base and total space of each bundle oriented compatibly. Now select points $x \in B_1$ and $y \in B_2.$ 
Restricting $E_1$ and $E_2$ to disc neighbourhoods $D^q_x \subset B_1,$ $D^p_y \subset B_2$ about the chosen points, we have a local product structure giving rise to diffeomorphisms $\pi_1^{-1}(D^q_x) \cong D^q_x \times D^p$ respectively $\pi_2^{-1}(D^p_y) \cong D^p_y \times D^q,$ 
where $\pi_i$ denotes the projection map of $E_i$ onto its base. Fixing these local trivializations, we next choose orientation preserving (reversing) diffeomorphisms $\phi_{+}:D^q \rightarrow D^p_y$ (respectively $\phi_{-}:D^p \rightarrow D^p_y$), and $\theta_{+}:D^q_x \rightarrow D^q$ (respectively $\theta_{-}:D^q_x \rightarrow D^q$). From these we construct diffeomorphisms 
$$I_{+}=(\phi_{+},\theta_{+}):D^p \times D^q_x \rightarrow D^p_y \times D^q;$$
$$I_{-}=(\phi_{-},\theta_{-}):D^p \times D^q_x \rightarrow D^p_y \times D^q.$$ 
Using either $I_{+}$ or $I_{-}$ together with the local trivializations, we can identify the two local bundle neighbourhoods to create a single manifold with boundary. It is easy to see that we can smooth this manifold near the site of the identification. This (smooth) manifold is the plumbing of $E_1$ and $E_2,$ which we will denote $E_1 \square E_2.$  Note that if one of $p$ or $q$ is even then $E_1 \square E_2$ is compatibly oriented with $E_1$ and $E_2$, whereas if $p$ and $q$ are both odd, $E_1 \square E_2$ is compatibly oriented with $E_1$ and $-E_2$.

In general, a plumbing manifold $W$ is a manifold constructed by a finite sequence of plumbings which can be described as follows.
Let $(\gamma_1, \dots, \gamma_r)$ be an ordered set of oriented smooth disc bundles $D^{r_i} \hookrightarrow E_i^m \to B^{s_i}$ 
where $(r_i, s_i) = (p, q)$ or $(q, p)$.  We then choose $E_{i_1}$ and $E_{i_2}$ with opposite fibre and base dimensions, and plumb them
together using either $I_+$ or $I_-$. We say we plumb with sign $+1$ if we use $I_{+},$ and sign $-1$ if we use $I_{-}.$  We continue in this manner a finite number of times - being careful to plumb on disjoint discs if we use the same bundle more than once.  
This {\em plumbing arrangement} for $W$ can be represented by a {\em labelled graph},
\[ \mathfrak{g}_{} = \bigl( (\V, \E), (\gamma_1, \dots, \gamma_r) \bigr), \]
which consists of a graph $(\V, \E)$ with ordered vertex set $\V = (v_1, \dots, v_r)$ and a set of  
directed edges $\E$ with signs $\epsilon(e) = \pm 1$ for each $e \in \E$, together with an $r$-tuple of labels $(\gamma_1, \dots, \gamma_r)$ with the label $\gamma_i$ being associated to vertex $v_i.$ The label $\gamma_i$ will represent the bundle with total space $E_i$. A directed edge $e$ in the graph from $v_i$ to $v_j$ with sign $\epsilon(e)$ then corresponds to plumbing $E_i$ to $E_j$ with sign $\epsilon(e)$. Thus a plumbing of disc bundles can be completely described by such a labelled graph.
Since we do not consider plumbing bundles with themselves, there are no edges from $v_i$ to itself.
Multiple edges from $v_i$ to $v_j$ will then represent multiple plumbings of $E_i$ with $E_j$.

The labelled plumbing graph $\mathfrak{g}$ determines the plumbing manifold $W = W(\mathfrak{g})$ up to 
diffeomorphism.
Except in Sections and \ref{subsec:connected_sums} and \ref{subsec:4k+1}, at least one of $p$ or $q$ will be even,
and in this situation we will always assume that our 
plumbing arrangements are constructed using the map $I_{+}$.
Moreover, when one of $p$  or $q$ is even, the ordering of edges is not important and so
in this case we omit both the signs and the ordering of edges from the plumbing graph.
When $p$ and $q$ are both odd, we will assume for simplicity that our plumbing graphs give
rise to orientable plumbing manifolds $W$ which are oriented compatibly with
the total space of the bundle $E_i$ with smallest value of $i$ in each connected component.

The case of plumbing most relevant for this paper is where $p=q=j \geq 2$ and $B_1=B_2=S^{j}.$ Oriented $D^{j}$-bundles over $S^{j}$ are classified by the homotopy class of their clutching maps, that is, by elements of $\pi_{j-1}(SO(j)) \cong \pi_{j}(BSO(j)).$ 


\begin{Definition} \label{def:plumbing_manifolds}
A $2j$-dimensional plumbing manifold $W = W(\mathfrak{g}_{})$ is a manifold obtained from a plumbing arrangement 
described by the labelled tree $\mathfrak{g}_{} = \bigl( (\V, \E), (\alpha_1, \dots \alpha_r) \bigr)$, $\alpha_i \in \pi_{j-1}(SO(j))$.  
We let $\cal{P}^{2j}$ denote the set of diffeomorphism class plumbing manifolds:
%
\[ \cal{P}^{2j} := \{ W(\mathfrak{g}_{}) \}. \]
\end{Definition}

It is easy to see that (each component of) a plumbing manifold $W = W(\mathfrak{g}_{})$ created by plumbing $D^{j}$-bundles over $S^{j}$ has the homotopy type of a wedge of  $1$-spheres and $j$-spheres, with one $j$-sphere for each bundle. Moreover the 1-skeleton of $W$ has the same homotopy type as the plumbing graph $(\V, \E)$. Thus $W$ is simply-connected if and only if the plumbing graph is simply-connected: i.e. if and only if $(\V, \E)$ is a tree.
In this case we shall call $W$ a {\em treelike plumbing manifold}.
We see immediately that a treelike plumbing $W$ is $(j-1)$-connected and that there is a homotopy equivalence $W \simeq \vee_{i=1}^rS^{2k}$,
since the base spheres of the plumbed bundles generate the homology group in dimension $2k$.  
In particular $H_{j}(W) \cong \Z^r$.  Moreover, if $j \geq 3$, then applying van Kampen's theorem to the boundary $\del W$,
we see that $\del W$ is simply connected.  

\begin{Definition} \label{def:treelike_plumbing_manifolds}
We let $\cal{TP}^{2j}$ denote the set of diffeomorphism class of treelike plumbing manifolds:
\[ \cal{TP}^{2j} \subset \cal{P}^{2j} := \{ W(\mathfrak{g}) \left| \text{$(\V, \E)$ is a tree} \right. \}. \]
%
\end{Definition}

We next turn our attention to the case $j = 2k$ and to 
the intersection of middle-dimensional homology classes, we see that the intersection number between any two distinct homology classes represented by base spheres is simply the number of plumbings between the corresponding bundles (assuming the identification map $I_{+}$ is used for all plumbings). In particular, in the treelike case, this number will either be 1 or 0. The intersection number between a homology class represented by the base sphere
of $E_i$ and itself is just the self-intersection number of the zero-section within the bundle in question, and this is well-known to be equal to the Euler number of the bundle $E_i$. (See for example \cite[V.1.5]{Bro}.) Expressing the intersection form with respect to this natural homology basis, we obtain a symmetric integer matrix $A$ with Euler numbers on the diagonal and non-negative numbers off the diagonal. Notice that the matrix $A$
and the bundles $E_i$ completely determine the plumbing arrangement. 
We thus obtain following result which can be found in \cite[V.2.1]{Bro}:
\begin{Theorem} \label{Browder} \label{thm:plumbing_realisation}
Let $(\alpha_1, \dots, \alpha_r)$ be a an $r$-tuple of homotopy classes $\alpha_i \in \pi_{2k-1}(SO(2k))$ with Euler numbers $a_i = e(\alpha_i),$ 
and let $A$ be an $(r \times r)$-symmetric integer matrix with diagonal entries $A_{ii} = a_i$ and non-negative off-diagonal entries.
Then for $k \geq 1$ there is a plumbing manifold $W^{4k}$ and a basis for $H_{2k}(W),$ such that with respect to this basis $A$ is the matrix of the intersection form $H_{2k}(W) \otimes H_{2k}(W) \rightarrow \Z.$
\end{Theorem}

We conclude this subsection by discussing the boundaries of plumbing manifolds from the point of view
of surgery and positive Ricci curvature.
It is not difficult to see that given bundles $D^p \hookrightarrow E_1 \rightarrow S^q,$ $D^q \hookrightarrow E_2 \rightarrow S^p$ together with balls 
$D_1^q \subset S^q$ and $D_2^p \subset S^p,$ we have 
$$\partial(E_1 \square E_2)=(\partial E_1-\pi_1^{-1}(D_1^q)) \cup (\partial E_2-\pi^{-1}_2(D_2^p)),$$
where $\pi_1$ and $\pi_2$ are the projection maps for the sphere bundles $\partial E_1,$ $\partial E_2.$ The union is formed using a certain diffeomorphism of the boundaries, which arises from choices of local trivializations for $\pi_1^{-1}(D_1^q)$ and $\pi_2^{-1}(D_2^p).$ We claim that this boundary construction can be viewed as a surgery performed on a fibre sphere of $\partial E_1.$ To see this, we need to look at the gluing diffeomorphism in more detail. First, fix a local trivialization 
$$\iota:\pi_1^{-1}(D_1^q) \rightarrow S^{p-1} \times D^q.$$ 
Such a trivialization is unique up to composition with a `twisting map' $T:S^{p-1} \times D^q \rightarrow S^{p-1} \times D^q,$ defined in terms of a smooth map $\tau:S^{p-1} \rightarrow \hbox{SO}(q)$ by $T(x,y):=(x,\tau(x)y).$ Next, we will view the bundle $\partial E_2$ as 
$$\partial E_2=(D_N^p \times S^{q-1}) \cup_{\psi} (D^p_S \times S^{q-1}),$$ 
where $D^p_N,$ $D^p_S$ denote `northern' and `southern' hemispheres of the base sphere $S^p,$ and $\psi$ is a clutching map 
$$\psi:\partial D^p_N \times S^{q-1} \rightarrow \partial D^p_S \times S^{q-1}.$$
We now observe that we can compose $\psi$ with the induced boundary map 
$$\partial \iota:\partial (\pi_1^{-1}(D_1^q)) \rightarrow S^{p-1} \times S^{q-1}$$
(after canonically identifying $\partial D^p_N \times S^{q-1}$ with $S^{p-1} \times S^{q-1}$) to obtain a diffeomorphism 
$$\psi \circ \partial \iota:\partial (\pi_1^{-1}(D_1^q)) \rightarrow \partial D^p_S \times S^{q-1},$$
and this composition is precisely the gluing map if we view $\partial(E_1 \square E_2)$ as 
$$\partial(E_1 \square E_2)=\bigl( \partial E_1-\pi_1^{-1}(D_1^q) \bigr) \cup (D^p_S \times S^{q-1}).$$
This construction is clearly surgery on a fibre sphere of $\partial E_1,$ as claimed. Notice that the clutching map $\psi$ takes the form 
$\psi(x,y)=(x,\tau(x)y)$ for some smooth map $\tau:S^{p-1} \rightarrow \hbox{SO}(q).$ We can extend $\psi$ to a map 
$T:S^{p-1} \times D^q \rightarrow S^{p-1} \times D^q$ in the obvious way, and then regard the gluing map $\psi \circ \partial \iota$ as the boundary map induced by the composition $T \circ \iota.$ So this surgery is precisely surgery on a fibre sphere of $\partial E_1$ using the normal bundle trivialization $T \circ \iota.$ Thus the structure of the bundle $E_2$ is encoded into the trivialization needed for the surgery. Of course, we could equally regard the boundary effect of this plumbing as a certain surgery on a fibre sphere of $\partial E_2.$

From a geometric perspective, plumbing is useful for constructing metrics with positive Ricci curvature. Its relevance is due to its close relationship with surgery, as described above. There are various results which show how, under the right conditions, surgery can be performed on a Ricci positive manifold in such a way that the Ricci positive metric can be extended across the surgery (see for example \cite{SY}, \cite{Wr2}). Roughly speaking, one needs a product metric in a neighbourhood of the surgery, where the factors of the product are a round sphere and a round normal disc. Moreover, one needs the radius of the sphere to be very small compared to the radius of the normal disc. This is the opposite situation to that usually encountered, and is a major reason why Ricci positive surgery results are difficult to apply. One situation where the right conditions are easily arranged is performing surgery on a fibre sphere of a sphere bundle over a Ricci positive base manifold. Since the boundary effect of plumbing two disc bundles over spheres is precisely such a surgery, it is not difficult to deduce that many manifolds plumbed from disc bundles over spheres according to a simply connected plumbing graph have a boundary which supports a Ricci positive metric 
(see Theorem 1.1 from the Introduction). Thus plumbing descriptions for bounding manifolds can be used as blueprints for constructing manifolds via a sequence of surgeries (starting from a sphere-bundle over a sphere), all of which are of the correct type to apply a Ricci positive surgery result. This can be an effective way of constructing Ricci positive manifolds, not least because the topology of plumbing manifolds and their boundaries is relatively easy to identify.

\subsection{Connected sums via plumbing} \label{subsec:connected_sums}
In this subsection we prove a connected sum theorem for boundaries of plumbings.  
Since the theorem holds for general plumbing manifolds, where the bases of the bundles
can be any connected closed smooth manifold, we first generalise the notation from
Section \ref{subsec:plumbing_arrangements}.

Let $p, q$ be poistive integers, let $m = p + q$ and let $\cal{C}^r$, $r = p, q$,
be a set of closed smooth oriented $r$-manifolds such that $S^r \in \cal{C}^r$.
We are now interested in $m$-manifolds $W$ which are obtained by plumbing smooth oriented bundles
$D^q \to E^m_i \to B^p_i$ and $D^p \to E^m_j \to B^q_j$ where $B^p_i \in \cal{C}^p$ and $B^q_j \in \cal{C}^q$.
As in Section \ref{subsec:plumbing_arrangements}, the plumbing arrangement for each $W$ can be described by a 
labelled graph $\mathfrak{g}_{}= \bigl((\V, \E), (\gamma_1, \dots \gamma_r)\bigr)$ 
where the labels $\gamma_i$ now label isomorphism classes of bundles $D^q \to E^m_i \to B^p_i$ or $D^p \to E^m_j \to B^q_j$, 
and we can only connect vertices which are labelled by bundles of complementary dimensions.
We call manifolds $W$ constructed in this way $(\cal{C}^p\!-\cal{C}^q$)-plumbings manifolds,
and we define 
\[ \cal{P}^m(\cal{C}^p, \cal{C}^q) := \{W(\mathfrak{g}) \}, \]
the set of diffeomorphism classes of $(\cal{C}^p, \cal{C}^q)$-plumbing manifolds.
With this notation the set $\cal{P}^{2j}$ of Definition \ref{def:plumbing_manifolds} is
$\cal{P}^{2j} = \cal{P}^{2j}(\{S^{j}\},\{S^{j}\})$.

As in Section \ref{subsec:plumbing_arrangements}, we call a plumbing manifold 
$W((\V, \E), (\gamma_1, \dots, \gamma_r))$ treelike if the graph $(\V, \E)$ is a tree and we define
\[ \cal{TP}^m(\cal{C}^p, \cal{C}^q) \subset \cal{P}^m(\cal{C}^p, \cal{C}^q), \]
to be the set of diffeomorphism classes of treelike $(\cal{C}^p\!-\cal{C}^q$)-plumbings.  We also define
\[ \del \cal{TP}^m(\cal{C}^p, \cal{C}^q) \subset \del \cal{P}^m(\cal{C}^p, \cal{C}^q), \]
%
to be the set of diffeomorphism classes of manifolds which are the boundaries of treelike $(\cal{C}^p\!-\cal{C}^q$)-plumbings,
respectively the boundaries of $(\cal{C}^p\!-\cal{C}^q)$-plumbings.

\begin{Theorem} \label{thm:connected_sum_for_delTP-G}
The sets $\del \cal{TP}^m(\cal{C}^p, \cal{C}^q)$ and $\del \cal{P}^m(\cal{C}^p, \cal{C}^q)$ are closed under connected sum.
\end{Theorem}

\begin{Remark} \label{rem:connected_sum_for_delTP-G}
For $p = q = 2$ and $\cal{C}^2 = \{S^2\}$ Theorem \ref{thm:connected_sum_for_delTP-G} follows from results in
\cite{Scharf}: see the move denoted RIII, \cite[p.\,73]{N-W}.  For $\cal{C}^2$ the set of all surfaces,
Theorem \ref{thm:connected_sum_for_delTP-G} follows from \cite[Proposition 2.1]{Neumann2} applied to the splitting move R6 \cite[p.\,305]{Neumann2}.
\end{Remark}

The rest of the Subsection is devoted to the proof of Theorem \ref{thm:connected_sum_for_delTP-G}.
Let $W_\infty$ be the plumbing manifold obtained by plumbing the trivial bundles
$D^q \times S^p$ and $D^p \times S^q$.  If $0_r$, $r = p, q$, denotes the trivial bundle
over $S^r$, then $W_\infty$ has plumbing graph
\[ \xymatrix{ 0_p \ar@{-}[r] & 0_q.}  \]
We observe that $W_\infty$ is diffeomorphic to the manifold
obtain by removing a small open disc from $S^p \times S^q$, so that $\del W_\infty \cong S^m$. 
To see this, simply write $S^p=D^p_+ \cup D^p_-$ and $S^q=D^q_+ \cup D^q_-$, and identify $D^m$ with $D^p_- \times D^q_-.$ Thus 
$$W_0=(D^p_+ \times D^q_+) \cup (D^p_+ \times D^q_-) \cup (D^p_- \times D^q_+)$$
which simplifies to 
$$W_0= (S^p \times D^q_-) \cup_{D^p_+ \times D^q_+} (D^p_+ \times S^q),$$
which is easily seen to be the plumbing of two trivial bundles as claimed.
Given two bundles $D^p$ bundles $\gamma_1$ and $\gamma_2$
over $q$-dimensional bases, we form the treelike plumbing manifold $Y$ with the following plumbing graph,
where the $+$ labels on the edges indicate that we plumb with sign $+1$.
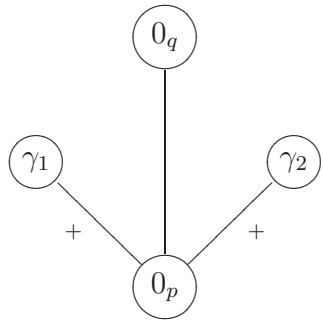
\begin{equation} \label{eq:Y}
\entrymodifiers={++[o][F-]}
\xymatrix{ *\txt{} & 0_q \ar@{-}[dd] \\ \gamma_1 \ar@{-}[dr]_(0.425){+} & *\txt{} & \gamma_2 \ar@{-}[dl]^(0.425){+} \\ *\txt{} & 0_p }  
\end{equation}
%

\begin{Proposition} \label{prop:basic}
If $\gamma_1$ and $\gamma_2$ are $D^p$-bundles over $q$-dimensional bases with total spaces $E_1$ and $E_2$, 
then the treelike plumbed manifold $Y$ obtained by plumbing $E_1$ and $E_2$ to $W_\infty$ as above in 
\eqref{eq:Y} 
is such that $\partial Y=\partial E_1 \sharp \partial E_2.$
\end{Proposition}

\begin{Remark} \label{rem:Scharf}
The reader can compare \eqref{eq:Y} and Proposition \ref{prop:basic} with the move RIII \cite[p.\,73]{N-W} and also the
splitting move R6 of \cite[p.\,305]{Neumann2} and \cite[Proposition 2.1]{Neumann2}.
\end{Remark}

\begin{proof}[Proof of Theorem \ref{thm:connected_sum_for_delTP-G} assuming Proposition \ref{prop:basic}] 

If suffices to consider two connected plumbing manifolds $W_1$ and $W_2$.
Suppose in the first case that the plumbing arrangements for each $W_i$ contain $D^p$-bundles with total spaces $E_i$ and $q$-dimensional bases $B^q_i$, $i = 1, 2$.
Hence we can write
$W_i=(W_i-E_i)\cup E_i,$ $i=1,2.$
We consider the disjoint union $W_1 \sqcup W_\infty \sqcup W_2$ and then plumb the bundles $E_i \subset W_i$
to $W_\infty$ according to the diagram \eqref{eq:Y},
that is, we incorporate $E_i$ into the plumbing arrangement of Proposition \ref{prop:basic}.  This has the overall effect of creating a new connected plumbed manifold $W$ given by 
$$W=(W_1-E_1)\cup Y \cup (W_2-E_2).$$ 
Clearly, plumbing a bundle in one treelike plumbing manifold to a bundle in a disjoint treelike plumbing manifold by 
a single plumbing creates a new treelike manifold. 
Applying this observation twice, we see that $W$ is treelike if the $W_i$ are treelike. It is now straightforward to see that 
$$\del W=\partial(W_1-E_1)\cup_{\partial E_1} \del E_1 \sharp \del E_2  \cup_{\partial E_2} \partial(W_2-E_2),$$ 
and this expression simplifies to $\partial W= \del W_1 \sharp \del W_2$ as required.

The above argument covers all situations except the case where at least one of $W_1$ or $W_2$ is a $D^q$-bundle over a $p$-dimensional base. 
Now recall that $S^{m-1}$ is the boundary of $W_\infty = (S^p \times S^q) - \textup{int}(D^m)$ 
is the plumbing of the trivial $D^p$-bundle over $S^q$ with the trivial $D^q$-bundle over $S^p$. 
We now apply the argument of the previous paragraph to $W_i$ and $W_{\infty}$ 
to show that $\del W_i \cong \del W_i \sharp S^{m-1}$ is the boundary of a plumbing manifold
$W_i'$ (treelike if $W_i$ is treelike), where the plumbing involves a $D^p$-bundle over a $q$-dimensional base. 
Thus the argument of the previous paragraph can be applied again to the new plumbing manifolds $W_i'$ 
to yield a final plumbing manifold $W$ (treelike if $W_1$ and $W_2$ are treelike) with boundary 
$\del W \cong \del W'_1 \sharp \del W'_2 \cong \del W_1 \sharp \del W_2.$
\end{proof} 

For the proof of Proposition \ref{prop:basic} we require three lemmas.  For the statement of Lemma \ref{lem:product} below,
we recall some basic terms of surgery.  Let $N$ be a compact $n$-manifold, either with boundary or closed, let $n = s+t-2$
and let $\phi \colon S^{s-1} \times D^{t-1} \to N$ be an embedding into the interior of $N$.  Then the trace of surgery on $\phi$
is the manifold 
\[ W_\phi : = (N \times I) \cup_\phi (D^{s} \times D^{t-1}) \]
where we regard $S^{s-1} = \del D^{s}$ and $S^{s-1} \times D^{t-1} \subset \del(D^{s} \times D^{t-1})$ as the domain of $\phi$, with the union taking place at one end of the product $N \times I.$ 
The outcome of surgery on $\phi$ is the manifold 
\[ N' : = (N \setminus \phi(S^s \times {\rm Int}(D^{t-1}))) \cup_{S^{s-1} \times S^{t-2}} (D^{s} \times S^{t-2}). \]
In particular, the trace of surgery on $\phi$ is a manifold with boundary
\[ \del W_\phi = N \cup (\del N \times I) \cup N'. \]
For the statement of the following lemma, let $i_\epsilon \colon [-\epsilon, \epsilon] \to [-1, 1]$ be the standard inclusion.

\begin{Lemma} \label{lem:product}
Let $N$ be a closed $n$-manifold for $n = s+t-2$.
Suppose that surgery on an embedding $\phi \colon S^{s-1} \times D^{t-1} \to N$ has outcome $N'$ and trace $W_\phi$.
Consider performing surgery on the embedding $\phi \times i_\epsilon \colon S^{s-1} \times D^{t-1} \times [-\epsilon, \epsilon] \to N \times [-1, 1]$
The outcome of this is the manifold $W_\phi \cup_{N'} W_\phi$ and has trace $W_\phi \times [-1,1].$
\end{Lemma}

\begin{proof} 

Suppose that the trace $W_\phi$ is formed from the original manifold by forming the product with an interval $J=[j_0,j_1],$ 
and then adding a handle to the `end' of the product corresponding to $j_1.$  Hence 
$W_\phi=(N \times J) \cup (D^s \times D^{t-1} \times \{j_1\}).$

Suppose that the surgery on $N \times [-1,1]$ involves cutting out $S^{s-1} \times D^{t-1} \times (-\epsilon,\epsilon).$ 
The trace of this surgery will clearly be diffeomorphic to the trace created where the $D^t$-factor of the handle 
$D^s \times D^t$ contains all of $[-1,1].$ This latter trace is 
$$(N \times [-1,1] \times J) \cup (D^s \times D^{t-1} \times [-1,1] \times \{j_1\}).$$ 
By a simple rearrangement of terms, this is easily seen to be $W_\phi \times I$ as claimed.

To identify the result of the surgery on $N \times [-1,1],$ recall that $\partial W_\phi=N \sqcup N'.$ 
Therefore $\partial(W_\phi \times [-1,1]) = (\partial W_\phi \times [-1,1]) \cup (W_\phi \times \{-1,1\}).$ In other words we have 
\begin{align*} 
\partial(W_\phi \times [-1,1])&=((N \sqcup N') \times [-1,1]) \cup (W_\phi \times \{-1\}) \cup (W_\phi \times \{1\}) 
\\ &=(N \times [-1,1]) \cup \Bigl[(N' \times [-1,1]) \cup (W_\phi \times \{-1\}) \cup (W_\phi \times \{1\})\Bigr]. 
\end{align*}
We claim that the outcome of the surgery on $N \times [-1,1]$ is the `other' part of the boundary $\partial(W_\phi \times I),$ 
namely $(N' \times [-1,1]) \cup (W_\phi \times \{-1\}) \cup (W_\phi \times \{1\}).$ 
To see this, note that a neighbourhood of the $N'$-boundary of $W_\phi$ can be expressed as 
$(D^s \times D^{t-1} \times \{j_1\})\cup (N \times [j_1-\epsilon,j_1))$ for small $\epsilon >0.$ Since this neighbourhood is diffeomorphic to $W_\phi$ there is a diffeomorphism
\begin{multline*}
(N' \times [-1,1]) \cup (W_\phi \times \{-1\}) \cup (W_\phi \times \{1\}) \\
\cong (N' \times [-1,1]) \cup (N \times [j_1-\epsilon,j_1) \times \{-1\}) \cup (N \times [j_1-\epsilon,j_1) \times \{1\}).
\end{multline*} 
But this is clearly the same as the result of surgery on $N \times [-1-\epsilon,1+\epsilon]$ where we cut out the manifold
$S^{s-1} \times D^{t-1} \times (-1,1),$ and thus diffeomorphic to the result of performing a surgery on $N \times [-1,1]$ where we cut out $S^{s-1} \times D^{t-1} \times (-\epsilon,\epsilon).$ We therefore conclude that the result of the surgery is just two copies of $W_\phi$ glued along $N' \times [-1,1].$ Up to diffeomorphism this is $W_\phi \cup_{N'} W_\phi$ as claimed.
\end{proof} 


\begin{Lemma}\label{lem:trace}
Suppose we perform surgery on $\phi \colon S^{q-1} \times D^{p-1} \times [-\epsilon, \epsilon] \to S^{q-1} \times S^{p-1} \times [-1, 1]$,
where $\phi$ is the product of 
the identity map, the inclusion $D^{p-1} \subset S^{p-1}$ and $i_\epsilon$.
The outcome of this surgery is $(S^{p-1} \times D^q) \sharp (S^{p-1} \times D^q).$
\end{Lemma}

\begin{proof} The effect of surgery on $S^{q-1} \times D^{p-1} \subset S^{q-1} \times S^{p-1}$ is 
\begin{align*}
(S^{q-1} \times S^{p-1} \setminus S^{q-1} \times D^{p-1}_{+}) \cup (D^q \times S^{p-2}) &= S^{q-1} \times D^{p-1}_{-} \cup D^m \times S^{p-2} \\ &= S^{p+q-2}, 
\end{align*} 
where we have written the $S^{p-1}$ factor as $D^{p-1}_{+} \cup D^{p-1}_{-},$ with $D^{p-1}_{+}$ taken to be the normal bundle of the surgery sphere. Thus we see that the trace $W$ of the surgery is 
$$W=S^{q-1} \times S^{p-1} \times J \cup (D^q \times D^{p-1}_{+} \times \{j_1\}),$$ 
where as in the previous lemma we assume the interval $J$ is $[j_0,j_1].$ Up to diffeomorphism $W$ is just 
$D^q \times D^{p-1}_{+}.$ If we then add $D^q \times D^{p-1}_{-} \cong D^{p+q-1}$ we obtain 
$D^q \times S^{p-1}.$ Thus 
$$W=D^q \times S^{p-1} \setminus \hbox{int}D^{p+q-1}.$$ 
By Lemma \ref{lem:product} we see that the result of performing the surgery on $\phi$ is $W \cup_{S^{p+q-2}} W,$ that is, 
$$(D^q \times S^{p-1} \setminus \hbox{int}D^{p+q-1}) \cup_{S^{p+q-2}} (D^q \times S^{p-1} \setminus \hbox{int}D^{p+q-1}),$$ 
and this is just $(S^{p-1} \times D^q) \sharp (S^{p-1} \times D^q)$ as claimed.
\end{proof}

Now let $E_1$ and $E_2$ be $D^p$-bundles with base dimension $q$.

\begin{Lemma}\label{lem:FCS} 
Plumb both $E_1$ and $E_2$ to a trivial disc bundle $E_3:=D^q \times S^p.$ 
Call the resulting plumbing manifold $X.$ Then $\partial X$ is the fibre connected sum of $\del E_1$ and $\del E_2.$
\end{Lemma}

\begin{proof} The boundary effect of plumbing is well known (e.g. \cite{Bro} page 118). In particular the boundary $\partial X$ can be decomposed as the following union: 
$$(\del E_1 \setminus S^{p-1} \times D^q) \cup (E_3 \setminus \coprod_{j=1}^2 S^{q-1} \times D^p_j) \cup (\del E_2 \setminus S^{p-1} \times D^q),$$ 
where the identification at the various boundaries is made in the obvious way. But the middle component in the above decomposition is homeomorphic to $S^{q-1} \times S^{p-1} \times [-1,1],$ and so $\partial X$ is nothing other than the fibre connected sum of $E_1$ and $E_2$ as claimed.
\end{proof}

\begin{proof}[Proof of Proposition \ref{prop:basic}] Let $X$ denote the plumbed manifold in Lemma \ref{lem:FCS}. Consider a further plumbing of a trivial disc bundle $E_4:=D^p \times S^q$ to $E_3$ in the plumbing arrangement for $X$. Call the resulting plumbing manifold $Y$. We claim that 
$\partial Y= \del E_1 \sharp \del E_2.$ To see this, first note that the boundary effect of plumbing this trivial bundle is precisely a surgery on a fibre sphere of the bundle $\del E_3,$ or more accurately a surgery on a fibre sphere of 
$$\del E_3 \setminus \coprod_{j=1}^2 S^{q-1} \times D^p_j.$$ 
In order to investigate the topological effect of such a surgery we will work locally in the first instance, and consider 
$\del E_3\setminus \coprod_{j=1}^2 S^{q-1} \times D^p_j$ in isolation. With this local point of view, our surgery is equivalent to surgery on an $S^{q-1}$ factor of the product $S^{q-1} \times S^{p-1} \times [-1,1].$ By Lemma \ref{lem:trace}, the outcome of this surgery is 
$$(S^{p-1} \times D^q) \sharp \ (S^{p-1} \times D^q).$$
Thus returning to a global perspective, the boundary of the plumbing $Y$ using all four bundles must be 
$$(\del E_1 \setminus S^{p-1} \times D^q) \cup (S^{p-1} \times D^q) \sharp (S^{p-1} \times D^q) \cup (\del E_2 \setminus S^{p-1} \times D^q),$$ 
which is clearly just $\del E_1 \sharp \del E_2$ as claimed.
\end{proof}
 


\section{Extended quadratic forms and extended quadratic linking forms}
\label{sec:eqfs_and_eqlfs}
In this algebraic section we establish basic facts from the theory of extended quadratic 
forms and their boundaries, which are extended quadratic linking forms.  Extended quadratic forms
are complete invariants of the handlebodies appearing in Section \ref{sec:handlebodies}
and extended quadratic linking forms are important invariants of the odd-dimensional
manifolds appearing Section \ref{sec:(4k-1)_manifolds}.



\subsection{Extended quadratic forms over $\pi_{n-1}\{SO(n)\}$} \label{subsec:eqfs}
Let $SO(n)$ denote the group of orientation preserving orthogonal motions of $\R^n$.
We start with the homotopy group $\pi_{n-1}(SO(n))$ which we regard as the set of isomorphism
classes of oriented vector bundles over $S^n$.  There are two important homomorphisms
involving $\pi_{n-1}(SO(n))$:
\[ {\rm h} \colon \pi_{n-1}(SO(n)) \to \Z, \quad \xi \mapsto e(\xi) \quad \text{and} \quad 
{\rm p} \colon \Z \to \pi_{n-1}(SO(n)), \quad k \mapsto k \cdot \tau_{S^n}. \]
The homomorphism ${\rm h}$ associates to each bundle its Euler number and the homomorphism ${\rm p}$
maps $1 \in \Z$ to $\tau_{S^n}$, the tangent bundle of $S^n$.
The homomorphisms ${\rm h}$ and ${\rm p}$ are such that 
${\rm h}{\rm p}{\rm h} = 2 {\rm h}$ and ${\rm p}{\rm h}{\rm p} = {\rm 2p}$, since $e(\tau_{S^n})=2$ if
$n$ is even and $e(\tau_{S^n}) = 0$ if $n$ is odd.
The quadruple 
\[ \pi_{n-1}\{SO(n)\} : = \bigl( \pi_{n-1}(SO(n)), \Z, {\rm h}, {\rm p} \bigr) \]
is a {\em quadratic module} as defined in \cite[Definition 8.1]{Bau1}; see also \cite[\S 10]{Ranicki}.
Let $\epsilon := (-1)^n$. A bilinear form is said to be $\epsilon$-symmetric if it is symmetric in the case $\epsilon=1$ and skew-symmteric in the case $\epsilon=-1.$

\begin{Definition} \label{def:eqf}
An {\em extended quadratic form over $\pi_{n-1}\{SO(n)\}$} is 
triple $(H, \lambda, \mu)$ where $H$ is a finitely generated free abelian group, $\lambda$ 
is an $\epsilon$-symmetric bilinear form 
\[ \lambda \colon H \times H \to \Z, \]
and $\mu \colon H \to \pi_{n-1}(SO(n))$ is
a function which for all $x, y \in H$ satisfies 
\begin{equation} \label{eq:mu_and_p}
\mu(x + y) = \mu(x) + \mu(y) + {\rm p}(\lambda(x, y)),
\end{equation}
and
\begin{equation} \label{eq:mu_and_h}
\quad {\rm h}(\mu(x)) = \lambda(x, x). 
\end{equation}
%
\end{Definition}
\noindent
There are obvious notions of isomorphism and orthogonal sum of extended quadratic forms.

In this paper we are mostly interested in the symmetric case when $n = 2k$ is even.  In this
case extended quadratic forms over $\pi_{2k-1}\{SO(2k)\}$ simplify as follows.
Let $S \colon \pi_{2k-1}(SO(2k)) \to \pi_{2k-1}(SO)$ be the suspension homomorphism.
By Bott periodicity, 
\[ \pi_{2k-1}(SO) \cong \Z,~\,\Z/2,~\,\Z,~\,0 \quad \text{as \quad $k \equiv 0,~\,1,~\,2,~\,3$~mod~\,$4$.} \]
Given an extended quadratic form $(H, \lambda, \mu)$, let
\[ \alpha \colon H \to \pi_{n-1}(SO) \]
be the stabilisation of $\mu$, that is, the composition $S \circ \mu.$ Since $S(\tau_{S^n}) = 0$, it follows from \eqref{eq:mu_and_p}
that $\alpha$ is a homomorphism.  Wall identified the properties of the pair $(\lambda, \alpha)$
as follows.  (Note that Wall uses $\alpha$ to denote what we call $\mu$ and $S \alpha$ to denote
what we call $\alpha$.)

\begin{Lemma}[{\cite[p.\,171]{Wa1}}] \label{lem:mu_and_alpha}
For every extended quadratic form $(H, \lambda, \mu)$ over $\pi_{2k-1}\{SO(2k)\}$,
$\mu(x)$ is uniquely determined by $\lambda(x, x) \in \Z$ and $\alpha(x) \in \pi_{2k-1}(SO)$.
In addition,
\begin{enumerate}
\item for $k = 2, 4$, $\lambda(x, x) \equiv \alpha(x)~mod~2$, (i.e.~$\alpha$ is {\em characteristic} for $\lambda$),
\item for $k \neq 2, 4$, $\lambda(x, x) \in 2 \Z$ (and $\alpha$ and $\lambda$ are not related).
\end{enumerate}
\end{Lemma}

In the light of Lemma \ref{lem:mu_and_alpha}, we shall 
identify the extended quadratic form
$(H, \lambda, \mu)$ with the triple $(H, \lambda, \alpha)$.  We define
\[ \cal{F}^{4k} := \{(H, \lambda, \alpha) \} \]
to be the set of isomorphism classes of extended quadratic forms over $\pi_{2k-1}\{SO(2k)\}$.

\subsection{Extended quadratic linking forms} \label{subsec:eqlfs}
In this subsection we define extended quadratic linking forms over $\pi_{2k-1}\{SO(2k)\}$. Given an extended quadratic form over $\pi_{2k-1}\{SO(2k)\}$ there is a corresponding `algebraic boundary' (discussed in Section \ref{subsec:boundaries_of_eqfs}). This algebraic boundary is an extended linking form. Such extended linking forms can be viewed as the torsion analogues of extended quadratic forms.

Let $G$ be a finite abelian group.  A nonsingular symmetric linking form of $G$ is a bilinear
function
\[ b \colon G \times G \to \Q/\Z \]
such that for all $x, y \in G$,
\begin{enumerate}
\item $b(x, y) = b(y, x)$,
\item $b(x, y) = 0$ for all $y \in G$ if and only if $x  = 0$.
\end{enumerate}

\begin{Example} \label{exa:linking_form}
Let $N$ be a closed $(4k{-}1)$ manifold and let $TH^{2k}(N)$ denote the torsion subgroup of $H^{2k}(N)$.
Poincar\'{e} duality for $N$ gives rise to a linking form 
$b_N \colon TH^{2k}(N) \times TH^{2k}(N) \to \Q/\Z$.  See Section \ref{subsec:rational_homotopy_spheres}.
\end{Example}

A quadratic refinement of a linking form $b$ is a function
\[ q \colon G \to \Q/\Z \]
such that for all $x, y \in G$
\[ q(x + y) - q(x) - q(y) = b(x, y). \]

For any such linking form $b,$ a corresponding $q$ exists \cite[Lemma 2.30]{C1}.  
%
%
%
The function $q$ has further properties, depending on $k$.  If $k = 2, 4$ then
\[ q(x) - q(-x) = b(x, \beta) \]
for some $\beta \in 2 \cdot G$: The element $\beta$ is called the {\em homogeneity defect} of $q$.
If $k \neq 2, 4$, then 
\[ q(x) = q(-x),\]
in which case $q$ is {\em homogeneous}.

\begin{Definition} \label{def:eqlfd}
An extended quadratic linking form over $\pi_{2k-1}\{SO(2k)\}$ is a quadruple $(G, b, q, \beta)$ where
\begin{enumerate}
\item $G$ is a finite abelian group;
\item $b \colon G \times G \to \Q/\Z$ is a nonsingular symmetric linking form on $G$;
\item An element $\beta$ where
\[ \beta \in \left\{ \begin{array}{cl} 
(2 \cdot G) \otimes \pi_{2k-1}(SO)& \text{if $k = 2, 4$,}\\
G \otimes \pi_{2k-1}(SO) & \text{if $k \neq 2, 4$;} 
\end{array} \right. \]
\item $q \colon G \to \Q/\Z$ is a quadratic refinement of $b$ which is homogeneous if $k \neq 2, 4$
and which has homogeneity defect $\beta$ if $k = 2, 4$;
\end{enumerate}
\end{Definition}

There are obvious notions of isomorphism and orthogonal sum of extended quadratic linking forms.
We define
\[ \cal{Q}^{4k-1} =\{ (G, b, q, \beta) \} \]
to be the set of isomorphism classes of extended quadratic linking forms over $\pi_{2k-1}\{SO(2k)\}$.

The classification of linking forms and their quadratic refinements is a tractable if delicate
problem.  For more on these topics, see \cite{Wa2, K-K, C1}.


\subsection{The boundaries of nondegenerate extended quadratic forms} \label{subsec:boundaries_of_eqfs}
We first identify two important subsets of $\cal{F}^{4k}$.
The adjoint homomorphism of the symmetric form $(H, \lambda)$ is the homomorophism
\[ \wh \lambda \colon H \to H^*, \quad x \mapsto (y \mapsto \lambda(x, y)), \]
where $H^* : = \Hom(H, \Z)$ is the dual of $H$.
A bilinear form $(H, \lambda)$ is called {\em nonsingular} if $\wh \lambda$ is an isomorphism
and {\em nondegenerate} if $\wh \lambda$ lies in a short
exact sequence
\begin{equation} \label{eq:adjoint}
0 \xra{~~~} H \xra{~\wh \lambda~} H^* \xra{~\pi~} G \to 0,
\end{equation}
where $G : = \coker(\wh \lambda)$ is a finite group, since it is a quotient of a finitely generated abelian 
group by a subgroup of the same rank.
We call an extended quadratic form $(H, \lambda, \alpha)$ nondegenerate, respectively nonsingular, if $(H, \lambda)$ is 
nondegenerate, respectively nonsingular.  We define 
\[ \cal{F}^{4k}_{\rm nd} : = \{(H, \lambda, \alpha) \,|\, \text{$(H, \lambda)$ is nondegenerate} \} \subset \cal{F}^{4k}, \]
and
\[ \cal{F}^{4k}_{\rm ns} : = \{(H, \lambda, \alpha) \,|\, \text{$(H, \lambda)$ is nonsingular} \} \subset \cal{F}^{4k}_{\rm nd}, \]
to be the the set of isomorphism classes of nondegenerate, respectively nonsingular, extended quadratic forms over $\pi_{2k-1}\{SO(2k)\}$.

Henceforth we assume that $(H, \lambda, \alpha) \in \cal{F}^{4k}_{\rm nd}$.
If we tensor the short exact sequence \eqref{eq:adjoint} with the rationals then, since $G$ is finite, we obtain an isomorphism
\[ \wh \lambda_\Q \colon H \otimes \Q \cong H^* \otimes \Q. \]
In particular $\wh \lambda_\Q$ is invertible with inverse $(\wh \lambda_\Q)^{-1}$
which induces a nondegenerate form on $H^*$ by pulling back $\lambda_\Q$, 
the extension of $\lambda$ to $H \otimes \Q$.  We denote this form by $\lambda^{-1}$:
\[ \lambda^{-1} \colon H^* \times H^* \to \Q.\]
Explicitly, $\lambda^{-1}$ is the composition $$H^* \times H^* \xra{(\hbox{id}\otimes 1,\hbox{id}\otimes 1)} (H^* \otimes \Q) \times (H^* \otimes \Q) \xra{(\wh \lambda_\Q)^{-1} \times (\wh \lambda_\Q)^{-1}} (H \otimes \Q) \times (H \otimes \Q) \xra{\lambda_\Q} \Q.$$ 

To define the algebraic boundary of $(H, \lambda, \alpha)$, we must distinguish the cases $k = 2, 4$, when
$\alpha$ is characteristic for $\lambda$ (in the sense that it determines $\lambda$ by Lemma \ref{lem:mu_and_alpha}), from the cases $k \neq 2, 4$ when $\lambda$ is even.  Notice that
$\alpha \in H^*$ if $k$ is even, $\alpha \in H^* \otimes \Z_2$ if $k \equiv 1$~mod~$4$, and $\alpha=0$ if $k \equiv 3$~mod~$4$.

\begin{Definition} \label{def:boundary_of_an_eqf}
Let $(H, \lambda, \alpha) \in \cal{F}^{4k}_{\rm nd}$.  
The boundary of $(H, \lambda, \alpha)$ is the quadruple
\[ \del(H, \lambda, \alpha) = (G, b, q, \beta), \]
where $G : = \coker(\wh \lambda)$ and $b$, $q$ and $\beta$ are defined as follows:
%
\begin{enumerate}
\item \label{def:boundary_of_an_eqf:linking_form}
For all $k$,
\[ b \colon G \times G \to \Q/\Z, \quad (\pi(x), \pi(y)) \longmapsto \lambda^{-1}(x, y)~\text{mod}~\Z, \]
where $\pi:H^* \rightarrow G$ is the map from \eqref{eq:adjoint}.
\item \label{def:boundary_of_an_eqf:q_k=2,4}
If $k = 2, 4$ then
\[ q \colon G \to \Q/\Z, \quad \pi(x) \longmapsto 
\frac{\lambda^{-1}(x, x) + \lambda^{-1}(x, \alpha)}{2} ~\text{mod}~\Z , \]
and $\beta : = \pi(\alpha) \in G$.  
%
%
\item \label{def:boundary_of_an_eqf:q_k<>2,4}
If $k \neq 2, 4$, then 
\[ q \colon G \to \Q/\Z, \quad \pi(x) \longmapsto \frac{\lambda^{-1}(x, x)}{2} ~\text{mod}~\Z , \]
and $\beta : = \pi(\alpha) \in G$ if $k$ is even,  
$\beta : = \pi \otimes \id_{\Z_2}(\alpha \otimes 1) \in G \otimes \Z_2$ if $k \equiv 1$~mod~$4$, and $\beta:=0$ if $k \equiv 3$~mod~$4$.
\end{enumerate}
\end{Definition}


It is immediate from the Definition \ref{def:boundary_of_an_eqf} that the isomorphism class of $\del(H, \lambda, \alpha)$ is an isomorphism invariant of $(H, \lambda, \alpha)$.  That is,
an isomorphism $B \colon (H_0, \lambda_0, \alpha_0) \cong (H_1, \lambda_1, \alpha_1)$ induces an isomorphism
\[ \del B \colon \del(H_0, \lambda_0, \alpha_0) \cong \del (H_1, \lambda_1, \alpha_1). \]
Hence there is a well defined function
\begin{equation} \label{eq:del}
\del \colon \cal{F}^{4k}_{\rm nd} \to \cal{Q}^{4k-1}, \quad (H, \lambda, \alpha) \mapsto \del(H, \lambda, \alpha) .	
\end{equation}

Recall that $\cal{F}^{4k}_{\rm ns} \subset \cal{F}^{4k}_{\rm nd}$ is the set of isomorphism classes of
nonsingular extended quadratic forms over $\pi_{2k-1}(SO(2k))$.  By definition, 
\[ (H, \lambda, \alpha) \in \cal{F}^{4k}_{\rm ns} \Longleftrightarrow \del(H, \lambda, \alpha) = (0, 0, 0, 0).\]
It follows that if we add a nonsingular form $(H_0, \lambda_0, \alpha_0)$ to a nondegenerate form 
$(H_1, \lambda_1, \alpha_1)$, we do not change the boundary:
\[ \del \bigl( (H_0, \lambda_0, \alpha_0) \oplus (H_1, \lambda_1, \alpha_1) \bigr)  = \del(H_1,\lambda_1, \alpha_1).\]
We shall call two nondegenerate extended quadratic forms $(H_0, \lambda_0, \alpha_0)$ and $(H_1, \lambda_1, \alpha_1)$
{\em stably isomorphic} if they become isomorphic after the addition of nonsingular forms.
It is clear that stably isomorphic forms have isomorphic boundaries.

%
%

\begin{Theorem} \label{thm:boundaries_of_eqfs}
Let $(H_0, \lambda_0, \alpha_0)$ and $(H_1, \lambda_1, \alpha_1) \in \cal{F}^{4k}_{\rm nd}$.
\begin{enumerate}
\item For every isomorphism $A \colon \del(H_0, \lambda_0, \alpha_0) \cong \del(H_1, \lambda_1, \alpha_1)$, there
are nonsingular extended quadratic forms $(H_2, \lambda_2, \alpha_2)$ and $(H_3, \lambda_3, \alpha_3)$ and
an isomorphism
\[ B \colon (H_0, \lambda_0, \alpha_0) \oplus (H_2, \lambda_2, \alpha_2) \cong (H_1, \lambda_1, \alpha_1) \oplus (H_3, \lambda_3, \alpha_3),\]
such that $\del B = A$.
\item $(H_0, \lambda_0, \alpha_0)$ and $(H_1, \lambda_1, \alpha_1)$ are stably isomorphic if and only if
there is an isomorphism $\del (H_0, \lambda_0, \alpha_0) \cong \del (H_1, \lambda_1, \alpha_1)$.

\end{enumerate} 
\end{Theorem}

\begin{proof}
For $k \equiv 3$~mod~$4$ then a nondegenerate extended quadratic form $(H, \lambda, \alpha)$ is just
a nondegenerate even symmetric bilinear form $(H, \lambda)$.  In this case, the theorem is proven by Wall
\cite[Theorem p.\,156]{Wa5}.  In the case where $k \equiv 0$~mod~$2$ this is proved in \cite[Lemma 3.12]{C1}.  
For the remaining case,
$k \equiv 1$~mod~$4$, an extended quadratic form $(H, \lambda, \alpha)$ is an even form $(H, \lambda)$ and a 
homomorphism $\alpha \colon H \to \Z_2$.  The arguments from \cite[Lemma 3.12]{C1} for the case $k \equiv 0$~mod~$2$ 
but $k \neq 2, 4$ can be easily adapted. 
%
\end{proof}

We conclude this section with an important lemma for computing the algebraic 
boundary $\del(H, \lambda, \alpha)$ in concrete cases.  
Recall the short exact sequence \eqref{eq:adjoint}
%
\[ 0 \xra{~~~} H \xra{~\wh \lambda~} H^* \xra{~\pi~} G \to 0. \]
We shall need to determine both the order of the finite group $G$ and the inverse form
$(H^*, \lambda^{-1})$.


\begin{Lemma}\label{lem:inverse_form}
Let $A$ be the (integral) matrix of the adjoint map $\wh \lambda$ with respect to some basis for $H$ and the corresponding dual basis for $H^*$.  
\begin{enumerate}
\item \label{lem:inverse_form:lambda}
The matrix of the 
form $\lambda^{-1}_\Q:H^* \otimes \Q \times H^* \otimes \Q \rightarrow \Q$ with respect to the induced dual basis on
$H^*$ is $A^{-1}.$
\item \label{lem:inverse_form:detA}
$|{\rm det}(A)| = |G|.$
\end{enumerate}
\end{Lemma}

\begin{proof}
\eqref{lem:inverse_form:lambda}
The map $\lambda^{-1}_\Q$ is the composition 
$$(H^* \otimes \Q) \times (H^* \otimes \Q) \xra{(\wh \lambda_\Q)^{-1} \times (\wh \lambda_\Q)^{-1}} 
(H \otimes \Q) \times (H \otimes \Q) \xra{\lambda_\Q} \Q.$$ 
Therefore given elements $v^* \otimes q$ and $w^* \otimes q'$ in $H^* \otimes \Q$ we have 
\begin{align} 
\lambda^{-1}_\Q(v^* \otimes q,w^* \otimes q')&=\lambda_\Q((\wh \lambda_\Q)^{-1}(v^* \otimes q),(\wh \lambda_\Q)^{-1}(w^* \otimes q')) \\
&= \wh \lambda_\Q((\wh \lambda_\Q)^{-1}(v^* \otimes q))((\wh \lambda_\Q)^{-1}(w^* \otimes q')) \\
&= (v^* \otimes q)((\wh \lambda_\Q)^{-1}(w^* \otimes q')).
\end{align}
Expressing $v^* \otimes q$ and $w^* \otimes q'$ with respect to the given basis for $H^* \otimes \Q$, we have $$\lambda^{-1}_\Q(v^* \otimes q,w^* \otimes q')=(v^* \otimes q)(A^{-1}(w^* \otimes q')).$$ Thus the matrix of 
$\lambda^{-1}_\Q$ with respect to the specified basis is precisely $A^{-1}$ as claimed.

\eqref{lem:inverse_form:detA}
Given any $\Z$-linear map $\Z^n \rightarrow \Z^n,$ we can choose bases for each copy of $\Z^n$ so that the corresponding 
matrix is in Smith Normal Form, that is, the matrix is diagonal with each non-zero diagonal entry dividing the next. 
In our case, this amounts to choosing unimodular matrices $P$ and $Q$ such that $PAQ=\hbox{diag}\{a_1,...,a_n\}.$ 
We have $|\det A|=\det (PAQ)=\Pi_{i=1}^n a_i.$
\par
By re-choosing the bases for $H$ and $H^*$ if necessary, we can now re-express the above short exact sequence as
$$0 \lra \Z^n \xra{~\lambda'~} \Z^n \xra{~\pi'~} \Z_{a_1} \times \dots \times \Z_{a_n} \lra 0,$$ 
where $\lambda'$ has matrix 
$\hbox{diag}\{a_1,...,a_n\}$ and $\pi'$ is reduction modulo $a_i$ on the $i^{th}$ factor of $\Z^n$ for each $i.$ 
By exactness and the first isomorphism theorem for groups we have 
$$G \cong \Z^n/\hbox{im}(\lambda') \cong \Z_{a_1} \times \dots \times \Z_{a_n}.$$ 
We therefore have $$|G|=\Pi_{i=1}^n a_i=|\det A|$$ as claimed.
\end{proof}

%


\subsection{Treelike forms} \label{subsec:treelike_forms}
In this subsection we define the set of treelike extended quadratic froms, $\cal{TF}^{4k}$,
which is the algebraic analogue of the set of treelike plumbings $\cal{TP}^{4k}$; see
Lemma \ref{lem:tree_recognition}.

Let $\ul{a} = (a_1, \dots, a_r)$ be an $r$-tuple of integers.  A {\em $\Z$-labelled tree}
$\mathfrak{t} = \bigl( (\V, \E), \ul{a} \bigr)$ is a tree $(\V, \E)$ with ordered vertex set $\V = (v_1, \dots, v_r)$ and
an integer $a_i$ assigned to each vertex $v_i \in \V$.  Since $(\V, \E)$ 
is a tree the set of edges between $v_i$ and $v_j$,
$\E_{i, j}$, is either empty or contains one edge: $|\E_{i, j}| = 0, 1$.
From a $\Z$-labelled tree $\mathfrak{t} = \bigl((\V, \E), \ul{a} \bigr)$
we obtain a symmetric bilinear form $\lambda_{\mathfrak{t}}$ on $H_{\mathfrak{t}}$, 
the free abelian group with basis $\V$, by setting
\begin{equation} \label{eq:plumb}  \lambda_{\mathfrak{t}}(v_i, v_j) 
= \left\{\begin{array}{cl} |\E_{i, j}|& \text{if $i \neq j$,}  \\
a_i \in \Z & \text{if $i = j$,} \end{array} \right. 
\end{equation}
and extending linearly to all of $H_\mathfrak{t}$.
For example if $\ul{a} = \{a_1, a_2, \dots , a_r\}$, the $\Z$-labelled tree
\begin{equation} \label{eq:D(a)}
\xymatrix{ \mathfrak{d}(\ul{a}) : = & a_1 \ar@{-}[r] & a_2 \ar@{-}[r] & \dots \ar@{-}[r] & a_r, }  
\end{equation}
defines the bilinear form $(H_{\mathfrak{d}(\ul{a})}, \lambda_{\mathfrak{d}(\ul{a})})$ with matrix
\begin{equation} \label{eq:AD(a)} 
A_{\mathfrak{d}(\ul{a})} = 
\left( \begin{array}{ccccccc} 
a_1 & 1 & 0 & \dots & 0 & 0\\
1 & a_2 & 1 & \dots & 0 & 0\\
0 & 1 & a_3 & \dots & 0 & 0\\
\vdots & \vdots & \vdots & \ddots & \vdots & \vdots \\
0 & 0 & 0 & \dots & a_{r-1} & 1 \\
0 & 0 & 0 & \dots & 1 & a_r
\end{array} \right)
\end{equation}
with respect to the obvious basis of $H_{\mathfrak{d}(\ul{a})}.$

\begin{Definition} \label{def:treelike_form}
A symmetric bilinear form $(H, \lambda)$ is called treelike if and only if 
\[ (H, \lambda) \cong (H_\mathfrak{t}, \lambda_\mathfrak{t}) \]
for some $\Z$-labelled tree $\mathfrak{t}$.  Such a form will be called `even' if the diagonal 
entries of the corresponding matrix are all even. We define the subset $\cal{TF}^{4k} \subset \cal{F}^{4k}$,
\[  \cal{TF}^{4k} : = \{ (H, \lambda, \alpha)\, | \, \text{$(H, \lambda)$ is treelike} \},  \]
to be the set of extended quadratic forms over $\pi_{2k-1}\{SO(2k)\}$ with treelike symmetric form $(H, \lambda)$.
\end{Definition}

In general, it seems to be a difficult problem to decide it a given symmetric bilinear from $(H, \lambda)$
is treelike, and Lemma \ref{lem:non-treelike} in Section \ref{sec:symmetric_forms_are_stably_treelike}
gives examples of forms which are not treelike.  
Also, we do not know if the set of treelike extended quadratic forms $\cal{TF}^{4k}$ is closed under
direct sum.  However, for the hyperbolic form $H_+(\Z) : = (H_{\mathfrak{d}(0, 0)}, \lambda_{\mathfrak{d}(0, 0)})$ 
with matrix
\begin{equation} \label{eq:hyperbolic_form}
\left( \begin{array}{cc}
0&1 \\
1&0 \\
\end{array}\right),
\end{equation}
if we add the extended form $(H, \lambda, \alpha) = (H_+(\Z), 0)$ to a sum of treelike forms,
we obtain a treelike form.

\begin{Lemma} \label{lem:stable_closure_of_treelike_forms}
For any $j \geq 0$, if $(H_i, \lambda_, \alpha_i) \in \cal{TF}^{4k}$, then
$\bigoplus_{i=1}^j (H_i, \lambda_i, \alpha_i) \oplus (H_+(\Z), 0) \in \cal{TF}^{4k}$.  
\end{Lemma}

%
%

\begin{proof}
The property of being treelike only depends on the underlying symmetric form.  
The following sequence of diagrams illustrates the basis change needed when $j = 2$.
We start from the disjoint union of two $\Z$-labelled trees:
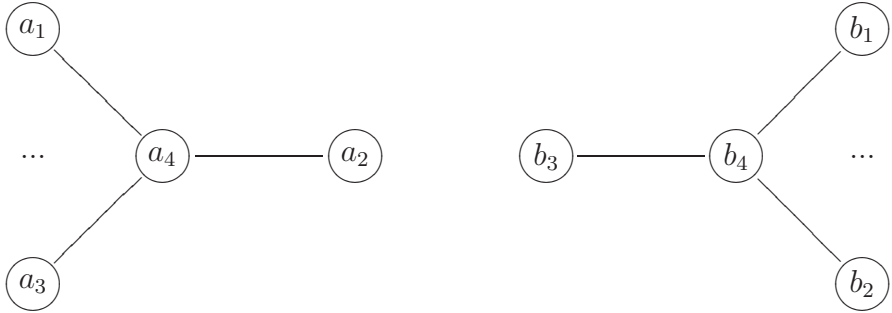
\[ 
\entrymodifiers={++[o][F-]}
\xymatrix{ a_1 \ar@{-}[dr] & *\txt{} & *\txt{} & *\txt{} & *\txt{} & *\txt{} & *\txt{} & *\txt{} & b_1 \\
*\txt{...} & a_4 \ar@{-}[rr] & *\txt{} & a_2 & *\txt{} & b_3 \ar@{-}[rr] & *\txt{} &  b_4 \ar@{-}[ur] \ar@{-}[dr] & *\txt{...}\\ 
a_3 \ar@{-}[ur] & *\txt{} & *\txt{} & *\txt{} & *\txt{} & *\txt{} & *\txt{} & *\txt{} & b_2 }  \]
Then we introduce a disjoint ``hyperbolic component'' with labels $(0, 0)$:
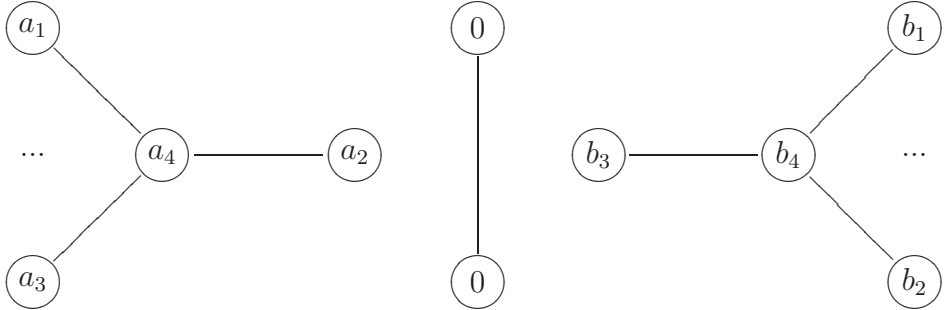
\[
\entrymodifiers={++[o][F-]}
\xymatrix{ a_1 \ar@{-}[dr] & *\txt{} & *\txt{} & *\txt{} & 0 \ar@{-}[dd] & *\txt{} & *\txt{} & *\txt{} & b_1 \\
*\txt{...} & a_4 \ar@{-}[rr] & *\txt{} & a_2 & *\txt{} & b_3 \ar@{-}[rr] & *\txt{} &  b_4 \ar@{-}[ur] \ar@{-}[dr] & *\txt{...}\\ 
a_3 \ar@{-}[ur] & *\txt{} & *\txt{} & *\txt{} & 0 & *\txt{} & *\txt{} & *\txt{} & b_2 }  \]
If the labels $(a_1, a_2, a_3, a_4), (b_1, b_2, b_3, b_4)$ and $(0, 0)$ correspond
to bases $\{x_1, x_2, x_3, x_4\}$, $\{y_1, y_2, y_3, y_4\}$ and $\{z_1, z_2\}$, we introduce the basis change
\[ x_2 \longmapsto x_2 + z_1, \quad y_3 \longmapsto y_3 + z_1, \]
leaving all other basis elements fixed.  The corresponding graph is the following tree:
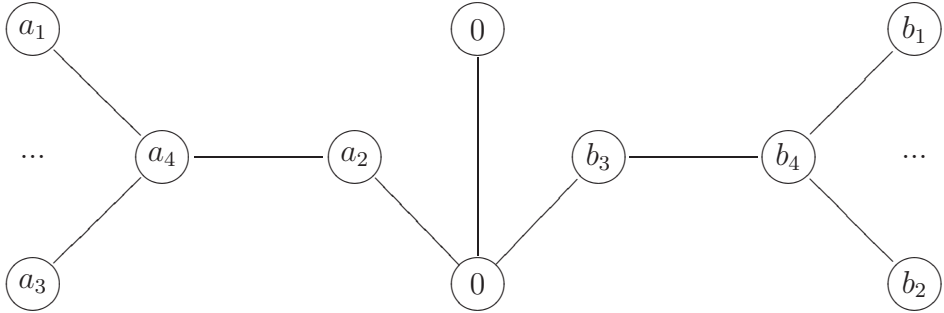
\[
\entrymodifiers={++[o][F-]}
\xymatrix{ a_1 \ar@{-}[dr] & *\txt{} & *\txt{} & *\txt{} & 0 \ar@{-}[dd] & *\txt{} & *\txt{} & *\txt{} & b_1 \\
*\txt{...} & a_4 \ar@{-}[rr] & *\txt{} & a_2 \ar@{-}[dr] & *\txt{} & b_3 \ar@{-}[rr] & *\txt{} &  b_4 \ar@{-}[ur] \ar@{-}[dr] & *\txt{...}\\ 
a_3 \ar@{-}[ur] & *\txt{} & *\txt{} & *\txt{} & 0 \ar@{-}[ur] & *\txt{} & *\txt{} & *\txt{} & b_2 }  \]

In the general case, we simply choose a basis element $x_i \in H_i$ and make the basis change
$x_i \mapsto x_i + z_1$.  The outcome is then a treelike basis for $\bigoplus_{i=1}^j(H_i, \lambda_i) \oplus H_+(\Z)$.
\end{proof}
%


\section{Handlebodies} \label{sec:handlebodies}
In this section we recall Wall's classifiction of $(n-1)$-connected $2n$-manifolds with
simply connected boundary, $n \geq 3$, by their extended intersection froms.
We then show that treelike plumbings in dimension $4k$ are precisely those handlebodies with treelike
extended intersection forms.  This allows us to give a second proof of a special
case of Theorem \ref{thm:connected_sum_for_delTP-G}.

\subsection{The classification of handlebodies} \label{subsec:classification_of_handlebodies}
Recall that Smale, \cite[\S 1]{Sm} identified $\cal{H}(n)$, the set 
diffeomorphism classes manifolds $W$ which are obtained by attaching a finite number of $n$-handles to $D^{2n}$.
That is, we identify $D^n \times S^{n-1} = D^n \times \del D^n$ as a codimension-$0$ submanifold of
the boundary of an $n$-handle $D^n \times D^n$, we take an embedding $\phi \colon \sqcup_{i=1}^r D^n \times S^{n-1} \to S^{2n-1}$
and then form $W$ by attaching $r$ $n$-handles $D^n \times D^n$ to $D^{2n}$ along the embedding $\phi$:
\[ W : = D^{2n} \cup_\phi (\sqcup_{i=1}^r D^n \times D^n) .\]
Such manifolds are examples of handlebodies.  
Since we shall work in the oriented category, {\em henceforth all manifolds are assumed to be oriented
and all diffeomorphisms are assumed orientation preserving}. Modifying Smale's notation we let
\[ \cal{H}_{}^{2n}  = \{ W | W \cong  D^{2n} \cup_\phi (\sqcup_{i=1}^r D^n \times D^n) \} \]
denote the set of (oriented) diffeomorphism classes of (oriented) handlebodies obtained from $D^{2n}$
by attaching a finite number of $n$-handles. Up to homotopy, adding an $n$-handle to a manifold with boundary is equivalent to adding an $n$-cell. Thus manifolds in $\cal{H}_{}^{2n}$ have the homotopy type of a wedge of $n$-spheres. In particular this means that $H_n(W)$ is free and $H_i(W)=0$ for any $0<i<n.$

\begin{Theorem}[{\cite[Theorem 1.2]{Sm}}] \label{thm:Smale}
Let $n \geq 3$ and let $W^{2n}$ be an $(n-1)$-connected manifold with $(n-2)$-connected boundary.  Then $W \in \cal{H}^{2n}$.
\end{Theorem}

\begin{Remark} \label{rem:Smale}
Notice that if $n \geq 3$ and the boundary of $W$, $\del W$, in Theorem \ref{thm:Smale} is simply connected, then
$\del W$ is indeed $(n-2)$-connected.  This is because Poincar\'{e} duality for $W$ shows that $\wt H_*(\del W) = 0$
for $* \leq n-2$, from which the connectedness of the boundary follows by the Hurewicz theorem. 
\end{Remark}

For $n \geq 3$, the classification the manifolds $W \in \cal{H}^{2n}$ was given by Wall \cite{Wa1}.
We present an equivalent, cohomological version of Wall's classification: see Baues \cite{Bau1} for a similar discussion
using homology.
The {\em intersection form} of $W$ is a bilinear pairing
\[ \lambda_W \colon H^{n}(W, \del W) \times H^{n}(W, \del W) \rightarrow \Z, \] 
which can be algebraically defined using the the connecting homomorphism
$j \colon H^{n}(W, \del W) \to H_n(W)$ coming from the homology long exact sequence of the pair $(W,\del W)$ together with the Poincar\'{e} duality isomorphism
$\PD \colon H^{n}(W) \cong H_{n}(W, \del)$.  We have
\begin{equation} \label{eq:lambda}
\lambda(x, y) := \an{\PD(j(x)), y}. 
\end{equation}
Geometrically, the intersection form can be computed via the intersection of oriented submanifolds which
are Poincar\'{e} duals to classes $x_1, x_2 \in H^n(W, \del W)$.
Suppose that $\wh x_1$ and $\wh x_2$ are $n$-dimensional homology classes which are dual to $x_1$ and
$x_2$ and which can be represented by compact oriented submanifolds $N_1$ and 
$N_2.$  Without loss of generality, suppose that $N_1$ and $N_2$ are in general position with respect to each other. 
This means that the submanifolds interesect transversally in a finite collection of points. There is a notion of sign 
which can be associated to such an intersection, depending on the orientations of the submanifolds and the global orientation
of $W$. It can be shown that the cohomological intersection of $x_1$ and $x_2$ defined in \eqref{eq:lambda} above
is equal to the number of (geometric) intersection points of $N_1$ and 
$N_2$ counted with sign. For more details see \cite[V \S 1]{Bro}. 

In addition to the intersection form of $W$, Wall identifies an extra invariant
\[ \mu_W \colon H^n(W, \del W) \to \pi_{n-1}(SO(n)), \quad x \mapsto \nu_{\wh x}. \]
Here $\wh x \in H_n(W)$ is Poincar\'{e} dual to $x$ and 
and $\nu_{\wh x}$ is the isomorphism class of the normal bundle of an embedding
\[ f_{\wh x} \colon S^n \to W \]
representing $\wh x$: as in Section \ref{subsec:eqfs}, we identify the homotopy group $\pi_{n-1}(SO(n))$
with the set of isomorphism classes of oriented rank $n$ vector bundles over $S^n$.
By the Hurewicz theorem $\pi_n(W) \cong H_n(W)$ and by Haefliger's classification of embeddings \cite[Th\'{e}or\`{e}m d'approximation]{H},
the homology class $\wh x$ is represented by an embedding $f_{\wh x}$ which is unique up to isotopy: hence the isomorphism class 
of the bundle $\nu_ {\wh x}$ is well-defined.  Wall \cite[Lemma 2]{Wa1} shows that the function $\mu_W$ is a 
quadratic refinement of $\lambda_W$ in the sense that 
together $\mu_W$ and $\lambda_W$ satisfy equations \eqref{eq:mu_and_p} and~\eqref{eq:mu_and_h} from 
Section \ref{subsec:eqfs} above.  
Hence the triple $(H^n(W, \del W), \lambda_W, \mu_W)$ is an extended quadratic form over $\pi_{n-1}\{SO(n)\}$,
called the {\em extended intersection form of $W$}.

We now assume that $n = 2k$ is even.
Recall from Section \ref{subsec:eqfs} that the stabilisation of $\mu_W$ is a homomorphism 
$\alpha_W = S\mu_W \colon H_{2k}(W) \to \pi_{2k-1}(SO)$ and that 
the pair $(\lambda_W, \alpha_W)$ determines the pair $(\lambda_W, \mu_W)$.
Hence when $n = 2k$, we shall use the equivalent triple
\[ (H^{2k}(W, \del W), \lambda_W, \alpha_W)  \]
for the extended intersection form of $W$.
Wall's arguments in \cite{Wa1} prove that a diffeomorphism of handlebodies
$f \colon W_0 \cong W_1$ induces an isomorphism of their extended quadratic forms: 
%
\[ f^* \colon (H^n(W_1, \del W_1), \lambda_{W_1}, \alpha_{W_1}) \cong (H^n(W_0, \del W_0), \lambda_{W_0}, \alpha_{W_0}). \]
We now state the cohomological version of Wall's classification of handlebodies in
the case $n = 2k$ is even.
%

%
%
%
%

\begin{Theorem}[{\cite[p.\,168]{Wa1}}] \label{thm:wall_2n_classification}
For all $k \geq 2$,
the assignment of its extended intersection form to a handlebody defines a bijection,
\[  \cal{H}^{4k} \equiv \cal{F}^{4k}, \quad W \mapsto (H^{2k}(W, \del W), \lambda_W, \alpha_W), \]
which maps the boundary connected sum of handlebodies to the orthogonal sum of forms;
\[ W_0 \natural W_1 \mapsto (H^{2k}(W_0, \del W_0), \lambda_{W_0}, \alpha_{W_0}) \oplus (H^{2k}(W_1, \del W_1), \lambda_{W_1}, \alpha_{W_1}).\]
Moreover, every isomorphism of extended intersection forms,
\[ A \colon (H^{2k}(W_1, \del W_1), \lambda_{W_1}, \alpha_{W_1}) \cong (H^{2k}(W_0, \del W_0), \lambda_{W_0}, \alpha_{W_0}), \]
is realised by a diffeomorphism $f_A \colon W_0 \cong W_1$.
\end{Theorem}

\subsection{Treelike plumbings} \label{subsec:treelike_plumbings}
%

In Section \ref{subsec:plumbing_arrangements}, we introduced the set of diffeomorphism classes of treelike 
plumbing manifolds $\cal{TP}^{4k}$ and the set of diffeomorphism classes of boundaries of treelike
plumbings
\[ \del\cal{TP}^{4k} = \{\del W | W \in \cal{TP}^{4k-1} \}. \]
Since each treelike plumbing $W$ is $(2k{-}1)$-connected we see that $\cal{TP}^{4k} \subset \cal{H}^{4k}$
and Wall's Theorem \ref{thm:wall_2n_classification} applies.  
In Section \ref{subsec:treelike_forms} we introduced the set
treelike extended quadratic forms $\cal{TF}^{4k}$.  
In Lemma \ref{lem:tree_recognition} below we prove that a handlebody $W \in \cal{H}^{4k}$ is a treelike plumbing
if an only if its extended intersection form is treelike. 
We then use this fact to
prove the following special case of Theorem \ref{thm:connected_sum_for_delTP-G}.

\begin{Theorem} \label{thm:TP_and_connected_sum}
If $M_0, M_1 \in \del \cal{TP}^{4k}$ then $M_0 \sharp M_1 \in \del \cal{TP}^{4k}$.
\end{Theorem}


\begin{Lemma} \label{lem:tree_recognition}
If $W \in \cal{H}^{4k}$, then
$W \in \cal{TP}^{4k}$ if and only if $(H^{2k}(W, \del W), \lambda_W, \alpha_W) \in \cal{TF}^{4k}$.
\end{Lemma}

\begin{proof}
If $(H^{2k}(W, \del W), \lambda_W, \alpha_W) \in \cal{TF}^{4k}$, we choose a $\Z$-labelled tree 
$\mathfrak{t} = ((\V, \E), (a_i))$ with vertex set $V = (v_1, \dots, v_r)$ and an isomorphism
\[  \theta \colon (H_\mathfrak{t}, \lambda_\mathfrak{t}) \cong (H^{2k}(W, \del W), \lambda_W).    \]
We define the bundles $\alpha_i := \mu_W(\theta(v_i)) \in \pi_{2k-1}(SO(2k))$, and consider
the plumbing arrangement determined by the labelled tree $\mathfrak{u} : = ((\V, \E), (\alpha_i))$
and form the plumbing manifold $W(\mathfrak{u})$.  By construction, the extended intersection
forms of $W$ and $W(\mathfrak{u})$ are isomorphic.  By Theorem \ref{thm:wall_2n_classification}
there is a diffeomorphism $W \cong W(\mathfrak{g})$ and so $W \in \cal{TP}^{4k}$.

%
%
%

Conversely, if $W \in \cal{TP}^{4k}$ then $W \cong W(\mathfrak{u})$, $\mathfrak{u} = ((\V, \E), (\alpha_i))$, 
is realised via a plumbing arrangement with a treelike graph $(\V, \E)$. 
By the discussion prior to Theorem \ref{thm:plumbing_realisation}, the intersection form of $W(\mathfrak{u})$ is treelike, 
and so by definition, the extended quadratic form of $W$ is treelike.
\end{proof}

Now consider the handlebody
\[W_\infty := (S^{2k} \times S^{2k} - \textup{int}(D^{4k})) \]
which has extended intersection form $(H_+(\Z), 0)$.

\begin{Proposition} \label{prop:stable_closure_of_treelike_plumbings}
Suppose that $W_0, W_1 \in \cal{TP}^{4k}$.  Then
$W_0 \natural W_\infty \natural W_1 \in \cal{TP}^{4k}$.
\end{Proposition}

\begin{proof}
By Lemma~\ref{lem:stable_closure_of_treelike_forms}, the extended intersection form
of $W_0 \natural W_\infty \natural W_1$ is treelike.  Hence by Lemma 
\ref{lem:tree_recognition}, $W_0 \natural W_\infty \natural W_1$ is treelike.
\end{proof}

\begin{proof}[Proof of Theorem \ref{thm:TP_and_connected_sum}]
Suppose that $M_i = \del W_i$ for $W_i \in \cal{TP}^{4k}$, $i = 0, 1$.
Then
\[ M_0 \sharp M_1 = M_0 \sharp S^{4k-1} \sharp M_1 = \del(W_0 \natural W_\infty \natural W_1). \]
By Lemma~\ref{prop:stable_closure_of_treelike_plumbings}, $W_0 \natural W_\infty \natural W_1 \in \cal{TP}^{4k}$
and so $M_0 \sharp M_1 \in \del \cal{TP}^{4k-1}$.
\end{proof}



\section{$(2k{-}1)$-parallelisable $(2k{-}2)$-connected $(4k{-}1)$-manifolds} 
\label{sec:(4k-1)_manifolds}
We define 
\[ \cal{M}_{2n-1, n-2}^{O\an{n-1}} := \{ M^{2n-1} \,|\, \text{$M$ is $(n-2)$-connected and $(n-1)$-parallelisable} \}, \]
to be the set of diffeomorphism classes of closed $(n-2)$-connected and $(n-1)$-parallelisable manifolds of dimension $2n-1$.
For questions about positive Ricci curvature, an important subset of $\cal{M}_{2n-1, n-2}^{O\an{n-1}}$ 
is made up of those manifolds $M$ which are the boundaries of handlebodies.
Hence we define 
%
\[ \del \cal{H}^{2n} := \{ M \, | \, \text{$M = \del W$ and $W \in \cal{H}^{2n}$} \} \subset \cal{M}_{2n-1, n-2}^{O\an{n-1}}. \]
Recall also that $\del \cal{TP}^{2n}$ is the set of diffeomorphism classes of manifolds which 
are the boundaries of treelike plumbings.
Evidently we have inclusions
\[ \del \cal{TP}^{2n} \subset \del \cal{H}^{2n} \subset \cal{M}_{2n-1, n-2}^{O\an{n-1}}.\]

Another important subset of $\cal{M}_{2n-1, n-2}^{O\an{n-1}}$ is $\Theta_{2n-1}$, the set 
of oriented diffeomorphism classes of manifolds $\Sigma$ which are homootpy equivelant to $S^{2n+1}$:
\[ \Theta_{2n-1} := \{ \Sigma \, | \, \Sigma \simeq S^{2n-1} \} .\]
The set $\Theta_{2n-1}$ forms a group under connected sum which was famously first studied by
Kervaire and Milnor \cite{Kervaire&Milnor}.  We now briefly recall the structure of the group
$\Theta_{2n-1}$ identified by Kervaire and Milnor and an imporant addition of Brumfiel when $n = 2k$.
Let $\Omega_{m}^{\rm fr}$ denote the framed bordism group of stably framed $m$-manifolds and let
\[ J_m \colon \pi_m(SO) \to \Omega_{m}^{\rm fr}, \quad \gamma \mapsto [S^m, \gamma \cdot F_0], \]
be the homomorphism obtained by using a homotopy class $\gamma \colon S^m \to SO$ to reframe
the standard, nullbordant framing $F_0$, of $S^m$, and taking the framed bordism class of
the resulting manifold.  The Pontrjagin-Thom isomorphism identifies $\Omega_m^{\rm fr}$,
with $\pi_m^S$, the stable $m$-step of homotopy classes of maps between spheres of dimensions differing by $m$
\cite{B-tD}[Satz 3.1 and 4.9].
By a Theorem of Serre  \cite{Se}, $\pi_m^S$ is finite if $m > 0$ and hence the cokernel of the $J$-homomoprhism,
${\rm Coker}(J_m)$ is finite.  Kevaire and Milnor \cite[Theorem 3.1]{Kervaire&Milnor} proved
that every homotopy sphere $\Sigma$ is stably parallelisable and then defined the homomorphism
\[ \eta \colon \Theta_m \to {\rm Coker}(J_m), \quad \Sigma \mapsto [\Sigma, F], \]
by choosing a stable framing $F$ for a homotopy sphere $\Sigma$, taking the framed bordism class of $(\Sigma, F)$,
and then modding out by the possible choices of stable framing.  

By definition, the kernel of $\eta$ is $bP_{m+1} \subset \Theta_{m}$,
the subgroup of homotopy spheres which bound parallelisable manifolds, and Kervaire and Milnor
proved that $bP_{m+1}$ is a finite cyclic group.
When $m = 4k-1$, Brumfiel \cite{Brumfiel} gave a splitting of the inclusion $bP_{4k} \subset \Theta_{4k-1}$ by showing
that every homotopy sphere $\Sigma \in \Theta_{4k-1}$ bounds a spin manifold $W$ with vanishing 
decomposable Pontrjagin numbers and signature $\sigma(W)$ divisible by $8$ \cite[Theorems 1.5 and 1.6]{Brumfiel}.  Brumfiel then defined
the homomorphism
\[ f \colon \Theta_{4k-1} \to bP_{4k}, \quad f(\Sigma) : = \frac{\sigma(W)}{8} \in \Z/|bP_{4k}| \cong bP_{4k}.\]

\begin{Theorem}[{\cite[Theorem 6.6]{Kervaire&Milnor}}, {\cite[Theorem 1.4]{Brumfiel}}] \label{thm:homotopy_spheres}
For all $2n-1 \geq 5$, there is a short exact sequence
\[ 0 \to bP_{2n} \to \Theta_{2n-1} \to {\rm Coker}(J_{2n-1}) \to 0.\]
When $n = 2k$ is even, this sequence is split by the Brumfiel invariant
$f \colon \Theta_{4k-1} \to bP_{4k}$.
\end{Theorem}

Now recall $BO\an{n} \to BO$, the $(n-1)$-fold connected covering of $BO$. 
In Section \ref{subsec:BO<n>-bordism} below, we define the bordism groups
of $\Omega_*^{O\an{n-1}}$ of $BO\an{n}$-manifolds and 
Lemma~\ref{lem:BO<n>-orientation}~\eqref{lem:BOan<n>-orientation:structure} shows that every 
$M \in \cal{M}^{O\an{n-1}}_{{2n-1},{n-2}}$ defines a unique bordism class 
$[M] \in \Omega_*^{O\an{n-1}}$.
When $n=2k \geq 4$, we have following theorem, 
the third part of which is the main topological result of this paper.

\begin{Theorem} \label{thm:delTP}
Let $k \geq 2$ and $M \in \cal{M}_{4k-1,2k-2}^{O\an{2k-1}}$.
\begin{enumerate}
\item \label{thm:delTP:bordism}
$M \in \del \cal{H}^{4k}$ if and only if $[M] = 0 \in \Omega_{4k-1}^{O\an{2k-1}}$.
\item \label{thm:delTP:homotopy_sphere}
There is a homotopy sphere $\Sigma$ with $f(\Sigma)= 0$ such that $M \sharp \Sigma \in \del \cal{H}^{4k}$. 
\item \label{thm:delTP:main}
If $M \in \del \cal{H}^{4k}$, then $M \in \del \cal{TP}^{4k}$.  Consequently
$\del \cal{H}^{4k} = \del \cal{TP}^{4k}$.
\end{enumerate}
\end{Theorem}

\noindent
Theorem \ref{thm:delTP} plays a central role in the proofs of our main results:
Theorem C is an equivalent statement of Theorem~\ref{thm:delTP}~\eqref{thm:delTP:main}, 
and Theorems A and A$'$ both follow by combining Theorem 1.1 with Theorem \ref{thm:delTP}.  
Moreover, Theorem B follows from Theorem 1.1 when combined with Corolloary \ref{cor:7-11} below.

\begin{Remark}
Part (2) of Theorem \ref{thm:delTP}, without the extra condition $f(\Sigma) = 0$, was 
proven in \cite[Theorem 8]{Wa4}, which was written before Brumfiel's paper \cite{Brumfiel}.
\end{Remark}

In \cite{Kervaire&Milnor} it is shown that $\Theta_7 \cong bP_8$ and $\Theta_{11} \cong bP_{12}.$ The isomorphisms are given by the Brumfiel invariant $f$. 
Since all 2-connected 7-manifolds (respectively 4-connected 11-manifolds) are automatically spin (respectively string) as a consequence of their connectedness, we obtain the following corollary to Theorem~\ref{thm:delTP} (2) and (3).

\begin{Corollary}\label{cor:7-11}
$\cal{M}_{7,2}^{Spin} = \del \cal{TP}^8$ and $\cal{M}_{11, 4}^{String} = \del \cal{TP}^{12}$. 
That is, every 2-connected 7-manifold is diffeomorphic to the boundary of
a treelike plumbing $W \in \cal{TP}^{8}$, 
and every 4-connected 11-manifold is diffeomorphic to the boundary of
a treelike plumbing $W \in \cal{TP}^{12}$.
%
\end{Corollary}

The proofs of parts (1) and (2) of Theorem~\ref{thm:delTP} are given in Section~\ref{subsec:BO<n>-bordism}.
The proof of part (3) consists of two parts and runs as follows.  In Section \ref{subsec:rational_homotopy_spheres} 
we present the classification of rational homotopy spheres
$M \in \del \cal{H}^{4k}$ up to connected sum with homotopy spheres and we reduce
Theorem~\ref{thm:delTP}~\eqref{thm:delTP:main} to the completely algebraic Theorem \ref{thm:realisation_of_eqlfs}.
We prove Theorem \ref{thm:realisation_of_eqlfs} in Section~\ref{sec:realising_extended_quadratic_linking_forms}.  

\subsection{$BO\an{n}$-bordism} \label{subsec:BO<n>-bordism}

In this subsection we prove Theorem \ref{thm:delTP}~\eqref{thm:delTP:bordism} and~\eqref{thm:delTP:homotopy_sphere}.
%
We begin by defining the bordism groups $\Omega_*^{O\an{n-1}}$.  Recall that a {\em $BO\an{n}$-structure}
on a compact manifold $X$ is an equivalence class of diagram
\[ \xymatrix{ & BO\an{n} \ar[d] \\
X \ar[ur]^{\bar \nu} \ar[r]^{\nu} & BO} \]
where $\nu \colon X \to BO$ classifies the stable normal bundle of $X$.  
Two closed $BO\an{n}$-manifolds
$(N_0, \bar \nu_0)$ and $(N_1, \bar \nu_1)$ are $BO\an{n}$-bordant if there there is a compact
$BO\an{n}$-manifold $(X, \bar \nu_X)$ with boundary the disjoint union $(N_0, \bar \nu_0)$ and
$(N_1, -\bar \nu_1)$, where $-\bar \nu_1$ is the $BO\an{n}$-structure induced on 
$N_1 \times \{0\} \subset N_1 \times I$
from a $BO\an{n}$-structure on $N_1 \times I$ which restricts to $\bar \nu_1$ on $N_1 \times \{1\}$.  
We have the bordism group of closed $m$-dimensional $BO\an{n}$-manifolds
\[ \Omega_m^{O\an{n}} : = \{ (N, \bar \nu) \,|\, \text{$(N, \bar \nu)$ is a closed $BO\an{n}$-manifold}\}/\text{$BO\an{n}$-bordism,} \]
where addition is given by disjoint union and $-[N, \bar \nu] = [N, -\bar \nu]$.
For more details about the definition of the $BO\an{n}$-bordism groups,
we refer the reader to \cite[Chapter II]{Stong}.

For the statement of the following lemma, 
we recall that a manifold $X$ is called $(n-1)$-parallelisable if its tangent bundle
is trivial when restricted to every $(n-1)$-skeleton of $X$.

\begin{Lemma} \label{lem:BO<n>-orientation}
\begin{enumerate}
\item \label{lem:BOan<n>-orientation:existence}
A manifold $X$ admits a $BO\an{n}$-structure $\bar \nu \colon X \to BO\an{n}$
if and only if $X$ is $(n-1)$-parallelisable.
\item \label{lem:BOan<n>-orientation:handlebody}
Every manifold $M \in \del \cal{H}^{2n}$ is $(n-1)$-parallelisable.
\item \label{lem:BOan<n>-orientation:structure}
Every connected $(n-2)$-connected $(n-1)$-parallelisable manifold $X$ admits
precisely two equivalence classes of $BO\an{n}$-orientation which are determined by
the orientation they induce on $X$. \qedhere
\end{enumerate}
\end{Lemma}

\begin{proof}
\eqref{lem:BOan<n>-orientation:existence}  Let $X^{(n-1)}$ be an $n$-skeleton for $X$.
Since $BO\an{n}$ is $(n-1)$-connected, if $X$ admits a $BO\an{n}$-structure $\bar \nu$,
then $\bar \nu|_{X^{(n-1)}}$ is null-homotopic, hence $\nu|_{X^{(n-1)}}$ is null-homotopic. From this we deduce that the tangent bundle of $X$ restricted to $X^{(n-1)}$ is stably trivial, since $\nu|_{X^{(n-1)}}$ is a $KO$-theory inverse for the stable tangent bundle over $X^{(n-1)}.$ Lemma 3.5 of \cite{Kervaire&Milnor} states that a vector bundle over a complex where the fibre dimension exceeds the base dimension is trivial if and only if it is stably trivial. As this is precisely the situation for $TX|_{X^{(n-1)}}$ we conclude that $X$ is $(n-1)$-parallelisable.
  
To argue in the other direction, we note that by definition, there is a fibration
$P^{n-2}(O) \to BO\an{n} \to BO$ where $P^{n-2}(O)$ is the $(n-2)^{nd}$ Postnikov
stage of $O$: see \cite[IX Theorem 2.8]{Wh} for the definition of $P^{n-2}(O)$.
The obstructions to lifting $\nu \colon X \to BO\an{n}$ lie in the group
$H^*(X; \pi_{*-1}(P^{n-2}(O))$ \cite[VI Theorem 6.2]{Wh}, which vanishes for $* > n-2$.  Since 
the restriction homomorphism $H^*(X; \pi_{*-1}(P^{n-2}(O)) \to H^*(X^{(n-1)}; \pi_{*-1}(P^{n-2}(O))$
is an isomorphism, we see that these obstructions vanish if $X$ is $(n-1)$-parallelisable.

\eqref{lem:BOan<n>-orientation:handlebody}  Let $M = \del W$, for $W \in \cal{H}^{2n}$.
Since $W$ is $(n-1)$-connected, it follows that $W$ is $(n-1)$-parallelisable.
Since the classifying map of the stable normal bundle of $M$, $\nu \colon M \to BO$
factors through the inclusion $M \to W$, it follows that $M$ is $(n-1)$-parallelisable.

\eqref{lem:BOan<n>-orientation:structure}  Let $\nu^+ \colon X \to BSO$ be an orientation
of the stable normal bundle of $X$ and let
$\bar \nu_i$, $i=0, 1$ be two 
$BO\an{n}$-structures on $X$ compatible with $\nu^+$.  We now
work with the $(n-1)$-connective covering over $BSO$, $BO\an{n} \to BSO$,
and consider maps $f_{\bar \nu_i} \colon X \to BO\an{n}$, $i = 0, 1$,
covering $\nu^+$, representing the two $BO\an{n}$-structures $\nu_0$ and $\nu_1$ on $X$.  
We may assume that both $f_{\bar \nu_0}$ and $f_{\bar \nu_1}$ are lifts of
the same map $f_{\nu^+} \colon X \to BSO$, representing $\nu^+$.
The obstructions to finding a vertical homotopy over $BO\an{n} \to BSO$ between $f_{\bar \nu_0}$ and 
$f_{\bar \nu_1}$ lie in the groups $H^*(X; \pi_{*}(P^{n-2}(SO))$, see \cite[VI Theorem 6.12]{Wh}, 
and these groups vanish since $P^{n-2}(SO)$ is connected 
and $X$ is $(n-2)$-connected. (Note that a vertical homotopy is a homotopy between maps into the total space of a fibration which projects to the same map on the base independent of the homotopy parameter.) 
It follows that $\bar \nu_0$ and $\bar \nu_1$ are equivalent $BO\an{n}$-structures on $X$.
\end{proof}

\begin{proof}[Proof of Theorem~\ref{thm:delTP}~\eqref{thm:delTP:bordism} and~\eqref{thm:delTP:homotopy_sphere}]
\eqref{thm:delTP:bordism}  One direction is clear: if $M \in \del \cal{H}^{2n}$, then by definition $M =\del W$ where
$W \in \cal{H}^{2n}$ and so by Lemma~\ref{lem:BO<n>-orientation} $W$ is a $BO\an{n}$-manifold, i.e. a $BO\an{n}$ null-bordism of $M$. Thus $[M]=0 \in \Omega^{O\an{2k-1}}_{4k-1}.$
%
%
Conversely, suppose that $W_0$ is a $BO\an{n}$-null bordism of $M$.
Then by \cite[Theorem 3]{Mi}, we may perform surgeries on the interior of $W_0$ to 
obtain an $(n-1)$-connected manifold $W$ with boundary $M$.  By Theorem \ref{thm:Smale},
$W \in \cal{H}^{2n}$ and hence $M \in \del \cal{H}^{2n}$.
%

\noindent
\eqref{thm:delTP:homotopy_sphere}  
Let $M \in \cal{M}_{4k-1}^{O\an{2k-1}}$. It suffices to show that $M$ is bordant over $BO\an{2k}$ to
some homotopy sphere $\Sigma$ with $s(\Sigma) = 0$, for then $[M \sharp (-\Sigma)] = 0 \in \Omega_{4k-1}^{O\an{2k-1}}$, and by part \eqref{thm:delTP:bordism}
$M \sharp (-\Sigma) \in \del \cal{H}^{4k}$.  Moreover, $s(-\Sigma) = 0$.
Let $\bar{\nu} \colon M \to BO\an{2k}$ be a lift of the classifying map
of the stable normal bundle of $M$.  We consider the problem of doing surgery on the normal 
map $\bar \nu \colon M \to BO\an{2k}$ until it is an isomorphism on all homotopy groups $\pi_i$ for $i \leq 2k-1$.
We see that the arguments of Kervaire and Milnor \cite{Kervaire&Milnor} can be used in this situation
and that there is no obstruction.  Hence $M$ is bordant over $BO\an{2k}$ to a homotopy sphere 
$\Sigma_0 \to BO\an{2k}$.  
Now there is a homotopy sphere $\Sigma_1 \in bP_{4k}$ such that $s(\Sigma_1) = -s(\Sigma_0)$.  
We set $\Sigma_2 : = \Sigma_1 \sharp \Sigma_0$, so that $s(\Sigma_2)= 0$.
Since $\Sigma_1$ bounds a parallelisable manifold, it bounds over $BO\an{2k}$,
and so
\[ [M] = [\Sigma_0] = [\Sigma_0 \sharp \Sigma_1] = [\Sigma_2] \in \Omega_{4k-1}^{O\an{2k-1}}.\]
Thus choosing a homotopy sphere $\Sigma$ so that $[\Sigma]=-[\Sigma_2] \in \Omega_{4k-1}^{O\an{2k-1}}$ means that $s(\Sigma)=0$ and $[M \sharp \Sigma]=0 \in \Omega_{4k-1}^{O\an{2k-1}}.$ Thus by (1) above, $M \sharp \Sigma \in \del \cal{H}^{2n}$.
\end{proof}


\subsection{Rational homotopy spheres} 
\label{subsec:rational_homotopy_spheres}


Let $\del \cal{H}^{4k}_{\QHS} \subset \del \cal{H}^{4k}$ be the set of diffeomorphism classes 
of manifolds $M$ which are the boundaries of handlebodies and also rational homology spheres:
\[ \del \cal{H}^{4k}_{\QHS} = \{ M  \in \del \cal{H}^{4k} 
\, | \, H^*(M; \Q) \cong H^*(S^{4k-1}; \Q) \}. \]
%
When $k = 1$, $\cal{H}^{4k}_{\QHS}$ is a set of rational {\em homotopy spheres} 
but we are mostly interested in the case $k \geq 2$ when
$\cal{H}^{4k}_{\QHS}$ is a set of rational {\em homology spheres}.
So for the remainder of this section {\em we assume $k \geq 2$.}

In this subsection we define the extended quadratic linking form of $M \in \del \cal{H}^{4k}_{\QHS}$,
which is closely related to the usual linking form of $M$; see \eqref{eq:linking_form} below.
We then use the extended quadratic linking form to give the diffeomorphism classification of manifolds in $\del \cal{H}^{4k}_{\QHS}$ up to connected sum with homotopy spheres.

We begin by recalling the definition of the linking form of a closed $(2n-1)$-manifold $N$.
Let $TH^{n}(N)$ denote the torsion subgroup of $H^{n}(N)$ and let $x, y \in TH^n(N)$.
Since $x$ maps to zero in $H^n(N; \Q)$ it has a preimage $\ol x \in H^{n-1}(N; \Q/\Z)$
under the Bockstein for the coefficient sequence $\Z \to \Q \to \Q/\Z$.  We define
\begin{equation} \label{eq:linking_form} 
b_N :TH^{n}(N) \times TH^{n}(N) \rightarrow \Q/\Z, \quad (x, y) \mapsto \an{\ol x \cup y, [N]} \in \Q/\Z. 
\end{equation}
The linking form $b_N$ is a well-defined nonsingular $(-1)^n$-symmetric pairing: see \cite[Exercises 53-55]{D-K}.
As with the intersection form, there is a well known geometric interpretation of the linking form.
For $x,y \in TH^{n}(N)$ let $\wh x = \PD(x)$ and $\wh y = PD(y)$.
These are classes in $TH_{n-1}(N)$ which are
represented by cycles $c_{\wh x}$ and $c_{\wh y}.$ Choose a chain $w_{\wh y} \in C_{n}(N)$ such that 
$\partial w_{\wh y}=k c_{\wh y}$ for some $k \in \Z.$ (Note that $w_{\wh y}$ exists because $\wh y$ is a torsion element of $H_{n-1}(N).$) 
The linking form of $N$ satisfies 
$$ b_N(x,y) = \frac{(c_{\wh x},w_{\wh y})}{k}~\in~\Q/\Z,$$ 
where $(\ ,\ ):C_{n-1}(N) \times C_{n}(N) \rightarrow \Z$ is the intersection form on chains.


In addition to the linking form, another important invariant of $M \in \del \cal{H}^{4k}_{\QHS}$
is the primary obstruction to the stable triviality of the tangent bundle
of $M$.  This is a cohomology class
\[ \beta_M \in H^{2k}(M; \pi_{2k-1}(SO)).\]
Note that by Bott-periodicity, $\pi_{2k-1}(SO) \cong \Z, \Z/2, \Z, 0$ for $k \equiv 0, 1, 2, 3$~mod~$4$.

\begin{Remark} \label{rem:Pontrjagin_class}
Although we do not use the following fact, we point out that by \cite{Kervaire}, when $k = 2j$ is even,
\[ a_j(2j-1)!\beta_M = p_j(M),  \]
where $p_j(M)$ is the $j^{th}$ Pontrjagin class of $M$ and $a_j = (3 - (-1)^j)/2$.
\end{Remark}

We now define the extended quadratic linking form of $M \in \del \cal{H}^{4k}$ which
extends and refines the triple $(H^{2k}(M), b_M, \beta_M)$.  
We exploit the fact that $M = \del W$, for $W$ a handlebody, and that handlebodies
are classified by their extended quadratic forms $(H^{2k}(W, \del W), \lambda_W, \alpha_W)$; see Theorem~\ref{thm:wall_2n_classification}.
Since there is a homotopy equivalence $W \simeq \vee S^{2k}$ and $M$ is a rational homotopy sphere, 
the cohomology long exact sequence of the pair $(W, M)$ contains the short exact sequence
\[ 0 \to H^{2k}(W, M) \xra{~\wh \lambda_W~} H^{2k}(W) \xra{~~~} H^{2k}(M) \to 0.\]
Hence $\coker{\wh \lambda_W} = H^{2k}(M)$ and it follows that $(H^{2k}(W), \lambda_W, \alpha_W)$
is nondengenerate.  We recall from 
Section \ref{subsec:boundaries_of_eqfs} that nondegenerate
extended quadratic form $(H^{2k}(W, \del W), \lambda_W, \alpha_W)$ defines an extended quadratic 
linking form on the finite abelian group $G : = \coker(\wh \lambda) = H^{2k}(M)$, and this
is how we define the extended quadratic linking form of $M$.
%
%

\begin{Definition} \label{def:eqlf_of_M}
Given $M \in \del \cal{H}^{4k}_{\QHS}$, chose any handlebody $W \in \cal{H}^{4k}$ with $\del W = M$.  
The extended quadratic linking form of $M$ is defined to be the quadruple
\[ (H^{2k}(M), b_M, q_M, \beta_M) : = \del(H^{2k}(W, M), -\lambda_W, \alpha_W).\]	
where the right-hand side is the algebraic boundary (Definition~\ref{def:boundary_of_an_eqf}) of the extended intersection form of $-W$ (see Section~\ref{subsec:classification_of_handlebodies}). 
\end{Definition}

\begin{Remark} \label{rem:eqlf_definition}
A few words are needed concerning Definition \ref{def:eqlf_of_M}.  
Firstly, the sign of $q_M$ in Definition \ref{def:eqlf_of_M}, differs from that given
in~\cite[Definition 2.10]{C1}.  This is because \cite[Definition 2.22]{C1} gave the
wrong sign for the linking of form $b_M$.
Indeed, the proof of
\cite[Theorem 2.1]{A-H-V} shows that the linking form $b_M$
as defined in \eqref{eq:linking_form} above is the negative of the linking form as defined in 
Definition~\ref{def:boundary_of_an_eqf}~\eqref{def:boundary_of_an_eqf:linking_form}.

The fact that extended quadratic linking form of $M$ is a well-defined almost diffeomorphism invariant
follows from the proof \cite[Theorem 8]{Wa4} when $k \neq 2, 4$, and when $k = 2, 4$ this is 
this follows from \cite[Lemma 2.15]{C1}: in both cases the signs must be changed (the proof
of \cite[Theorem 8]{Wa4} also uses the wrong sign for the linking form).  
For the reader's convenience, and the sake of a unified presentation, we give the proof in
Theorem~\ref{thm:classification_of_QHS}~\eqref{thm:classification_of_QHS:invariant} below.
\end{Remark}

Recall that an almost diffeomorphism between closed $m$-manifolds $N_0$ and $N_1$
is a diffeomorphism $f \colon N_0 \cong N_1 \sharp \Sigma$ for some homotopy $m$-sphere $\Sigma$.
When working with almost diffeomorphisms, it is convenient to regard $N_1$ and $N_1 \sharp \Sigma$
as the same topological space.  In particular we identify $H^*(N_1) = H^*(N_1 \sharp \Sigma)$.
The following classification theorem states that the extended quadratic linking form $(H^{2k}(M), b_M, q_M, \beta_M)$
is a complete almost diffeomorphism invariant of a rational homotopy sphere $M \in \del \cal{H}^{4k}_{\QHS}:$

\begin{Theorem} \label{thm:classification_of_QHS}
Let $k \geq 2$ and $M_0, M_1 \in \del \cal{H}^{4k}_{\QHS}$.
\begin{enumerate}
\item \label{thm:classification_of_QHS:invariant}
An almost diffeomorphism $f \colon M_0 \cong M_1 \sharp \Sigma$ induces an isomorphism
of extended quadratic linking forms:
\[ f^* \colon (H^{2k}(M_1); b_{M_1}, q_{M_1}, \beta_{M_1}) \cong (H^{2k}(M_0); b_{M_0}, q_{M_0}, \beta_{M_0}). \]
\item \label{thm:classification_of_QHS:complete}
For any isomorphism of extended quadratic linking forms
\[ A \colon (H^{2k}(M_1); b_{M_1}, q_{M_1}, \beta_{M_1}) \cong (H^{2k}(M_0); b_{M_0}, q_{M_0}, \beta_{M_0}), \]
there is a homotopy sphere $\Sigma \in \del \cal{H}^{4k}$ 
and a diffeomorphism  $f_A \colon M_0 \cong M_1 \sharp \Sigma$ such that
\[ f^*_A = A \colon H^{2k}(M_1) \cong H^{2k}(M_0) .\]
\end{enumerate}
\end{Theorem}

\begin{Remark} \label{rem:Wall}
For $k \neq 2, 4$, the above classification result Theorem \ref{thm:classification_of_QHS} (2) can also be deduced from \cite[Theorem 7]{Wa4}: see \cite[\S 14]{Wa4} and the proof of \cite[Theorem 8]{Wa4}.
\end{Remark}

\noindent{\it Proof of Theorem \ref{thm:classification_of_QHS}}
We first note that by \ref{thm:delTP}~\eqref{thm:delTP:bordism} we have $[M_0]=[M_1]=0 \in \Omega^{O\an{2k-1}}_{4k-1}.$ Given a diffeomorphism $M_0 \cong M_1 \sharp \Sigma$ we see that $[M_1\sharp \Sigma]=0 \in \Omega^{O\an{2k-1}}_{4k-1}$ also. From this we deduce that $[\Sigma]=0,$ and so by \ref{thm:delTP}~\eqref{thm:delTP:bordism} again we have $\Sigma \in \del \cal{H}^{4k}.$ 
Hence there is a handlebody $W_\Sigma$ with $\del W_\Sigma = \Sigma,$ as well as handlebodies $W_i$ with $\del W_i = M_i.$
At this point  we shall need the following
theorem of Wilkens PhD thesis \cite{Wi1} which was restated with Wilkens' proof in \cite{C1}.

\begin{Theorem}[{\cite[Theorem 3.2]{Wi1}}, {\cite[Theorem 2.24]{C1}}] \label{thm:Wilkens}
With the notation above, there are additional handlebodies $W_2$ and $W_3$ with boundaries
the standard sphere and a diffeomorphism $F$ such the following diagram commutes:
\[ \xymatrix{ M_0 \ar[d] \ar[r]^f & M_1 \sharp \Sigma \ar[d] \\ 
W_0 \natural W_2 \ar[r]^(0.4)F & W_1 \natural W_\Sigma \natural W_3. } \]
\end{Theorem}

\begin{proof}[Proof of Theorem \ref{thm:classification_of_QHS} continued.] Let $W_4 : = W_0 \natural W_2$ and $W_5 : = W_1 \natural W_\Sigma \natural W_3$.
Since $\del W_2, \del W_\Sigma$ and $\del W_3$ are all homotopy spheres.  It follows that
\[ (H^{2k}(M_0), b_{M_0}, q_{M_0}, \beta_{M_0}) = \del(H^{2k}(W_4, \del W_4), -\lambda_{W_4}, \alpha_{W_4}) \]
and
\[ (H^{2k}(M_1), b_{M_1}, q_{M_1}, \beta_{M_1}) = \del(H^{2k}(W_5, \del W_5), -\lambda_{W_5}, \alpha_{W_5}).\]
The diffeomorphism $F$ in Theorem \ref{thm:Wilkens} induces an isomorphism of extended quadratic forms
\[ F^* \colon (H^{2k}(W_5), -\lambda_{W_5}, \alpha_{W_5}) \cong (H^{2k}(W_4), -\lambda_{W_4}, \alpha_{W_4}) \]
and hence $f^* = \del F^*$ is an isomorphism of the extended quadratic linking forms on the boundary,
as required.  This proves Theorem~\ref{thm:classification_of_QHS}~\eqref{thm:classification_of_QHS:invariant}.

Now suppose we are given $A \colon H^{2k}(M_1) \cong H^{2k}(M_0)$, an isomorphism of extended quadratic
linking forms.  For $i=0,1$ there are handlebodies $W_i$ such that $M_i = \del W_i$ and  
\[ \del(H^{2k}(W_i, M_i), -\lambda_{W_i}, \alpha_{W_i}) = (H^{2k}(M_i), b_{M_i}, q_{M_i}, \beta_{M_i}).\]
By Theorem \ref{thm:boundaries_of_eqfs} there are non-singular extended quadratic forms 
$(H_2,\lambda_2,\alpha_2)$, $(H_3,\lambda_3,\alpha_3)$ and an isomorphism of extended quadratic forms
\[ B \colon (H_0, -\lambda_0, \alpha_0) \oplus (H_2, -\lambda_2, \alpha_2) \cong (H_1, -\lambda_1, \alpha_1) \oplus (H_3, -\lambda_3, \alpha_3),\]
where we have abbreviated $(H^{2k}(W_i, \del W_i), \lambda_{W_i}, \alpha_{W_i}) = (H_i, \lambda_i, \alpha_i)$ for $i=0,1.$
Next, by Theorem ~\ref{thm:wall_2n_classification} there are handlebodies $W_2,$ $W_3$ with extended intersection forms $(H_2,\lambda_2,\alpha_2)$, $(H_3,\lambda_3,\alpha_3).$ We claim that the boundaries of these handlebodies are homotopy spheres. To see this, recall that the algebraic boundary of a non-singular extended quadratic form is zero (as observed before Theorem \ref{thm:boundaries_of_eqfs}). In particular this means that the group $G=H_{2k-1}(\del W_i)=0$ for $i=2,3,$ from which the claim follows easily. Next, by Theorem~\ref{thm:wall_2n_classification} again, the isomorphism $B$ above is realised by a diffeomoprhism $f_B \colon W_0 \natural W_2 \cong W_1 \natural W_3$.
Restricting $f_B$ to the boundary, we obtain a diffeomorphism
\[ \del f_B \colon M_0 \sharp \Sigma_2 \cong M_1 \sharp \Sigma_3 .\]
This entails that there is a diffeomorphism $f \colon M_0 \cong M_1 \sharp \Sigma$, where 
$\Sigma = \Sigma_2 \sharp (-\Sigma_1) \in \del \cal{H}^{4k}$.
This proves Theorem~\ref{thm:classification_of_QHS}~\eqref{thm:classification_of_QHS:complete}.
\end{proof}

Now recall the boundary map of $\del \colon \cal{F}^{4k}_{\rm nd} \to \cal{Q}^{4k-1}$ of \eqref{eq:del}, which
associates to every nondegenerate extended quadratic form the boundary extended 
quadratic linking form.  We define $\cal{TF}^{4k}_{\rm nd} \subset \cal{F}^{4k}_{\rm nd}$ be
the set of isomorphism classes of nondegenerate treelike extended quadratic forms.
The following theorem is the main algebraic result of this paper and will
be proven in Section \ref{sec:realising_extended_quadratic_linking_forms}. 

\begin{Theorem} \label{thm:realisation_of_eqlfs}
The boundary map $\del \colon \cal{TF}^{4k}_{\rm nd} \to \cal{Q}^{4k-1}$ is onto.  
That is, for every extended quadratic linking form $(G, b, q, \beta)$, 
there is a nondegenerate treelike 
extended quadratic form $(H, \lambda, \alpha)$
such that $(G, b, q, \beta)$ is isomorphic to $\del(H, \lambda, \alpha)$.
%
%
%
%
%
\end{Theorem}


\begin{proof}[Proof of Theorem~\ref{thm:delTP}~\eqref{thm:delTP:main}]
Let $M$ be the boundary of a handlebody $W \in \cal{H}^{4k}$.  We must show that $M$ is the boundary of a treelike plumbing.
We first reduce to the case where $M$ is a rational homotopy sphere.
By \cite[Theorem 7]{Wa4}, every $M \in \del \cal{H}^{4k}$ can be written as a connected sum
\[ M = M_T \sharp \bigl( \sharp_{i = 1}^b M_i \bigr), \]
where $M_T$ is a rational homotopy sphere, $b$ is the rank of $H_{2k-1}(M; \Q)$ and each 
$M_i$ is the total space of a $(2k{-}1)$-sphere bundle over $S^{2k}$.  Since each $M_i$ clearly
belongs to $\del \cal{TP}^{2m}$, and by Theorem \ref{thm:connected_sum_for_delTP-G} or
Theorem~\ref{thm:TP_and_connected_sum}, $\del \cal{TP}^{2m}$ is closed under connected sum, it remains to show that $M_T \in \del \cal{TP}^{4k}$. 

Let $M_T$ have extended quadratic linking form $(H^{2k}(M_T), b_{M_T}, q_{M_T}, \beta_{M_T})$.  
By Theorem~\ref{thm:realisation_of_eqlfs}, there is a nondegenerate treelike extended quadratic form
$(H, \lambda, \alpha)$ and an isomorphism 
\[ \del(H, \lambda,\alpha) \cong (H^{2k}(M_T), b_{M_T}, q_{M_T}, \beta_{M_T}).\]
By Lemma~\ref{lem:tree_recognition} there is a treelike plumbing $W$ with extended intersection
form isomorphic to $(H, -\lambda, \alpha)$.  Hence $\del W$ and $M_T$ have isomorphic extended
quadratic linking forms and so by
Theorem~\ref{thm:classification_of_QHS}~\eqref{thm:classification_of_QHS:complete}, there is a homotopy sphere
$\Sigma \in \del \cal{H}^{4k}$ and a diffeomorphism $\del W \sharp \Sigma \cong M_T$.  Now, 
$\del W$ and $\Sigma$ both belong to $\del \cal{TP}^{4k}$, and so by Theorem \ref{thm:connected_sum_for_delTP-G} or
Theorem~\ref{thm:TP_and_connected_sum}, $M_T \in \del \cal{TP}^{4k}$.
\end{proof}

\section{Realising linking forms as the boundaries of treelike forms} 
\label{sec:realising_extended_quadratic_linking_forms}
In this section we prove Theorem \ref{thm:realisation_of_eqlfs} which states
that for every extended quadratic linking form $(G, b, q, \beta)$ there is a nondegenerate 
treelike extended quadratic form $(H, \lambda, \alpha)$
such that $(G, b, q, \beta)$ is isomorphic to $\del(H, \lambda, \alpha)$.
In this situation we shall say that $(H, \lambda, \alpha)$ {\em presents} $(G, b, q, \beta)$.
We begin by outlining the strategy of the proof. 

We call an extended quadratic linking form $(G, b, q, \beta)$ 
{\em decomposable} if it can be written as a nontrivial orthogonal sum, and {\em indecomposable} otherwise.
By Lemma \ref{lem:stable_closure_of_treelike_forms}, the image of the boundary map 
$\del \colon \cal{TF}^{4k} \to \cal{Q}^{4k-1}$
is closed under orthogonal sum.  Hence it suffices to prove that every indecomposable extended quadratic linking form
$(G, b, q, \beta)$ 
can be presented by a nondegenerate treelike form $(H, \lambda, \alpha)$.
We define
\[ \cal{Q}^{4k-1}(G_0, b_0, q_0) := \{ (G, b, q, \beta) \,|\,(G, b, q) \cong (G_0, b_0, q_0) \} \subset \cal{Q}^{4k-1} \]
and 
\[ \cal{Q}^{4k-1}(G_0, b_0) := \{ (G, b, q, \beta) \,|\,(G, b) \cong (G_0, b_0) \} \subset \cal{Q}^{4k-1} \]
to be, respectively, the set of isomorphism classes of linking forms with a fixed isomorphism class of quadratic refinement,
repsectively a fixed isomorphism class of linking form.  Recall that the quadratic refinement $q$ determines $b$ and
so $\cal{Q}^{4k-1}(G_0, b_0, q_0) \subset \cal{Q}^{4k-1}(G_0, b_0)$.

The next result says that if an extended quadratic form $(H,\lambda,\alpha)$ over $\pi_{2k-1}\{\hbox{SO}(2k)\}$ has boundary linking form $(G_0,b_0)$ in the case $k=2,4,$ respectively boundary linking form with quadratic refinement $(G_0,b_0,q_0)$ in the case $k \neq 2,4,$ then we can present all possible extended quadratic linking forms extending $(G_0,b_0)$ respectively $(G_0,b_0,q_0)$ as the boundaries of treelike forms. In fact, in the former case we can produce all quadratic refinements of $(G_0, b_0)$ and corresponding elements $\beta$ in this way, and in the latter case all possible elements $\beta.$  


\begin{Lemma} \label{lem:removing_beta}
Let $(H_0, \lambda_0, \alpha_0)$ be a nondegenerate treelike extended quadratic form.
\begin{enumerate}
\item \label{lem:removing_beta:24}
If $k = 2, 4$ and $\del(H_0, \lambda_0, \alpha_0) \in \cal{Q}^{4k-1}(G_0, b_0)$, then 
$\cal{Q}^{4k-1}(G_0, b_0) \subset \del \cal{TF}^{4k}$. 
\item \label{lem:revmoing_beta:not24}
If $k \neq 2, 4$ and $\del(H_0, \lambda_0, \alpha_0) \in \cal{Q}^{4k-1}(G_0, b_0, q_0)$, then 
$\cal{Q}^{4k-1}(G_0, b_0, q_0) \subset \del \cal{TF}^{4k}$. 
\end{enumerate}
\end{Lemma}

\begin{proof}
As $\alpha \in H^* \otimes \pi_{2k-1}(SO)$ varies over all possible values of $\alpha$
for which $(H_0, \lambda_0, \alpha) \in \cal{F}^{4k}$, we look at the linking forms $\del(H_0, \lambda_0, \alpha)$
which arise.  
If $k = 2, 4$ then by the proof of \cite[Corollary 2.23 (2)]{C1}, every linking form 
$(G, b, q, \beta) \in \cal{Q}^{4k-1}(G_0, b_0)$ arises in this way.
If $k \neq 2, 4$ then $\alpha$ is independent of $\lambda$, by Lemma \ref{lem:mu_and_alpha}, and since $\pi \colon H^* \to G$ is onto,
all possible values of $\beta \in G \otimes \pi_{2k-1}(SO)$ arise from $\alpha \in H^* \otimes \pi_{2k-1}(SO)$.
\end{proof}

Notice that every linking form $(G, b)$ has a homogeneous quadratic refinement $q$. This is axiomatic in the case $k \neq 2,4,$ 
(see Definition \ref{def:eqlfd}) and in the case $k=2,4$ this follows from 
Lemma~\ref{lem:removing_beta}~\eqref{lem:removing_beta:24} by choosing the homogeneity defect $\beta$ to be 0. With this in mind, the above discussion shows that it suffices to prove that every homogeneous quadratic linking form $(G, b, q)$ can be presented by an even treelike form $(H, \lambda)$ as
in Lemma~\ref{def:boundary_of_an_eqf}~\eqref{def:boundary_of_an_eqf:linking_form}:
Recall from Section \ref{subsec:treelike_forms}, that a $\Z$-labelled tree $\mathfrak{t} = ((\V, \E), (\ul a))$ defines 
a symmetric bilinear from $(H_\mathfrak{t}, \lambda_\mathfrak{t})$ which is even if and only if each $a_i \in \ul a$ is even.
We shall call a $\Z$-lablled tree with even labels an {\em even labelled tree} and
we shall say that an even labelled tree $\mathfrak{t}$ presents a linking form $(G, b, q)$
if the even symmetric form $(H_\mathfrak{t}, \lambda_\mathfrak{t})$ presents $(G, b, q)$.
Note the significance of evenness here: by Lemma \ref{lem:mu_and_alpha} 
the bilinear form $\lambda$ takes values in $2\Z$ if $k \neq 2,4.$ 
(Evenness is not required in the case $k=2,4,$ though assuming it does not create an impediment.) 
We now give a list which contains all the indecomposable homogeneous quadratic linking forms on finite abelian groups.

\begin{Theorem}[{\cite[Theorems 4 and 5]{Wa2}}] \label{thm:indcomposable_qlfs}
Let $p$ be a prime, let $j \geq 0$ be an integer and let $\theta$ be an integer prime to $p$ such that $-p^j < \theta <p^j$.
Let $x$ generate $\Z_{p^j}$ and let $y, z$ generate $\Z_{p^j}^2$.
Every indecomposable homogeneous quadratic linking form is isomorphic to one of the following:
\[ \text{Cyclic}: \quad q^\theta_{p^j} : = (G, b, q) = \left(\Z_{p^j}, \left( \frac{\theta}{p^j} \right),\  q(x) = \frac{\theta}{2p^j} \right);  \]
\[ \text{Hyperbolic}: \quad H(\Z_{2^j}) : =  (G, b, q) = \left( \Z_{2^j}^2, 
\left( \begin{array}{cc} 0 & 2^{-j} \\ 2^{-j} & 0 \end{array}\right),\  q(y) = 0 = q(z)  \right)  \]
\[ \text{Pseudo-hyperbolic}: \quad F(\Z_{2^j}) := (G, b, q) = \left( \Z_{2^j}^2, 
\left( \begin{array}{cc} 2^{1-j} & 2^{-j}\\ 2^{-j} & 2^{1-j} \end{array} \right),\  q(y) = 2^{-j}= q(z)  \right).  \]
\end{Theorem}


To prove Theorem \ref{thm:realisation_of_eqlfs} it suffices to list even labelled trees
which present the quadratic linking forms listed in Theorem \ref{thm:indcomposable_qlfs}.
This seems to us to be a non-trivial task: Wall \cite[Theorem 6]{Wa2} lists the rational inverses of 
the intersection matrices of even forms which present the quadratic linking forms in Theorem \ref{thm:indcomposable_qlfs}.  
In the cyclic case it is not
obvious that the corresponding integral even forms are treelike and in Section \ref{subsec:non-treelike_forms} below
we show that Wall's forms in the hyperbolic and psuedo-hyperbolic case are not treelike.

In Section \ref{subsec:cyclic_linking_forms} we modify Wall's arguments to find even labelled trees presenting 
all indecomposable cyclic linking forms; see Lemma \ref{lem:cyclic}.  In Sections \ref{subsec:hyperbolic_linking_forms} and 
\ref{subsec:pseudo-hyperbolic_qlfs} we find even labelled trees which present all indecomposable hyperbolic 
and pseudo-hyperbolic quadratic linking forms; see Lemmas \ref{lem:hyperbolic} and \ref{lem:psuedo-hyperbolic1}.
Theorem \ref{thm:indcomposable_qlfs} and 
Lemmas~\ref{lem:stable_closure_of_treelike_forms}, \ref{lem:cyclic}, \ref{lem:hyperbolic} and 
\ref{lem:psuedo-hyperbolic1} combine to prove Theorem \ref{thm:realisation_of_eqlfs}.



%
%

\subsection{Non-treelike forms} \label{subsec:non-treelike_forms}
In this subsection we consider symmetric bilinear forms $(H, \lambda)$ which are
not assumed to be even.  The boundary linking form of $(H, \lambda)$ as defined in 
Lemma~\ref{def:boundary_of_an_eqf}~\eqref{def:boundary_of_an_eqf:linking_form} will
be denoted $\del(H, \lambda)$.  We shall show that there are many non-treelike symmetric forms
by applying the following simple observation, which gives a criterion for finding non-treelike symmetric forms.

\begin{Lemma} \label{lem:treelike_criterion}
Let $(H, \lambda)$ be a nondegenerate symmetric bilinear form and suppose that there is an
integer $k > 1$ such that $(H, \lambda) \cong (H, k\lambda_0)$ for some other 
nondegenerate symmetric bilinear form $\lambda_0$ on $H$.  If $(H, \lambda)$ is treelike
then the linking form $\del(H, \lambda)$ is isomorphic to a sum of cyclic linking forms.
\end{Lemma}

\begin{proof}
Let $(h_1, \dots, h_r)$ be a treelike basis for $(H, \lambda)$, let $A$ and $A^0$ be the matricies
of $\lambda$ and $\lambda_0$ with respect to this basis.  Since $(h_1, \dots, h_r)$ is a treelike basis
for $(H, \lambda)$, it follows that the off-diagonal matrix elements $A_{ij}$ are either zero or one.
But $A_{ij} = k A^0_{ij}$, and so $A_{ij} = 0$.  It follows that $A$ is diagonal and hence 
$\del(H, \lambda)$ is a sum of cyclic linking forms.
\end{proof}

We next define some symmetric forms where we can apply Lemma \ref{lem:treelike_criterion}.
Firstly, we let $b_H(\Z_{2^j})$ and $b_F(\Z_{2^j})$ be the linking
forms underlying the quadratic linking forms $H(\Z_{2^j})$ and $F(\Z_{2^j})$ respectively.
Let $H_+(2^j)$ denote the symmetric bilinear form on $\Z^2$ with intersection matrix
$$\left( \begin{array}{cc}
0&2^j \\
2^j&0 \\
\end{array}\right).$$
The boundary linking form of $H_+(2^j)$
is $b_H(\Z_{2^j})$ and clearly $H_+(2^j) = 2^j H_+(1)$. 

In the pseudo-hyperbolic case, we define $F_+(2^j)$ to be the even symmetric bilinear form over $\Z^4$
as follows.  Setting
\[ a_j : = \frac{1}{3}(2^j - (-1)^j) \quad \text{and} \quad b_j : = (-1)^{j-1},\]
with $j \geq 1$, the form $F_+(2^j) = (\Z^4, \lambda_{F_+(2^j)})$ is represented by the following matrix:
\[  A(F_+(2^j)) : = \left( \begin{array}{cccc}
2^{j+1}(4a_jb_j-1 -2^{j}b_j)~& -2^j(4a_jb_j-1)~& 2^{j+1}b_j~& -2^j~ \\
-2^j(4a_jb_j-1) & 2^{j+1}(4a_jb_j-1) & -2^{j+2}b_j & 2^{j+1} \\
2^{j+1}b_j & -2^{j+2}b_j & 6 b_j & -12 \\
-2^j & 2^{j+1} & -12 & 12a_j - 2^{j+1} \end{array}   \right)  \]
One calculates that the boundary linking form of $F_+(2^j)$ is $b_F(\Z_{2^j})$.
Indeed, Wall \cite[\S 7]{Wa2} writes down a symmetric matrix 
$B(F_{2^j})$ over $\Q^4$ whose inverse matrix is $A(F_+(2^j))$
and whose quadratic boundary is the pseudo-hypberbolic quadratic linking form $F(\Z_{2^j})$.  
It is clear that $(\Z^4, \lambda_{F_+(2^j)}) = (\Z^4, 2\lambda'_{F_+(2^j)})$ for an integral
form $(\Z^4, 2\lambda'_{F_+(2^j)})$.

Now let $(G, b) = (G_0, b_0) \oplus (G_1, b_1)$ where $(G_1, b_1)$ is 
isomorphic to either $b_H(\Z_{2^j})$ or $b_F(\Z_{2^j})$ and where 
$2^j \cdot G_0$ is an odd torsion group.  Then for all $x \in G$, there is 
an odd integer $q$ such that
$q 2^{j-1}b(x, x) = 0$.  It follows that
$(G, b)$ contains no cyclic orthogonal summand on $\Z_{2^j}$
and in particular $(G, b)$ is not isomorphic to a sum of cyclic linking forms.  
Applying Lemma~\ref{lem:treelike_criterion} to the above discussion we obtain

\begin{Lemma} \label{lem:non-treelike}
Fix a positive integer $j$ and 
let $(H, 2 \lambda)$ be a nondegenerate symmetric bilinear form such that $2^{j} \cdot \coker(\wh \lambda)$
is odd torsion.  Then the symmetric bilinear forms
\[ (H, 2\lambda) \oplus H_+(2^j) \quad \text{and} \quad
(H, 2 \lambda) \oplus F_+(2^{j}) \]
are not treelike.  In particular for any positive integers $i_s$ and $j_t$, any non-trivial sum of forms $H_+(2^{i_s})$ and $F_+(2^{j_t})$ is not treelike. \qed
\end{Lemma}

\subsection{Cyclic linking forms} \label{subsec:cyclic_linking_forms}
In this subsection subsection we show how to present $2$-primary cyclic linking forms by $2\Z$-labelled treelike
graphs.  Our arguments in this section are based on the proof of \cite[Theorem 6]{Wa1}.
We shall require the follow lemma which is a simple consequence of the classification of cyclic quadratic linking forms.
%
%

\begin{Lemma} \label{lem:theta}
Every linking form $q^\theta_{p^j}$ is isomorphic to some $q^{\theta'}_{p^j}$ where $\theta'$ and $p$ 
have the opposite parity.		
\end{Lemma}

\begin{proof}
In the case $p = 2,$ note that both $\theta$ and $\theta'$ must be odd. 
We can therefore assume that $p$ is odd. In this case, linking forms over $\Z_{p^j}$ are determined up to isomorphism by whether $\theta$ mod $p$ is a quadratic residue or a quadratic non-residue \cite[Theorem 4]{Wa2}. Since $1$ and $4$ are both quadratic residues mod $p$ for all $p>3,$ it follows that for such $p$ the set of quadratic residues and the set of quadratic non-residues both contain even integers. Thus for any $\theta$ we can find an {\it even} $\theta'$ such that $q^{\theta}_{p^j} \cong q^{\theta'}_{p^j}.$ If $p = 3$, then for any admissable choice of $\theta$ we can take either $\theta' = -2$ or $\theta' = 2,$ since $-2$ is a quadratic residue and 2 is a non-residue. 
\end{proof}

We will need the following specific version of the Euclidean algorithm.
We suppose that $d_1$ and $d_2$ are co-prime 
integers of opposite parity and that $|d_1| > |d_2|$.
We shall be interested in the situation where $(d_1, d_2) = (p^j, \theta)$.
Since $\frac{-\theta}{-p^j} = \frac{\theta}{p^j}$, we shall
have the freedom to change the sign of {\em both} $d_1$ and $d_2$.
We define even integers $a_i \in 2 \Z$ and integers $d_i \in \Z$, as follows. 
Starting from $i = 1$, let
\begin{equation} \label{eq:a_i}
d_i = a_i \cdot d_{i+1} - d_{i+2} 
\end{equation}
where $0 \leq |d_{i+2}| < |d_{i+1}|$.  
That is we choose the even integer multiple of $d_{i+1}$ which is closest to $d_i$,
so that $|d_{i+2}| < |d_{i+1}|$.  As in the usual Euclidean algorithm, if 
$d_{i+2}$ divides $d_{i+1}$, then by induction $d_{i+2}$ divides $d_1$ and $d_2$.
Since we assume that $d_1$ and $d_2$ are coprime, we see that the numbers
$|d_i|$ form a descending sequence of non-negative integers with alternating parity
and which must therefore finish with 1 and 0.
Let $d_{n+2} = 0$ so that $d_{n+1} = \pm 1$ and 
define the sequence $\{d_1', d_2', \dots , d_{n+1}'\}$ by 
\[ d_i' := d_{n+1} \cdot d_i .\]
We note that the $d_i'$ also satisfy \eqref{eq:a_i} and that if $d_{n+1} = -1$,
then we have changed the sign of both $d_1$ and $d_2$.  Moreover, putting
$i = n-1$ and $i=n$ into \eqref{eq:a_i} above we obatin the equations,
\[ d'_{n-1} = a_{n-1}\cdot d'_{n} - 1 \quad \text{and} \quad d_n' = a_n.\]
Setting $\ul{a} = (a_1, \dots, a_n)$ we take the even labelled tree $\mathfrak{d}(\ul{a})$ as 
defined in \eqref{eq:D(a)} in Section \ref{subsec:treelike_forms}.

\begin{Lemma} \label{lem:D(a)}
With the notation above, set $(d_1, d_2) = (p^j, \theta)$, where $\theta$ and $p$ have opposite parity.
The even labelled tree $\mathfrak{d}(\ul{a})$ presents $q^\theta_{p^j}$.
\end{Lemma}

\begin{proof}
Let $(H_{\mathfrak{d}(\ul{a})}, \lambda_{\mathfrak{d}(\ul{a})})$ be the even symmetric form defined by $\mathfrak{d}(\ul{a})$
and let $(G, b, q) = \del(H_{\mathfrak{d}(\ul{a})}, \lambda_{\mathfrak{d}(\ul{a})})$.
We wrote down the intersection matrix $A_{\mathfrak{d}(\ul{a})}$ of $\lambda_{\mathfrak{d}(\ul{a})}$ with respect to 
the canonical basis of $H_{\mathfrak{d}(\ul{a})}$ in \eqref{eq:AD(a)} in Section \ref{subsec:treelike_forms}.
By induction, the last $i$ rows and columns of $A_{\mathfrak{d}(\ul{a})}$ have determinant $d'_{n-i+1}$.
In particular, $A_{\mathfrak{d}(\ul{a})}$ has determinant $d_1' = d_{n+1} \cdot p^j$ and the submatrix 
$M_{1, 1}(A)$ obtained by deleting the first row and column of $A$ has determinant $d_2' = d_{n+1} \cdot \theta$.  
It follows that the $(1, 1)$-entry of $A^{-1}_{\mathfrak{d}(\ul{a})}$
is
\[  A^{-1}_{\mathfrak{d}(\ul{a})}(1, 1) =  \frac{d_2'}{d_1'} = \frac{\theta}{p^j}.\]
The free abelian group $H_{\mathfrak{d}(\ul{a})}$ is generated by the vertices $\{v_1,...,v_n\}$ of the tree $\mathfrak{d}(\ul{a}),$ and we have the dual basis element $v_1^* \in H^*$. 
Since $|\det(A_{\mathfrak{d}(\ul{a})})| = p^j$, we deduce from 
Lemma~\ref{lem:inverse_form}~\eqref{lem:inverse_form:detA} that $G := \coker(\wh \lambda_{\mathfrak{d}(\ul{a})})$
has $|G|=p^j.$  Since $A^{-1}_{\mathfrak{d}(\ul{a})}(1, 1) =\theta/p^j$ we infer from Definition~\ref{def:boundary_of_an_eqf}~\eqref{def:boundary_of_an_eqf:linking_form} and 
Lemma~\ref{lem:inverse_form}~\eqref{lem:inverse_form:lambda}, that if we set $g:=\pi(v_1^*) \in G,$ we must have $q(g)=\theta/2\cdot p^j$ and hence $b(g, g) = \theta/p^j$.  
It follows from this that $G=\langle g \rangle,$ i.e. $G \cong \Z_{p^j}.$  Thus $(G, b, q)$ is an indecomposable cyclic linking form on $\Z_{p^j},$ and so is isomorphic to $q_{p^j}^\theta$.  Thus $\mathfrak{d}(\ul{a})$ presents $q^\theta_{p^j}$.
\end{proof}

From Theorem~\ref{thm:indcomposable_qlfs} and Lemmas~\ref{lem:D(a)} and~\ref{lem:theta}, we obtain

\begin{Lemma} \label{lem:cyclic}
Every indecomposable cyclic 
linking form can be presented by an even labelled tree.
\end{Lemma}

\subsection{Hyperbolic linking forms} \label{subsec:hyperbolic_linking_forms}
In this subsection, we show how to present the $2$-primary hyperbolic linking forms $H(\Z_{2^j})$ by 
even labelled trees.  We begin with some general considerations which we shall also use in the pseudo-hyperbolic
case which follows in Section \ref{subsec:pseudo-hyperbolic_qlfs}.

Consider the following situation.  We let $\mathfrak{t} = ((\V, \E), (a_i))$ be a $\Z$-labelled tree with distinguished
vertex $v_0$.  We let $\wh{\mathfrak{t}}$ be the labelled subgraph obtained by removing $v_0$ and of course all the edges containing $v_0$.  
Given $a \in \Z$, we form a new labelled tree by adding new verticies 
$\{v_{-1}, \dots, v_{-m}\}$ each with label $a$ and a single edge from each new vertex to $v_0$.  We shall consider just the two labelled trees
where $m = 1, 2$,
\[ a-\mathfrak{t} := 
\entrymodifiers={++[o][F-]}
\xymatrix{a \ar@{-}[r] & *\txt{\,$\mathfrak{t}$} }\,, \]
and
\[ 
\entrymodifiers={++[o][F-]}
\xymatrix{ *\txt{} & a \ar@{-}[dr] \\ *\txt{$(a,a)-\mathfrak{t}$ := } & *\txt{} & *\txt{\,$\mathfrak{t}$\,.} \\ *\txt{} & a \ar@{-}[ur]  } \]
Let $A$ be the intersection matrix of $\lambda_{(a,a)-\mathfrak{t}}$ obtained from the
basis $\{v_{-2}, v_{-1}, v_0, v_1, \dots, v_r\}$ and let
$M_{-1, -2}(A)$ be the submatrix of $A$ obtained by deleting the -1 row and -2 column from $A$.

\begin{Lemma} \label{lem:det}
Let $x = \det(\lambda_\mathfrak{t})$ and $y = \det(\lambda_{\wh{\mathfrak{t}}})$.
We adopt the convention that if $\wh{\mathfrak{t}}$ is the empty graph, then $\det(\lambda_{\wh{\mathfrak{t}}}) = 1$.
Then
\begin{enumerate}
\item \label{lem:det:a}
$\det(\lambda_{a-\mathfrak{t}}) = ax - y$,
\item \label{lem:det:aa}
$\det(\lambda_{(a,a)-\mathfrak{t}}) = a^2x - 2ay$,
\item \label{lem:det:M}
$\det(M_{-1,-2}) = -y$.
\end{enumerate}
\end{Lemma}

\begin{proof}
\eqref{lem:det:a}
Ordering the vertex set of $a-\mathfrak{t}$ as $(v_{-1}, v_0, \dots, v_n)$ where $(v_0, \dots v_n)$ is the vertex
set of $X$ and $v_0$ is the distinguished vertex where we attached the new edge, we see that $\lambda_{a-\mathfrak{t}}$
has intersection matrix
\[ \left( \begin{array}{ccccc} a & 1 & 0 & 0 & \dots\\
1 & a_{0} & a_{01} & a_{02} & \dots \\
0 & a_{01} & a_1 & a_{12} & \dots \\
0 & a_{02} & a_{12} & a_2 & \dots \\
\vdots & \vdots & \vdots & \vdots & \ddots \end{array} \right). \]
Using the first column expansion for the determinant of $\lambda_{a-\mathfrak{t}}$ gives the result.

\eqref{lem:det:aa}
Follows by applying part \eqref{lem:det:a} to $\lambda_{(a,a)-\mathfrak{t}}$.

\eqref{lem:det:M}
We order the vertex set of $(a,a)-\mathfrak{t}$ as $\{v_{-2}, v_{-1}, v_0, \dots, v_r \}$ where $\{v_0, \dots v_r\}$ is the vertex set of $\mathfrak{t}$ and $v_0$ is the distinguished vertex where we attached the new edge.  With respect to this basis, the intersection matrix of $\lambda_{(a, a)-\mathfrak{t}}$ is
\[ 
A = \left( \begin{array}{cccccc} 
a & 0 & 1 & 0 & 0 & \dots\\
0 & a & 1 & 0 & 0 & \dots \\
1 & 1 & a_0 & a_{01} & a_{02} & \dots \\
0 & 0 & a_{01} & a_1  & a_{12} & \dots \\
0 & 0 & a_{02} & a_{12} & a_2 & \dots \\
\vdots & \vdots & \vdots & \vdots & \vdots & \ddots  \end{array} \right). \]
%
Let $M := M_{-1, -2}(A)$ and consider the following matricies:
\[ 
M = 
\left( \begin{array}{ccccc} 
0 & 1 & 0 & 0 & \dots \\
1 & a_0 & a_{01} & a_{02} & \dots \\
0 & a_{01} & a_1  & a_{12} & \dots \\
0 & a_{02} & a_{12} & a_2 & \dots \\
\vdots & \vdots & \vdots & \vdots & \ddots\end{array} \right),~~
M_1 = 
\left( \begin{array}{ccccc} 
1 & a_{01} & a_{02} & \dots \\
0 & a_1  & a_{12} & \dots \\
0 & a_{12} & a_2 & \dots \\
\vdots & \vdots & \vdots & \ddots\end{array} \right),~~
M_2 = 
\left( \begin{array}{ccccc} 
a_1  & a_{12} & \dots \\
a_{12} & a_2 & \dots \\
\vdots & \vdots & \ddots \end{array} \right).
\]
Here $M_1$ is a submatrix of $M$ and $M_2$ is a submatrix of $M_1$ 
such that the following hold:
\begin{enumerate}
\item $M_2$ is the intersection matrix of $\lambda_{\wh{\mathfrak{t}}}$ with respect to the basis $\{v_0, \dots, v_r\}$,
\item $\det(M) = - \det(M_1) = -\det(M_2) = -y$.
\end{enumerate}
\end{proof}

%
%

Consider the even labelled tree
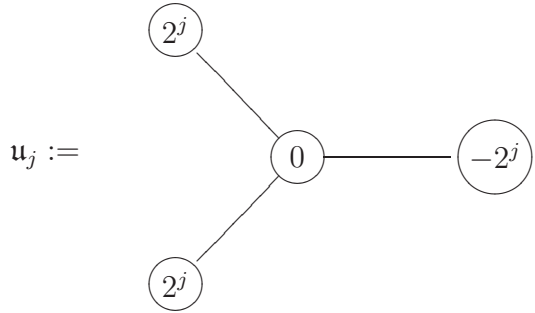
\[ 
\entrymodifiers={++[o][F-]}
\xymatrix{ *\txt{} & 2^j \ar@{-}[dr] \\ *\txt{$\mathfrak{u}_j : =$} & *\txt{} & 0 \ar@{-}[rr] & *\txt{} & -2^j \\ *\txt{} & 2^j \ar@{-}[ur] } \]
which has the intersection matrix:
\begin{equation} \label{eq:u_k} 
A_{\mathfrak{u}_j} = 
\left( \begin{array}{cccc} 2^j & 0 & 0 & 1 \\
0 & 2^j & 0 & 1 \\
0 & 0 & -2^j & 1\\
1 & 1 & 1 & 0 \end{array} \right). 
\end{equation}

\begin{Lemma} \label{lem:hyperbolic}
The even labelled tree $\mathfrak{u}_j$ presents the linking form $H(\Z_{2^j})$.
\end{Lemma}

\begin{proof}
Let $\mathfrak{u}'_j$ be the labelled tree
\[ \mathfrak{u}'_j\,:= 
\entrymodifiers={++[o][F-]}
\xymatrix{ 0 \ar@{-}[r] & -2^j} \]
with preferred vertex labelled by $0$.  Then $\mathfrak{u}_j = (2^j, 2^j)-\mathfrak{u}'_j$.  
Now $\det(\lambda_{\mathfrak{u}'_j}) = -1$ and $\det(\lambda_{\wh{\mathfrak{u}'_j}}) = -2^j$.
It follows by 
Lemma~\ref{lem:det}~\eqref{lem:det:aa} that 
$\det(\lambda_{\mathfrak{u}_j}) = -2^{2j} + 2^{2j+1} = 2^{2j}$.
If $A_{\mathfrak{u}_j}$ is the intersection matrix of $\lambda_{\mathfrak{u}_j}$, 
Lemma~\ref{lem:det}~\eqref{lem:det:a} gives that the submatrices $M_{-1, -1}(A_{\mathfrak{u}_j})$ and 
$M_{-2, -2}(A_{\mathfrak{u}_j})$ satisfy
\[ \det(M_{-1, -1}) = \det(M_{-2, -2}) = \det(\lambda_{2^j-\mathfrak{u}'_j}) = 
-2^j + 2^j = 0 .\]
By Lemma~\ref{lem:det}~\eqref{lem:det:M}, the minor $M_{-1, -2}$ of $A_{\mathfrak{u}_j}$ has
\[ \det(M_{-1, -2}) = -\det(\lambda_{\wh{\mathfrak{u}'_j}}) = 2^j. \]
It follows that $A_{\mathfrak{u}_j}^{-1}$ contains the submatrix
\[ \left( \begin{array}{ccccc} 
0  & 2^{-j} \\
2^{-j} & 0 \end{array} \right).    \]
Let $(G, b, q) = \del(H_{\mathfrak{u}_j}, \lambda_{\mathfrak{u}_j})$.
Applying Definition~\ref{def:boundary_of_an_eqf}~\eqref{def:boundary_of_an_eqf:q_k<>2,4}, 
we see that the boundary of $(H_{\mathfrak{u}_j},\lambda_{\mathfrak{u}_j})$ contains $H(\Z_{2^j})$ as a sublinking form.  
But since $\det(A_{\mathfrak{u}_j}) = 2^{2j}$,
it follows by Lemma~\ref{lem:inverse_form}~\eqref{lem:inverse_form:detA} that $|G|=2^{2j},$ and thus $(G, b, q) \cong H(\Z_{2^j})$. 
That is, $\mathfrak{u}_j$ presents $H(\Z_{2^j})$.
%
%
\end{proof}

\subsection{Pseudo-hyperbolic quadratic linking forms} \label{subsec:pseudo-hyperbolic_qlfs}
In this subsection we apply Lemma \ref{lem:det} to find an even labelled tree $\mathfrak{v}_j$ which presents
the pseudo-hyperbolic linking form $F(\Z_{2^j})$.
%
%
We define $\epsilon_j = \pm 1$ by setting $\epsilon_j = 1$ if $j$ is odd and $\epsilon_j = -1$ if $j$ is even.
We then consider the even labelled tree
\[  \mathfrak{t}_j\,:= \quad
\entrymodifiers={++[o][F-]}
\xymatrix{  0 \ar@{-}[r] & r_j \ar@{-}[r] & 0 \ar@{-}[r] & 0 \ar@{-}[r] & 2 \ar@{-}[r] & 2  }  \]
where $r_j := \frac{2(1 -\epsilon_j \cdot 2^{j-1})}{3} \in 2\Z$.  Choosing the left-most vertex of $\mathfrak{t}_j$ as
the preferred vertex $v_0$, we then form the even labelled tree 
$\mathfrak{v}_j : = (\epsilon_j \cdot 2^j, \epsilon_j \cdot 2^j)-\mathfrak{t}_j$, which can be written as follows:
\[ 
\entrymodifiers={++[o][F-]}
\xymatrix{ *\txt{} & \epsilon_j \cdot 2^j \ar@{-}[dr] \\ 
*\txt{$\mathfrak{v}_j\,:=$} & *\txt{} & 0 \ar@{-}[r] & r_j \ar@{-}[r] & 0 \ar@{-}[r] & 0 \ar@{-}[r] & 2 \ar@{-}[r] & 2
\\ *\txt{} & \epsilon_j \cdot 2^j \ar@{-}[ur]  }  \]

\begin{Lemma} \label{lem:psuedo-hyperbolic1}
The labelled tree $\mathfrak{v}_j = (\epsilon_j \cdot 2^j, \epsilon_j \cdot 2^j)-\mathfrak{t}_j$
presents the linking form $F(\Z_{2^j})$.
\end{Lemma}

Lemma \ref{lem:psuedo-hyperbolic1} follows straight away from the following two lemmas.

\begin{Lemma} \label{lem:psuedo-hyperbolic2}
Let $\mathfrak{t}$ be a labelled tree with preferred vertex $v_0$ and subgraph $\wh{\mathfrak{t}}$ obtained by deleting $v_0$.
For $\epsilon = \pm 1$, if $\det(\lambda_{\mathfrak{t}}) = 3$ and 
$\det(\lambda_{\wh{\mathfrak{t}}}) = \epsilon \cdot 2^j$, then the labelled
graph $\mathfrak{v}_j : = (\epsilon \cdot 2^j, \epsilon \cdot 2^j)-\mathfrak{t}$ presents $F(\Z_{2^j})$.	
\end{Lemma}

\begin{Lemma} \label{lem:psuedo-hyperbolic3}
The labelled tree $\mathfrak{t}_j$ with preferred vertex the left-most vertex satisfies
\[ \det(\lambda_{\mathfrak{t}_j}) = 3 \quad \text{and} \quad \det(\lambda_{\wh{\mathfrak{t}_j}}) = \epsilon_j \cdot 2^j. \]
\end{Lemma}

\begin{proof}[Proof of Lemma \ref{lem:psuedo-hyperbolic2}]
By Lemma~\ref{lem:det}~\eqref{lem:det:aa}, 
$\det(\mathfrak{v}_j) = \epsilon^2(3 \cdot 2^{2j} - 2 \cdot 2^j \cdot 2^j) = 2^{2j}$.
By Lemma~\ref{lem:det}~\eqref{lem:det:a}, if $A_{\mathfrak{v}_j}$ is the intersection matrix of $\lambda_{\mathfrak{v}_j}$, we see that the minors $M_{-2, -2}$ and $M_{-1, -1}$ of $A_{\mathfrak{v}_j}$ satisfy
\[ \det(M_{-1, -1}) = \det(M_{-2, -2}) = \det(\lambda_{\epsilon \cdot 2^j-\mathfrak{t}_j}) = \epsilon(3 \cdot 2^j 
-2^j) = \epsilon \cdot 2^{j+1} .\]
By Lemma~\ref{lem:det}~\eqref{lem:det:M}, the minor $M_{-1, -2}$ of $A_{\mathfrak{v}_j}$ has
\[ \det(M_{-1, -2}) = - \epsilon \cdot 2^j. \]
It follows that $A_{\mathfrak{v}_j}^{-1}$ contains the submatrix
\[  \epsilon \cdot \left( \begin{array}{ccccc} 
2^{1-j}  & 2^{-j} \\
2^{-j} & 2^{1-j} \end{array} \right).    \]
Let $(G, b, q) = \del(H_{\mathfrak{v}_j}, \lambda_{\mathfrak{v}_j})$.
Applying Definition~\ref{def:boundary_of_an_eqf}~\eqref{def:boundary_of_an_eqf:q_k=2,4}, 
we see that boundary of $(H_{\mathfrak{v}_j}, \lambda_{\mathfrak{v}_j})$
contains a sublinking form isomorphic to $F(\Z_{2^j})$.
But since $\det(A_{\mathfrak{v}_j}) = 2^{2j}$, it follows Lemma~\ref{lem:inverse_form}~\eqref{lem:inverse_form:detA} 
that $|G| = 2^{2j}$ and thus $(G, b, q) \cong F(\Z_{2^j})$. 
That is, $\mathfrak{v}_j$ presents $F(\Z_{2^j})$.
%
%
\end{proof}

\begin{proof}[Proof of Lemma \ref{lem:psuedo-hyperbolic3}]
%
%
Applying Lemma \ref{lem:det}~\eqref{lem:det:a} repeatedly from the left-most vertex of $\mathfrak{t}_j$, we see that $\det(\mathfrak{t}_j) = 3$.
Choosing $v_0$ as the left-most vertex, then we have the even labelled tree:
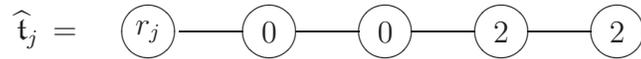
\[  \wh{\mathfrak{t}}_j\,=\quad 
\entrymodifiers={++[o][F-]}
\xymatrix{  r_j \ar@{-}[r] & 0 \ar@{-}[r] & 0 \ar@{-}[r] & 2 \ar@{-}[r] & 2 }  \]
Again, applying Lemma \ref{lem:det}~\eqref{lem:det:a} repeatedly from the left-most vertex, we see that 
\[ \det(\wh{\mathfrak{t}}_j) = -3r_j + 2 = \epsilon_j \cdot 2^{j}. \]
%
%
%
%
\end{proof}

\begin{Remark} \label{rem:tomDieck}
There is a paper by tom Dieck \cite{tomDieck} which computes linking numbers for the linking forms
of the boundaries of treelike plumbings $M \in \del \cal{TP}^{4k}$.
Lemmas \ref{lem:D(a)}, \ref{lem:hyperbolic} and \ref{lem:psuedo-hyperbolic1} 
in Sections~\ref{subsec:cyclic_linking_forms}, \ref{subsec:hyperbolic_linking_forms} and
\ref{subsec:pseudo-hyperbolic_qlfs} above can also be verified by applying \cite[Satz 1]{tomDieck}.
\end{Remark}

\section{Treelike plumbings and their boundaries in other dimensions} \label{sec:other_dimensions}
In Section \ref{sec:handlebodies} and \ref{sec:(4k-1)_manifolds} we focussed on treelike plumbings
of dimension $4k$, $k \geq 2$.  In this section we consider other dimensions.
In Section \ref{subsec:4k+1}, we look at the analogues of our results for plumbing manifolds of dimension $4k+2 \geq 6$.
In Section \ref{subsec:3-manifolds}, we briefly consider plumbing manifolds of dimension $4$.

\subsection{$2k$-parallelisable $(2k{-}1)$-connected $(4k{+}1)$-manifolds} \label{subsec:4k+1}
In this subsection we work in dimension $4k{+}1$ and consider the question of which 
$2k$-parallelisable $(2k{-}1)$-connected $(4k{+}1)$-manifolds $M$ bound treelike plumbings.  
The principal difference between 
dimensions $4k{+}1$ and $4k{-}1$ is the role of the torsion linking form.  In dimensions
$4k{-}1$, we have seen that the torsion linking form does not obstruct bounding a plumbing
manifold, but in dimensions $4k{+}1$ any torsion in $H^{2k+1}(M)$ prevents $M$ from bounding
a plumbing manifold.
We begin our discussion by stating the analogue of Theorem \ref{thm:delTP} in
dimensions $4k{+}1$.


\begin{Theorem} \label{thm:delTP-4k+2}
Let $n=4k+2 \geq 6$ and let $M \in \cal{M}_{4k+1,2k-1}^{O\an{2k}}$.
\begin{enumerate}
\item \label{thm:delTP-4k+2:bordism}
$M \in \del \cal{H}^{4k+2}$ if and only if $[M] = 0 \in \Omega_{4k+1}^{O\an{2k}}$.
\item \label{thm:delTP-4k+2:homotopy_sphere}
There is a homotopy sphere $\Sigma$ such that $M \sharp \Sigma \in \del \cal{H}^{4k+2}$. 
\item \label{thm:delTP-4k+2:main}
If $M \in \del \cal{H}^{4k+2}$, then $M \in \del \cal{TP}^{4k+2}$ if and only if $H^*(M)$ is torsion free.
\end{enumerate}
\end{Theorem}


%
%

Theorem D follows immediately from Theorem \ref{thm:delTP-4k+2} and Theorem 1.1.

The proof of part Theorem \ref{thm:delTP}~\eqref{thm:delTP:bordism} and~\eqref{thm:delTP:homotopy_sphere} carries
over to the proof of Theorem~\ref{thm:delTP-4k+2}~\eqref{thm:delTP-4k+2:bordism} and~\eqref{thm:delTP-4k+2:homotopy_sphere},
except that for part~\eqref{thm:delTP:homotopy_sphere}, we do not worry about the Brumfiel invariant.  
(For all dimensions $4k+1 \neq 2^j-3$, Brumfiel \cite{Brumfiel2} has defined an invariant which splits the Kervaire-Milnor
sequence of Theorem \ref{thm:homotopy_spheres}, but we shall not go into these subtleties here.)
Hence, it remains to prove Theorem~\ref{thm:delTP-4k+2}~\eqref{thm:delTP-4k+2:main}, and for this we need
to consider $(4k+2)$-dimensional plumbing manifolds and the labelled trees which describe them.

We begin by defining the intersection form of a labelled tree in the skew-symmetric case.
We define a skew $0$-labelled tree $\mathfrak{g} = ((\V, \E), (0, \dots, 0))$ to be a tree with ordered
vertext set $\V = (v_1, \dots v_r)$ and signed directed edge set $\E = \{e_{ij}\}$. 
The set of directed edges from $v_i$ to $v_j$, $\E_{ij}$, is either empty or has one element $e_{ij}$, 
and if $|\E_{ij}| = 1$, then $|\E_{ji}| = 0$.  We define $\epsilon_{ij} \in \{ -1, 0, 1\}$ to
be the sign of $e_{ij}$ if $\E_{ij}$ is non-empty and to be zero otherwise.
%
%
A skew $0$-labelled tree gives rise to a skew-symmetric bilinear form 
$\lambda_{\mathfrak{t}}$ on $H_{\mathfrak{t}}$, the free abelian group with basis $\V$, by defining
\begin{equation} \label{eq:skew-plumb}  
\lambda_{\mathfrak{t}}(v_i, v_j) := \epsilon_{ij}|\E_{ij}| - \epsilon_{ji}|\E_{ji}|,
\end{equation}
and extending linearly to all of $H_\mathfrak{t}$.  A skew-symmetric form $(H, \lambda)$ is called
treelike if it is isomorphic to $(H_\mathfrak{t}, \lambda_\mathfrak{t})$ for some skew $0$-labelled tree.
An extended quadratic form $(H, \lambda, \mu)$ over $\pi_{2k}\{SO(2k{+}1)\}$ is called treelike if
$(H, \lambda)$ is treelike.  We define 
\[ \cal{F}^{4k+2} :=\{(H, \lambda, \mu) \} \]
to be the
set of isomorphism classes of extended quadratic forms over $\pi_{2k}\{SO(2k{+}1)\}$ and we define
\[ \cal{TF}^{4k} \subset \cal{F}^{4k+2} \]
to be the set of isomorphism classes of treelike extended quadratic forms over $\pi_{2k}\{SO(2k{+}1)\}$.
Lemma \ref{lem:tree_recognition} and its proof carry over to dimensions $4k+2$ and so we have

\begin{Lemma} \label{lem:tree_recognition:4k+2}
If $W \in \cal{H}^{4k+2}$, then
$W \in \cal{TP}^{4k+2}$ if and only if $(H^{2k+1}(W, \del W), \lambda_W, \mu_W) \in \cal{TF}^{4k+2}$. \qed
\end{Lemma}

Now let $H_-(\Z)$ be the skew-symmetric hyperbolic form on $\Z^2$ with intersection matrix
$$\left( \begin{array}{cc}
~0&1~ \\
-1~&0~ \\
\end{array}\right)$$
and let $H_-(\Z^s) := \oplus_s H_-(\Z)$ be the $s$-fold orthongal sum of $H_-(\Z)$.  A fundamental,
if elementary fact, is that every nonsingular skew-symmetric form $(H, \lambda)$ is isomorphic to
$H_-(\Z^s)$ for some $s$; see for example \cite[Assertion p.\,55]{Mi}.  
We shall need the following two algebraic results about skew-symmetric
treelike forms.  The first follows since
Lemma \ref{lem:stable_closure_of_treelike_forms} and its proof carry over to the skew-symmetric case. 
The second is the algebraic reason why torsion cannot appear in the homology of the boundary of
treelike plumbings.

\begin{Lemma} \label{lem:stable_closure_of_skew-treelike_forms}
For any $j \geq 0$, if $(H_i, \lambda_, \alpha_i) \in \cal{TF}^{4k+2}$, then
$\bigoplus_{i=1}^j (H_i, \lambda_i, \alpha_i) \oplus (H_-(\Z), 0) \in \cal{TF}^{4k}$. \qed
\end{Lemma}

\begin{Proposition} \label{prop:skew-symmetric_treelike_forms}
Let $(H_\mathfrak{t}, \lambda_\mathfrak{t})$ be a skew-symmetric treelike form, then 
$(H_\mathfrak{t}, \lambda_\mathfrak{t})$ is isomorphic to the orthogonal sum $H_-(\Z^s) \oplus (\Z^t, 0)$ for some $s, t$. 
\end{Proposition}

\begin{proof}
We shall argue by induction on the number of verticis of $\mathfrak{t}$.  
The Proposition is certainly true for the empty tree and for a tree with one vertex.
Let $\mathfrak{t}$ have vertex set $\V = (v_0, \dots, v_r)$ and let $v_l$ be a leaf of $(\V, \E)$
so that there is a single edge $e$ between $v_l$ and $v_k$ say, which is incident with $v_l$.
We let $\mathfrak{s}$ be 
subgraph of $\mathfrak{t}$ obtained by removing $v_k$, $v_l$ and all edges involving $v_k$
and $v_l$.  The $0$-labelled graph $\mathfrak{s} = \sqcup_{i = 1}^j \mathfrak{t}_i$ is a disjoint union of $0$-labelled trees $\mathfrak{t}_i$, where each tree $\mathfrak{t}_i$
is connected to $v_k$ via a single edge in $\mathfrak{t}$.  Hence, by the proof
of Lemma \ref{lem:stable_closure_of_treelike_forms}, (or rather its extension to the skew-symmetric case) 
the treelike form $(H_\mathfrak{t}, \lambda_\mathfrak{t})$ is isomorphic to the orthongal sum
\[ (H_\mathfrak{t}, \lambda_\mathfrak{t}) \cong \bigoplus_{i = 0}^j (H_{\mathfrak{t}_i}, \lambda_{\mathfrak{t}_i}), \]
where $\mathfrak{t}_0$ is the $0$-labelled tree consisting of $v_k$, $v_l$ and a single edge $e$ between them.
Now $(H_{\mathfrak{t}_0}, \lambda_{\mathfrak{t}_0}) \cong H_-(\Z)$, and
by induction, for $i > 0$ each $(H_{\mathfrak{t}_i}, \lambda_{\mathfrak{t}_i})$ is isomorphic to the orthogonal sum of a nonsingular form
and a zero form.  It follows that $(H_\mathfrak{t}, \lambda_\mathfrak{t})$ is isomorphic to the orthogonal sum
of a nonsingular form and a zero form.
\end{proof}

\begin{proof}[Proof of Theorem \ref{thm:delTP-4k+2}~\eqref{thm:delTP-4k+2:main}]
Suppose that $M = \del W$ where $W$ is a treelike plumbing.  
From the long exact sequence of the pair $(W, M)$ and the fact that $W \simeq \vee S^{2k+1}$, 
we obtain the short exact sequence
\begin{equation} \label{eq:M_and_W}
0 \to H^{2k}(M) \to H^{2k+1}(W, M) \to H^{2k+1}(W) \to H^{2k+1}(M) \to 0 .
\end{equation}
By Lemma \ref{lem:tree_recognition:4k+2},
the extended intersection form of $W$ is treelike and by Proposition \ref{prop:skew-symmetric_treelike_forms},
the intersection form of $W$ is the sum of a zero form and a nonsingular form.  Since the homomorphism
$H^{2k+1}(W, M) \to H^{2k+1}(W)$ is the adjoint of the intersection form of $W$, it follows
that $H^*(M)$ is torsion free.

Now suppose that $M = \del W$, $W \in \cal{H}^{4k+2}$ and that $H^*(M)$ is torsion free.
Again we use the exact sequence \eqref{eq:M_and_W} which still applies.
Since $H^{2k}(M) \cong H^{2k+1}(M)$ is torsion free, we conclude that the intersection
form of $W$, $(H^{2k+1}(W, \del W), \lambda_W)$ splits as an orthogonal sum $(H, \lambda) \oplus (F, 0)$
where $(H, \lambda)$ is nonsingular and so $(H, \lambda) \cong H_-(\Z^s)$ for some $s$.
If we now take the connected sum $W' : = W \sharp (S^{2k+1} \times S^{2k+1})$, we have $M = \del W'$
and that the intersection form of $W'$ splits as $(H, \lambda) \oplus (F, 0) \oplus H_-(\Z)$.
Since $(H, \lambda) \cong H_-(\Z^s)$ is treelike, $(F, 0)$ is a sum of zero forms $(\Z, 0)$ and $(\Z, 0)$ is treelike, 
by Lemma \ref{lem:stable_closure_of_skew-treelike_forms} we see that $(H^{2k+1}(W'), \lambda_{W'}, \mu_{W'})$ is treelike.
Hence by Lemma \ref{lem:tree_recognition:4k+2}, $W' \in \cal{TP}^{4k+2}$ and so $M \in \del \cal{TP}^{4k+2}$.
\end{proof}

\subsection{Tree-manifolds and $4$-dimensional treelike plumbings} \label{subsec:3-manifolds}
In dimensions $3$ and $4$ treelike plumbings and their boundaries have been intensively studied.
Plumbings of $D^2$-bundles over surfaces also arise in the study of complex manifolds and 
complex singularities in complex dimension two: see for example \cite{Hirzebruch}.
In the topological setting, von Randow \cite{vonRandow} used the term ``Baummannigfaltigkeit'', literally ``tree-manifold''
to describe $3$-manifolds $M$ obtained by plumbing $D^2$-bundles over $2$-spheres according to trees.
In the terminology of this paper, $\del \cal{TP}^4$ is precisely the set of oriented diffeomorphism
classes of tree-manifolds in von Randow's sense.  Tree-manifolds are in turn a special case of {\em graph manifolds}
where $D^2$-bundles over general surfaces are allowed, the same bundle can be plumbed to itself,
and the pluming graph need not be a tree.  Graph manifolds were the subject of early work of Waldhausen \cite{Waldhausen}
and a comprehensive calculus of graph manifolds was given in \cite{Neumann2}.

The motivation for studying plumbing manifolds in dimension four has so far largely arisen from questions in topology,
complex surfaces and complex singularities, and these are themes which are rather different from our motivation.
In particular, the methods we use for constructing metrics of positive Ricci curvature on the boundaries of plumbings,
i.e.~Theorem 1.1, do not apply to $3$-dimensional manifolds.
We shall mention just two highlights of previous work on tree-manifolds. 
The Poincar\'{e} conjecture was proven rather early for tree-manifolds: \cite{vonRandow, Scharf}.  Generalising this work,
Neumann and Weintraub \cite[Theorem p.\,71]{N-W} proved every $W \in \cal{TP}^4$ with boundary $\del W \cong S^3$ is diffeomorphic to one of the following connected sums: $\sharp_s(S^2 \times S^2)$ or $\sharp_s(\C P^2 \sharp (- \C P^2))$.

So far as we can tell, the question of which linking forms arise as the linking forms of tree-manifolds may not have been completely
answered in the literature up until now.  On the other hand, \cite[Theorem 6.1]{K-K} states that every torsion linking form $(G, b)$
is isomorphic to the linking form of {\em some} $3$-manifold. From our earlier results we deduce

\begin{Proposition} \label{prop:3-manifolds} 
Every isomorphism class of linking form $(G, b)$ is realised as the linking form
a tree-manifold $M \in \del \cal{TP}^4$.
\end{Proposition}

\begin{proof}
It follows from Theorem \ref{thm:realisation_of_eqlfs}
that every linking form $(G, b)$ arises as the boundary of an even symmetric treelike bilinear form $(H, \lambda)$.
Applying the discussion before Theorem \ref{thm:plumbing_realisation}, $(H, \lambda)$ is realised
as the intersection form of $-W$ for some $W \in \cal{TP}^4$, and so the linking form of $\del W$ is isomorphic to $(G, b)$.
\end{proof}



\section{Symmetric forms are stably treelike} \label{sec:symmetric_forms_are_stably_treelike}

We conclude this paper with a purely algebraic implication of our results from 
Section \ref{sec:realising_extended_quadratic_linking_forms}.
We consider even symmetric bilinear forms $(H, \lambda)$: for the general case,
see Remark \ref{rem:symmetric_case} below.
Recall that an even labelled tree $\mathfrak{t} = ((\V, \E), \ul{a})$ 
defines an even symmetric form $(H_\mathfrak{t}, \lambda_\mathfrak{t})$ and that an even symmetric
form $(H, \lambda)$ is treelike if it is isomorphic to some treelike form $(H_\mathfrak{t}, \lambda_\mathfrak{t})$.
In Section \ref{subsec:non-treelike_forms} we showed that there are infinite families of even symmetric
forms which are not treelike and remarked that determing if a given form $(H, \lambda)$ is treelike seems
to be a hard problem.  In this direction, Newman \cite{Newman} has show that every symmetric bilinear form
is ``nearly treelike'' in the sense that it can presented as the intersection matrix of a $\Z$-labelled connected graph
$(\V, \E)$ which fails to be a tree only because there is a pair of verticies which may have artibrarily many edges
between them.

One theme in the study of symmetric forms is to consider them stably; i.e.~to consider their properties under the addition of copies of the 
the standard hyperbolic form $H_+(\Z)$, defined in \eqref{eq:hyperbolic_form}.
In this section we show that even forms are stably treelike.
Let $H_+(\Z^s) = \oplus_s H_+(\Z)$ denote the $s$-fold orthogonal sum of the standard hyperbolic form.
We shall say that two even forms $(H_0, \lambda_0)$ and $(H_1, \lambda_1)$ are stably equivalent
if there are $s, t \geq 0$ and an isomorphism
\[ (H_0, \lambda_0) \oplus H_+(\Z^s) \cong (H_1, \lambda_1) \oplus H_+(\Z^t). \]

\begin{Theorem} \label{thm:treelike-forms} 
For every even symmetric bilinear form $(H, \lambda)$ there is an $s \geq 0$ such that
$(H, \lambda) \oplus H_+(\Z^s)$ is treelike.
\end{Theorem}



\begin{proof}
Firstly we write $(H, \lambda) = (H_0, \lambda_0) \oplus (F, 0)$ where $(F, 0)$ is the zero form
and $(H, \lambda_0)$ is nondegenerate.
By a theorem of Nikulin \cite[Corollary 1.13.4]{Nikulin}, two nondegenerate even forms $(H_0, \lambda_0)$
and $(H_1, \lambda_1)$ are stably equivalent if and only if they have the same signature and
isomorphic boundary quadratic linking forms; $\del(H_0, \lambda_0) \cong \del(H_1, \lambda_1)$.
Moreover, by a theorem of Milgram, \cite[Theorem, Appendix 4]{M-H}, the signature of $(H_i, \lambda_i)$
is determined mod $8$ by $\del(H_i, \lambda_i)$.  

We start with any nondegenerate even form $(H_0, \lambda_0)$.
By Theorem \ref{thm:indcomposable_qlfs} and Lemmas~\ref{lem:stable_closure_of_treelike_forms}, \ref{lem:cyclic}, \ref{lem:hyperbolic} and \ref{lem:psuedo-hyperbolic1}, every quadratic linking form is realised as the boundary of an even treelike form.  Hence there is a treelike form $(H_\mathfrak{t}, \lambda_\mathfrak{t})$ such that
$\del(H, \lambda) \cong \del(H_\mathfrak{t}, \lambda_\mathfrak{t})$.  By Milgram's theorem the signatures of
$(H_0, \lambda_0)$ and $(H_\mathfrak{t}, \lambda_\mathfrak{t})$ agree modulo $8$.  
Let $E_8$ denote and $\wh E_8$ denote the symmetric forms with the following matricies:
%
%
\[ E_8 = \left( \begin{array}{cccccccc} 
2 & 1 & 0 & 0 & 0 & 0 & 0 & 0\\
1 & 2 & 1 & 0 & 0 & 0 & 0 & 0\\
0 & 1 & 2 & 1 & 0 & 0 & 0 & 0\\
0 & 0 & 1 & 2 & 1 & 0 & 0 & 0\\
0 & 0 & 0 & 1 & 2 & 1 & 0 & 1\\
0 & 0 & 0 & 0 & 1 & 2 & 1 & 0\\
0 & 0 & 0 & 0 & 0 & 1 & 2 & 0\\
0 & 0 & 0 & 0 & 1 & 0 & 0 & 2\\
 \end{array}   \right) \]
and
\[
 \wh E_8 = \left( \begin{array}{cccccccc} 
-2 & 1  & 0  & 0  & 0  & 0  & 0  & 0 \\
1 & -2 & 1 & 0 & 0 & 0 & 0 & 0\\
0 & 1 & -2 & 1 & 0 & 0 & 0 & 0\\
0 & 0 & 1 & -2 & 1 & 0 & 0 & 0\\
0 & 0 & 0 & 1 & -2 & 1 & 0 & 1\\
0 & 0 & 0 & 0 & 1 & -2 & 1 & 0\\
0 & 0 & 0 & 0 & 0 & 1 & -2 & 0\\
0 & 0 & 0 & 0 & 1 & 0 & 0 & -2\\
 \end{array}   \right)    \]
It is well known that $E_8$ is nonsingular and has signature $8$.  Repeated application
of Lemma \ref{lem:det}~\eqref{lem:det:a} shows that $\wh E_8$ is nonsingular, and since 
$\wh E_8$ is negative definite, it follows that $\wh E_8$ has signature $-8$.
Notice that both $E_8$ and $\wh E_8$ are even treelike forms.
For an integer $t$,
let $\oplus_t \wt E_8$ denote the $|t|$-fold orthongal sum of $E_8$, if $t \geq 0$ 
and the $|t|$-fold orthongal sum of $\wh E_8$ if $t < 0$, so that $\oplus_t \wt E_8$ has signature
$8t$.
Applying Nikulin's result, \cite[Corollary 1.13.4]{Nikulin}, we can find non-negative integers 
$r, s, t$ with $s \geq 1$, and an isomorphism
\[ (H_0, \lambda_0) \oplus (F, 0) \oplus H_+(\Z^r) \cong 
(H_\mathfrak{t}, \lambda_\mathfrak{t}) \oplus (F, 0)\oplus H_+(\Z^s) \oplus \oplus_t \wt E_8. \]
Now by Lemma~\ref{lem:stable_closure_of_treelike_forms},
$(H_\mathfrak{t}, \lambda_\mathfrak{t}) \oplus (F, 0) \oplus H_+(\Z^s) \oplus \oplus_t \wt E_8$ is treelike,
which proves the theorem.
\end{proof}

\begin{Remark} \label{rem:removing_hyperbolics}
At present, we do not know precisely how many hyperbolics need to be added to 
a general even form until it becomes treelike.  Note that by
\cite[Corollary 1.13.4]{Nikulin}, if $(H, \lambda)$
is nondegenerate, with $H$ having rank $r$ and $\lambda$ having signature $\sigma$,
then $(H, \lambda) \oplus H_+(\Z)$ is the unique even symmetric form with rank $r + 2$,
signature $\sigma$ and boundary quadratic linking form isomorphic to $\del(H, \lambda)$.
\end{Remark}

\begin{Remark} \label{rem:symmetric_case}
We point out that Theorem \ref{thm:treelike-forms} has a symmetric analogue in which the
form $H_+(\Z^s)$ is replaced by the the $s$-fold sum of the form $\an{1} \oplus \an{-1}$
where $\an{\epsilon}$ denotes the symmetric form $(\Z, \epsilon)$, $\epsilon = \pm 1$.  The proof is similar, 
except that one now uses the boundary linking form and the fact that $\an{\epsilon}$ is treelike.
\end{Remark}


\vskip 1cm
\noindent
\emph{Diarmuid Crowley}\\
  {\small
  \begin{tabular}{@{\qquad}l}
    Institute of Mathematics\\
    University of Aberdeen\\
	Aberdeen AB24 3UE, United Kingdom\\
    \textsf{dcrowley@abdn.ac.uk}\\
  \end{tabular}


\bigskip

\noindent
\emph{David Wraith}\\
  {\small
  \begin{tabular}{@{\qquad}l}
   Department of Mathematics\\
   National University of Ireland Maynooth\\
   Maynooth, Co.~Kildare, Ireland \\
    \textsf{david.wraith@maths.nuim.ie}\\
  \end{tabular}

\end{document}